\documentclass[10pt]{article}

\usepackage{amsmath,amsthm} 
\usepackage{amssymb,mathrsfs} 
\usepackage{amssymb}
\usepackage{amsfonts}
\usepackage{upgreek}
\usepackage{nicefrac}
\usepackage{geometry}
\usepackage{authblk}
\usepackage{hyperref}
\hypersetup{
  colorlinks = true,
  citecolor = blue,
}
 \usepackage{color}

\usepackage{enumitem}
\usepackage{url}
\usepackage{lmodern} 
\usepackage{xcolor}
\usepackage{bbm}
\usepackage{xargs}
\usepackage{aliascnt}
\usepackage{cleveref}

\theoremstyle{plain}

\newtheorem{theorem}{Theorem}
\crefname{theorem}{theorem}{Theorems}
\Crefname{Theorem}{Theorem}{Theorems}

\newaliascnt{proposition}{theorem}
\newtheorem{proposition}[proposition]{Proposition}
\aliascntresetthe{proposition}
\crefname{proposition}{proposition}{propositions}
\Crefname{Proposition}{Proposition}{Propositions}

\newaliascnt{lemma}{theorem}
\newtheorem{lemma}[lemma]{Lemma}
\aliascntresetthe{lemma}
\crefname{lemma}{lemma}{lemmas}
\Crefname{Lemma}{Lemma}{Lemmas}

\theoremstyle{remark}
\newaliascnt{remark}{theorem}
\newtheorem{remark}[remark]{Remark}
\aliascntresetthe{remark}
\crefname{remark}{remark}{remarks}
\Crefname{Remark}{Remark}{Remarks}

\newtheorem{example}[theorem]{Example}
\crefname{example}{example}{examples}
\Crefname{Example}{Example}{Examples}

\crefformat{assumption}{{\textbf{A}}#2#1#3}

\newtheorem{assumptionH}{\textbf{H}\hspace{-3pt}}
\Crefname{assumptionH}{\textbf{H}\hspace{-3pt}}{\textbf{H}\hspace{-3pt}}
\crefname{assumptionH}{\textbf{H}}{\textbf{H}}

\newtheorem{assumptionB}{\textbf{B}\hspace{-3pt}}
\Crefname{assumptionB}{\textbf{B}\hspace{-3pt}}{\textbf{B}\hspace{-3pt}}
\crefname{assumptionB}{\textbf{B}}{\textbf{B}}

\newcommand{\keywords}[1]{\textbf{\textit{Key words:}} #1}
\newcommand{\mathsubjclass}[1]{\textbf{\textit{Mathematics Subject Classification:}} #1}

\newcommand{\Var}{\mathrm{Var}}
\def\msh{\mathsf{H}}
\def\rset{\mathbb{R}}

\def\rmd{\mathrm{d}}

\newcommandx{\abs}[2][1=]{\ifthenelse{\equal{#1}{}}{\left\vert #2 \right\vert}{\left\vert #2 \right\vert^{#1}}}
\newcommandx{\absLigne}[2][1=]{\ifthenelse{\equal{#1}{}}{\vert #2 \vert}{\vert #2\vert^{#1}}}

\newcommand{\parenthese}[1]{\left(#1 \right)}

\newcommand{\parentheseDeux}[1]{\left[ #1 \right]}
\newcommandx{\norm}[2][1=]{\ifthenelse{\equal{#1}{}}{\left\Vert #2 \right\Vert}{\left\Vert #2 \right\Vert^{#1}}}

\def\iid{i.i.d.}
\def\eg{e.g.}
\def\Id{\operatorname{Id}}
\newcommand{\PP}{\mathbb{P}}
\def\barmu{\bar{\mu}}
\def\Lip{\mathrm{Lip}}
\def\rme{\mathrm{e}}
\def\eqsp{\;}
\newcommand{\Hess}{\mathrm{Hess}}
\newcommand{\Law}{\mathrm{Law}}
\newcommand{\1}{\mathbbm{1}}
\def\rml{\mathrm{L}}
\def\nset{\mathbb{N}}
\newcommand{\plusinfty}{+\infty}
\def\barf{\bar{f}}
\def\torus{\mathbb{T}}
\renewcommand{\iint}[2]{\{ #1,\ldots,#2\}}

\title{Sticky nonlinear SDEs and convergence of McKean-Vlasov equations without confinement}

\author[1]{Alain Durmus}
\author[2]{Andreas Eberle}
\author[3]{Arnaud Guillin}
\author[2]{Katharina Schuh}
\affil[1]{Universit\'e Paris-Saclay, ENS Paris-Saclay, CNRS,\\ Centre Borelli, F-91190 Gif-sur-Yvette, France. \textit{E-mail:} \href{mailto:alain.durmus@cmla.ens-cachan.fr}{alain.durmus@cmla.ens-cachan.fr}}
\affil[2]{Universit\"at Bonn, Institut f\"ur Angewandte Mathematik, Endenicher Allee 60, 53115 Bonn, Germany.  \textit{E-mail:} \href{mailto:eberle@uni.bonn.de}{eberle@uni-bonn.de}, \linebreak \href{mailto:katharina.schuh@uni-bonn.de}{katharina.schuh@uni-bonn.de}}
\affil[3]{Laboratoire de Math\'ematiques Blaise Pascal, CNRS- UMR 6620, Unviversit\'e Clermont-Auvergne, Avenue de Landais, 63177 Aubiere cedex, France. \linebreak \textit{E-mail:} \href{mailto:guillin@math.univ-bpclermont.fr}{guillin@math.univ-bpclermont.fr} }

\date{}

\begin{document}
\maketitle
\pagenumbering{arabic}

\begin{abstract}
We develop a new approach to study the long time behaviour 
of solutions to nonlinear stochastic differential equations
in the sense of McKean, as well as propagation of chaos for the corresponding mean-field 
particle system approximations. Our approach is based
on a sticky coupling between two solutions to the equation.
We show that the distance process between the two copies is
dominated by a solution to a one-dimensional nonlinear stochastic differential equation with a sticky boundary at
zero. This new class of equations is then analyzed carefully.
In particular, we show that the dominating equation has a phase transition. In the regime where the Dirac measure at zero 
is the only invariant probability measure, we prove exponential
convergence to equilibrium both for the one-dimensional equation, and for the original nonlinear SDE. Similarly,
propagation of chaos is shown by a componentwise sticky
coupling and comparison with a system of one dimensional
nonlinear SDEs with sticky boundaries at zero. The approach 
applies to equations without confinement potential and 
to interaction terms that are not of gradient type. \\ \ \\
 \\
\mathsubjclass{60H10, 60J60, 82C31}
\end{abstract}

\section{Introduction}

The main objective of this paper is to study and quantify
convergence to equilibrium for McKean-Vlasov type nonlinear
stochastic differential equations of the form
 \begin{equation} \label{eq:nonlinearSDE}
\rmd \bar{X}_t= \parentheseDeux{\int_{\rset^d} b(\bar{X}_t-x) \rmd \bar{\mu}_t(x) }\rmd t+\rmd B_t \eqsp, \qquad \bar{\mu}_t=\Law(\bar{X}_t) \eqsp,
\end{equation}
where $(B_t)_{t\geq 0}$ is a $d$-dimensional standard Brownian motion and $b:\mathbb{R}^d\to\mathbb{R}^d$ is a Lipschitz continuous function. This nonlinear SDE is the probabilistic counterpart of the {\em Fokker-Planck equation} 
\begin{align} \label{eq:fokker_planck}
\frac{\partial}{\partial t} u_t= \nabla\cdot\Big[(1/2)\nabla u_t -(b\ast u_t)u_t\Big] \eqsp,
\end{align}
which describes the time evolution of the density 
$u_t$ of $ \bar{\mu}_t$ with respect to the Lebesgue measure on $\mathbb{R}^d$. 
Moreover, we also study uniform in time propagation of chaos for the approximating mean-field interacting particle systems 
\begin{align}
\rmd X_t^{i,N}& = \frac{1}{N}\sum_{j=1}^N b (X_t^{i,N}-X_t^{j,N})\rmd t+\rmd B_t^i \eqsp, && i\in \iint{1}{N} \eqsp, \label{eq:meanfield}
\end{align}
with i.i.d.\ initial values $X_0^{1,N},\ldots,X_0^{N,N}$, and
driven by independent $d$-dimensional Brownian motions
$\{(B_t^i)_{t\geq 0} \}_{i=1}^N$. Our results are based on
a new probabilistic approach relying on sticky couplings and
comparison with solutions to a class of nonlinear stochastic differential equations on the real interval $[0,\infty )$ with a sticky boundary at $0$.
The study of this type of equations carried out below might
also be of independent interest.

The equations \eqref{eq:nonlinearSDE} and \eqref{eq:fokker_planck} have been studied in many works.
Often a slightly different setup is considered, 
where the interaction $b$ is assumed to be of gradient type, i.e., $b=-\nabla W$ for an \emph{interaction potential} function $W:\mathbb R^d\to\mathbb R$, 
and an additional \emph{confinement potential} function $V:\mathbb R^d\to\mathbb R$ satisfying $\lim_{|x|\to\infty}V(x)=\infty$
is included in the equations. The corresponding Fokker-Planck equation
\begin{eqnarray} \label{eq:fokker_planckwithV}
\frac{\partial}{\partial t} u_t&=& \nabla\cdot\Big[(1/2)\nabla u_t +(\nabla V+\nabla W\ast u_t)u_t\Big] \eqsp,
\end{eqnarray}
occurs for example in the modelling of
granular media, see \cite{Vi03,BeCaCaPu98} and the references therein.  Existence and uniqueness of solutions to
\eqref{eq:nonlinearSDE}, \eqref{eq:fokker_planck} and \eqref{eq:fokker_planckwithV} have been studied intensively. Introductions to this topic can be found for example in \cite{Fu84, McKean66,Me96,Sz91}, while recent results have been established in \cite{MiVe16,HaSiSz18}. Under appropriate conditions, it can be shown that the solutions converge to a unique stationary distribution at some given rate, see \eg~\cite{CaMcVi03,CaMcVi06,BoGeGu13,EbGuZi16,DuEbGuZi20,GuLiWuZh19}. 
In the case without confinement considered here,
convergence to equilibrium of $(\bar{\mu}_t)_{t \geq 0}$ defined by \eqref{eq:nonlinearSDE} 
can only be expected 
for centered solutions, or after
recentering around the center of mass of $ \bar{\mu}_t$. It has first been analyzed in \cite{CaMcVi03,CaMcVi06} by an analytic approach 
and under the assumption that $b=-\nabla W$ for a convex function $W$. In particular, exponential convergence to equilibrium has
been established under the strong convexity assumption $\Hess(W)\geq \rho \Id$ for some $\rho>0$, and polynomial convergence in the case where $W$ is only degenerately strictly convex.
Similar results and some extensions have been derived in \cite{Ma03,CaGuMa08} using a probabilistic approach.

Our first contribution aims at complementing these results, and extending them to non-convex interaction potentials and interaction functions that are not of gradient type. More precisely, suppose that
\begin{equation}
  \label{eq:form_nabla_W_intro}
  b(x) = - Lx + \gamma(x) \eqsp, \qquad x \in \rset^d \eqsp, 
\end{equation}
where $L\in (0,\infty) $ is a positive real constant, and $\gamma: \rset^d \to \rset^d$ is a bounded function.
Then we give conditions on $\gamma$ ensuring
exponential convergence of centered solutions to \eqref{eq:nonlinearSDE} to a unique stationary distribution in the standard $\rml^1$ Wasserstein metric.
More generally, we show in Theorem \ref{thm:nonlinearSDE} that under these conditions there exist constants
$M,c\in (0,\infty )$ that depend only on $L$ and $\gamma$ such that
if $(\bar{\mu}_t)_{t \geq 0}$ and $(\bar{\nu}_t)_{t \geq 0}$ are the marginal distributions of two solutions of \eqref{eq:nonlinearSDE}, then for all $t\geq 0$,
\begin{align*}
\mathcal{W}_1(\bar{\mu}_t,\bar{\nu}_t)&\leq M \rme^{-ct}\mathcal{W}_1(\bar{\mu}_0,\bar{\nu}_0)  \eqsp .
\end{align*}
Using a coupling approach, related results have been derived in the previous works \cite{EbGuZi16,DuEbGuZi20} for the case where an additional confinement term is included in the equations. 
However, the arguments in these works rely on treating 
the equation with confinement and interaction term as a perturbation of the corresponding equation without interaction term, which has good ergodic properties.
In the unconfined case this approach does not work, since the equation without interaction is transient and
hence does not admit an invariant probability measure.
Moreover, we are not aware of results for this framework with non-convex interaction potentials and non-gradient interaction functions that rely on classical analytical methods.
Therefore, we have to develop a new approach
for analyzing the equation without confinement.

Our approach is based on sticky couplings, an
idea first developed in \cite{EbZi19} to control the total variation distance between the marginal
distributions of two non degenerate diffusion processes with identical noise but different drift coefficients.
Since two solutions of \eqref{eq:nonlinearSDE} differ
only in their drifts, we can indeed couple them 
using a sticky coupling in the sense of \cite{EbZi19}.
It can then be shown that the coupling distance process
can be controlled by the solution
$(r_t)_{t\geq 0}$ of a nonlinear SDE on
$[0,\infty )$ with a sticky boundary at $0$ of the form
\begin{align} \label{eq:one-dim_stickydiff}
\rmd r_t=[\tilde{b}(r_t)+a \PP(r_t >0)]\rmd t+2 \1_{(0,\infty)}(r_t)  \rmd W_t \eqsp,
\end{align}
Here $\tilde{b}$ is a real-valued function on $[0,\infty )$ satisfying $\tilde b (0)=0 $, $a$ is a positive constant, and $(W_t)_{t\geq 0}$ is a one-dimensional standard Brownian motion. Solutions to SDEs with diffusion coefficient $r \mapsto \1_{(0,\infty)}(r)$, as in \eqref{eq:one-dim_stickydiff}, 
have a sticky boundary at $0$, i.e., if the drift at $0$ is strictly positive, then the set of all time points $t\in [0,\infty )$ such that $r_t=0$ is
a fractal set with strictly positive Lebesgue measure
that does not contain any open interval. 
{Sticky }SDEs have attracted wide interest, starting from \cite{Fe54a,Fe54b} in the one-dimensional case. Multivariate extensions have been considered in \cite{Ik60,Wa71a,Wa71b} building upon results obtained in \cite{McKean63,Sk61,Sk62}, while corresponding martingale problems have been investigated in \cite{StVa71}. 
Versions of sticky processes occur among others in natural sciences \cite{CaFa12,GaGeMa85} and finance \cite{KaKiRi07}.
Note that in general no strong solution for this class of SDEs exists as illustrated in \cite{Ch97}. 
We refer to  \cite{EnPe2014,Ba14} and the references therein for recent contributions on this topic.
Note, however, that in contrast to standard sticky SDEs, the equation \eqref{eq:one-dim_stickydiff} is nonlinear in the sense of McKean. We are not aware of previous studies
of such nonlinear sticky equations, which seems to be a very interesting topic on its own. 

Intuitively, one would hope that as time evolves, more mass gets stuck at $0$, 
i.e., $\PP(r_t >0)$ decreases. As a consequence, the drift 
at $0$ in Equation \eqref{eq:one-dim_stickydiff} decreases, which again forces even more mass to get stuck at $0$.
Therefore, under appropriate conditions one could hope that $\PP(r_t =0)$ converges
to $1$ as $t\to\infty$. On the other hand, if $a$ is too large then the drift at $0$ might be too strong so that not all of the mass gets stuck at $0$ eventually.
This indicates that there might be a \emph{phase transition} 
for the nonlinear sticky SDE depending on the size of the constant $a$
compared to $\tilde b$.
In Section \ref{sec:non-linear_sticky_diff}, we prove rigorously that this intuition is correct.
Under appropriate conditions on $\tilde{b}$, we show at first that existence and uniqueness in law holds for solutions 
of \eqref{eq:one-dim_stickydiff}.
Then we prove
that for $a$ sufficiently small,
the Dirac measure at $0$ is the unique invariant probability measure, and geometric ergodicity holds. 
As a consequence, under corresponding assumptions, the
sticky coupling approach yields
exponential convergence to equilibrium for the 
original nonlinear SDE \eqref{eq:nonlinearSDE}.
On the other hand, we prove the existence of multiple invariant probability measures for \eqref{eq:one-dim_stickydiff} if the smallness condition on $a$ is not satisfied. In this case, we cannot make a statement on the behaviour of the distance function corresponding to the sticky coupling approach since based on this approach we only get upper bounds and  the existence of multiple invariant measure for the dominating sticky nonlinear SDE does not imply that the underlying distance function does not converge. If the unconfined SDE \eqref{eq:nonlinearSDE} has multiple invariant measures and if the two copies of the unconfined SDE in the sticky coupling start in two different equilibria, then the law of the distance function does not converge to the Dirac measure at zero.
Our results for \eqref{eq:nonlinearSDE} can also be adapted to deal with nonlinear SDEs over the torus $\mathbb{T} = \mathbb{R}/(2\pi\mathbb{Z})$, as considered in \cite{De20}.
As an example, we discuss the application to the Kuramoto model for which a more explicit analysis is available
\cite{AcBoPeFaSp05,BeGiPa10,BeGiPo14,CaGvPaSc20}.\medskip

Finally, in addition to studying the long-time behaviour of the nonlinear SDE \eqref{eq:nonlinearSDE},  we are also interested in establishing  {\em propagation of chaos} for the mean-field particle system approximation \eqref{eq:meanfield}. 
The propagation of chaos phenomenon first introduced by Kac \cite{kac56} describes the convergence of the empirical measure of the mean-field particle system \eqref{eq:meanfield} to the solution \eqref{eq:nonlinearSDE}. More precisely, in  \cite{Sz91,Me96} it has been shown under weak assumptions on $W$ that for i.i.d.\ initial laws, the random variables $X_t^{i,N} $, $i\in\{ 1,\ldots ,N\}$, become asymptotically independent as $N\to\infty$,
and the common law
$\mu_t^{N}$ of each of these random variables converges to $\bar{\mu}_t$. However, the original results are only valid uniformly over a finite time horizon.
 Quantifying the convergence uniformly for all times 
 $t\in\mathbb R_+$ is an important issue.  The case with a confinement potential has been studied for example in \cite{DuEbGuZi20}, see also the references therein. Again, the case when there is only interaction is more difficult. Malrieu \cite{Ma03} seems the first to consider the case without confinement. By applying a
  synchronous coupling, he proved uniform in time 
  propagation of chaos for strongly convex interaction potentials. Later on, assuming that the interaction potential  is loosing strict convexity only in a finite number of points (e.g., $W(x)= |x|^3$), Cattiaux, Guillin and Malrieu \cite{CaGuMa08} have shown uniform in time propagation of chaos with a rate getting worse with the degeneracy in convexity. 
  In a very recent work, Delarue and Tse \cite{DeTs21} prove uniform in time weak propagation of chaos (i.e., observable by observable) on the torus
  via Lions derivative methods. Remarkably, their results are  not limited to the unique invariant measure case.
 
Our contribution is in the same vein using probabilistic tools in place of analytic ones. 
We endow the space $\mathbb R^{Nd}$ consisting of $N$
particle configurations $x=(x^i)_{i=1}^N$ with the
semi-metric $l^1\circ\pi$, where
\begin{equation}
  \label{eq:def_l1}
          l^1(x,y)\ =\ \frac{1}{N}\sum\nolimits_{i=1}^N \abs{x^i-y^i } 
\end{equation}
is a normalized $l^1$-distance between configurations
$x,y\in\mathbb R^{Nd}$, and
\begin{equation}
\label{eq:proj_pi}
\pi(x,y)\ =\  \left(\left(x^i-\frac{1}{N}\sum\nolimits_{j=1}^Nx^{j}\right)_{i=1}^N,\left(y^{i}-\frac{1}{N}\sum\nolimits_{j=1}^Ny^{j}\right)_{i=1}^N\right)\eqsp, 
\end{equation}
is a projection from $\mathbb{R}^{Nd}\times \mathbb R^{Nd}$ to the subspace $\msh_N\times \msh_N$, where
\begin{equation}
  \label{eq:def_H_N}
 \msh_N\ =\ \{x\in\mathbb{R}^{Nd} :\sum\nolimits_{i=1}^N x^{i} =0\}  \eqsp.
\end{equation}
Let $\mathcal{W}_{l^1\circ \pi}$ denote the $L^1$ Wasserstein semimetric on probability measures on
$\mathbb{R}^{Nd}$ corresponding to the cost function
$l^1\circ \pi$. Then under assumptions stated below,
we prove uniform in time propagation of chaos for the mean-field particle system in the following sense:
Suppose that $(X_t^{1,N},\ldots ,X_t^{N,N})_{t\ge 0}$
is a solution of \eqref{eq:meanfield} such that
$X_0^{1,N},\ldots,X_0^{N,N}$ are i.i.d.\ with  distribution $\bar{\mu}_0$ having finite second moment. Let $\nu_t^N$ denote
the joint law of the random variables $X_t^{i,N}$, $i\in\{ 1,\ldots N\}$, and let $\bar\mu_t$ denote the law of the solution of \eqref{eq:nonlinearSDE} with initial law $\bar\mu_0$. Then
there exists a constant $C \in [0,\infty )$ such that for any $N \in \nset$,
\begin{equation}
  \label{eq:intro_prop_chaos}
\sup_{t \geq 0} \, \mathcal{W}_{l^1\circ \pi}(\bar{\mu}_t^{\otimes N},\nu_t^N)\leq CN^{-1/2} \eqsp .
\end{equation} 
The proof is based on a componentwise sticky coupling,
and a comparison of the coupling difference process
with a system of one-dimensional sticky nonlinear SDEs.\bigskip

The paper is organised as follows. In \Cref{sec:main_contraction_results}, we state our main results  regarding the long-time behaviour of \eqref{eq:nonlinearSDE}.  The main results on one-dimensional nonlinear SDEs with a sticky boundary at zero are stated in Section \ref{sec:non-linear_sticky_diff}. Sections \ref{sec:propachaos} and \ref{sec:system_non-linear_sticky_diff} contain the corresponding results on uniform (in time) propagation of chaos and mean-field systems of sticky SDEs. All the proofs are given in \Cref{sec:proofs}.
In \Cref{appendix}, we carry the results over to nonlinear sticky SDEs over $\mathbb{T}$ and consider the application to the Kuramoto model.

\paragraph*{Notation}
The Euclidean norm on $\mathbb{R}^{d}$ is denoted by $|\cdot|$. For $x\in\mathbb{R}$, we write $x_+=\max(0,x)$. 
For some space $\mathbb{X}$, which here is either $\mathbb{R}^{d}$, $\mathbb{R}^{Nd}$ or $\mathbb{R}_+$, we denote its Borel $\sigma$-algebra by $\mathcal{B}(\mathbb{X})$. The space of all probability measures on $(\mathbb{X},\mathcal{B}(\mathbb{X}))$ is denoted by $\mathcal{P}(\mathbb{X})$. 
Let $\mu,\nu\in\mathcal{P}(\mathbb{X})$. A coupling $\xi$ of $\mu$ and $\nu$ is a probability measure on $(\mathbb{X}\times \mathbb{X},\mathcal{B}(\mathbb{X})\otimes\mathcal{B}(\mathbb{X}))$ with marginals $\mu$ and $\nu$. 
$\Gamma(\mu,\nu)$ denotes the set of all couplings of $\mu$ and $\nu$. The  $\rml^1$  Wasserstein distance with respect to a distance function $d:\mathbb{X}\times\mathbb{X}\to\mathbb{R}_+$ is defined by
\begin{align*}
\mathcal{W}_d(\mu,\nu)=\inf_{\xi \in \Gamma(\mu,\nu)}\int_{\mathbb{X}\times\mathbb{X}}d(x,y)\xi(\rmd x \rmd y)\eqsp.
\end{align*}
We write $\mathcal{W}_1$ if the underlying distance function is the Euclidean distance.

We denote by $\mathcal{C}(\mathbb{R}_+,\mathbb{X})$ the set of continuous functions from $\mathbb{R}_+$ to $\mathbb{X}$, and by $\mathcal{C}^2(\mathbb{R}_+,\mathbb{X})$ the set of twice continuously differentiable functions.

Consider a probability space $(\Omega, \mathcal{A},P)$ and a measurable function $r:\Omega\to \mathcal{C}(\mathbb{R}_+,\mathbb{X})$. Then $\PP=P\circ r^{-1}$ denotes the law on $\mathcal{C}(\mathbb{R}_+,\mathbb{X})$, and $P_t=P\circ {r_t}^{-1}$ the marginal law on $\mathbb{X}$ at time $t$.

\section{Long-time behaviour of McKean-Vlasov diffusions} \label{sec:main_contraction_results}

We establish our results regarding \eqref{eq:nonlinearSDE} and \eqref{eq:meanfield} under the following assumption on $b$. 
\begin{assumptionB} 
  \label{ass:decomp_W}
  The function $b:\mathbb{R}^d\to\mathbb{R}^d$ is Lipschitz continuous and anti-symmetric, i.e., $b(z)=-b(-z)$, and there exist $L\in(0,\infty)$, a function $\gamma:\mathbb{R}^d\to\mathbb{R}^d$ and a Lipschitz continuous function $\kappa:[0,\infty)\to\mathbb{R}$ such that
  \begin{align} \label{eq:nabla_W}
        b(z)=-Lz+ \gamma(z) \qquad \text{for all } z\in\mathbb{R}^d \eqsp,
    \end{align}
and the following conditions are satisfied for all $x,y\in\mathbb{R}^d$:
\begin{align} \label{eq:kappa}
\langle x-y,\gamma(x)-\gamma(y)\rangle\leq \kappa(|x-y|)|x-y|^2\eqsp,
\end{align}
and 
\begin{align} \label{eq:nonexpl_cond}
    \limsup_{r\to\infty}(\kappa(r)-L)<0 \eqsp.
\end{align}
\end{assumptionB}

Let $\bar{b}(r)=(\kappa(r)-L)r$. If \eqref{eq:nonexpl_cond} holds, then there exist $R_0,R_1 \geq 0$ such that for 
\begin{align} 
 \bar{b}(r)&<0 \eqsp, \qquad \qquad &\text{ for any  $r>
 R_0 $} \eqsp, \label{eq:definition_R0} \\
 \bar{b}(r)/r&\leq -4/[R_1(R_1-R_{0})] \eqsp, \qquad \qquad &\text{ for any  $ r\geq R_1  $ } \eqsp. \label{eq:definition_R1}                                          
\end{align}
In addition, we assume
\begin{assumptionB} \label{ass:gamma_bound}
\begin{align*}
\|\gamma\|_\infty \leq \Big(4\int_0^{R_1}\exp\Big(\frac{1}{2}\int_0^s \bar{b}(r)_+ \rmd r\Big)\rmd s\Big)^{-1} \eqsp.
\end{align*}
\end{assumptionB}

Often drifts of gradient type are considered, i.e., $b\equiv \nabla U$ for some potential $U\in\mathcal{C}^2$.
Then, \Cref{ass:decomp_W} is satisfied for instance for 
$L$-strongly convex potentials and condition \eqref{eq:kappa} holds for $\kappa\equiv 0$. In this case, \Cref{ass:gamma_bound} reduces to $\|\gamma\|_{\infty}\le \sqrt{L}/8$. But, the assumptions include also asymptotically $L$-strongly convex potentials as double-well potentials and more general drifts provided the deviation represented by the function $\gamma$ to the linear term $-Lz$ is sufficiently small in terms of the generalized one-sided Lipschitz bound and the bound in the supremum norm. In particular, this can always be obtained by considering a sufficiently small multiple of $\gamma$.

Additionally, we consider the following condition on the initial distribution. 
\begin{assumptionB} \label{ass:init_distr}
The initial distribution $\mu_0$ satisfies $\int_{\rset^d} \norm{x}^4 \mu_0(\rmd x) < \plusinfty$ and $\int_{\rset^d} x \, \mu_0(\rmd x) = 0$. 
\end{assumptionB}

Note that under conditions \Cref{ass:decomp_W} and \Cref{ass:init_distr}, unique strong solutions $(\bar{X}_t)_{t\geq 0}$ and $(\{X_t^{i,N}\}_{i=1}^N)_{t\geq 0}$
exist for \eqref{eq:nonlinearSDE} and \eqref{eq:meanfield},  see \eg~\cite[Theorem 2.6]{CaGuMa08}.  
In addition, note that since $b$ is assumed to be anti-symmetric, by an easy localisation argument, we get that $
\rmd \mathbb{E}[\bar{X_t}]/ \rmd t =\mathbb{E}[b\ast\mu_t(\bar{X}_t)]=0$
 and $\rmd\mathbb{E}[ N^{-1}\sum_{i=1}^N X_t^{i,N}]/\rmd t=0$. 
Thus, if $\bar{X}_0$ and $\{X_0^{i,N}\}_{i=1}^N$ have distribution $\mu_0$ and $\mu_0^{\otimes N}$, respectively, with $\mu_0$ satisfying \Cref{ass:init_distr}, then 
it holds $\mathbb{E}[\bar{X}_t]=0$ and $\mathbb{E}[ N^{-1}\sum_{i=1}^N X_t^{i,N}]=0$ for all $t\geq 0 $.

Suppose $f:\mathbb{R}_+\to\mathbb{R}_+$ is an increasing, concave function vanishing at zero. Then $d(x,y)=f(|x-y|)$ defines a distance. The corresponding $L^1$ Wasserstein distance is denoted by $\mathcal{W}_f$.
Note that in the case $f(t) = t$ for any $t \geq 0$, $\mathcal{W}_f$ is simply $\mathcal{W}_1$.

\begin{theorem}[Contraction for nonlinear SDE]\label{thm:nonlinearSDE}
Assume \Cref{ass:decomp_W} and \Cref{ass:gamma_bound}. Let $\barmu_0, \bar{\nu}_0$ be probability measures on $(\mathbb{R}^d,\mathcal{B}(\mathbb{R}^d))$ satisfying \Cref{ass:init_distr}. For any $t\geq 0$, let $\barmu_t$ and $\bar{\nu}_t$ denote the laws of $\bar{X}_t$ and $\bar{Y}_t$ where $(\bar{X}_s)_{s\geq 0}$ and $(\bar{Y}_s)_{s\geq 0}$ are solutions of \eqref{eq:nonlinearSDE} with initial distribution $\barmu_0$ and $\bar{\nu}_0$, respectively.
 Then, for all $t\geq 0$,
\begin{equation} \label{eq:thm_nonlinearSDE}
\mathcal{W}_f(\barmu_t, \bar{\nu}_t) \leq \rme^{-\tilde{c} t} \mathcal{W}_f(\barmu_0,\bar{\nu}_0) \qquad  \text{ and }\qquad 
 \mathcal{W}_1(\barmu_t, \bar{\nu}_t)  \leq M_1 \rme^{-\tilde{c} t} \mathcal{W}_1(\barmu_0,\bar{\nu}_0) \eqsp,
\end{equation}
where the function $f$ is defined by \eqref{eq:definition_f}
and the constants $\tilde{c}$ and $M_1$ are given by
\begin{align} 
\tilde{c}^{-1}&=2\int_0^{R_{1}}\int_0^s \exp\Big(\frac{1}{2}\int_r^s\bar{b}(u)_+ \ \rmd u\Big)\rmd r\rmd s \eqsp, \label{eq:definition_tildec}
\\ M_1 & = 2\exp\Big(\frac{1}{2}\int_0^{R_0} \bar{b}(s)_+\rmd s\Big) \eqsp. \label{eq:def_M_1}
\end{align}
\end{theorem}

\begin{proof}
  The proof is postponed to \Cref{sec:proof_thm_nonlinear}.
\end{proof}

The construction and definition of the underlying distance function $f(|x-y|)$ mentioned in \Cref{thm:nonlinearSDE} is based on the one introduced by \cite{Eb16}. 

To prove \Cref{thm:nonlinearSDE} we use a coupling $(\bar{X}_t,\bar{Y}_t)_{t\geq 0}$ of two copies of solutions to the nonlinear stochastic differential equation \eqref{eq:nonlinearSDE} with different initial conditions.
The coupling $(\bar{X}_t,\bar{Y}_t)_{t\geq 0}$ will be defined as the weak limit of a family of couplings $(\bar{X}_t^\delta,\bar{Y}_t^\delta)_{t\geq 0}$, parametrized by $\delta>0$. Roughly, this family is mixture of synchronous and reflection couplings and can be described as follows. For $\delta>0$, $(\bar{X}_t^\delta,\bar{Y}_t^\delta)_{t\geq 0}$ behaves like a reflection coupling if $|\bar{X}_t^\delta-\bar{Y}_t^\delta|\geq \delta$, and like a  synchronous coupling if $|\bar{X}_t^\delta-\bar{Y}_t^\delta|=0$. For $|\bar{X}_t^\delta-\bar{Y}_t^\delta|\in(0,\delta)$ we take an interpolation of synchronous and reflection coupling. We argue that the family of couplings $\{(\bar{X}_t^\delta,\bar{Y}_t^\delta)_{t \geq 0}:\delta>0\}$ is tight and that a subsequence $\{(\bar{X}_t^{\delta_n},\bar{Y}_t^{\delta_n})_{t\geq 0}:n\in\mathbb{N}\}$ converges to a limit $(\bar{X}_t,\bar{Y}_t)_{t\geq 0}$. This limit is a coupling which we call the sticky coupling associated to \eqref{eq:nonlinearSDE}.

To carry out the construction rigorously, we take two Lipschitz continuous functions $\mathrm{rc}^\delta, \mathrm{sc}^\delta:\mathbb{R}_+\to [0,1]$ for $\delta>0$ such that
\begin{equation} \label{eq:condition_rc_sc}
\begin{aligned}
& \mathrm{rc}^\delta(0)=0\eqsp, \ \mathrm{rc}^\delta(r)=1 \text{ for } r\geq \delta\eqsp,  \mathrm{rc}^\delta(r)>0 \text{ for }r>0 \text{ and } 
 \mathrm{rc}^\delta(r)^2+\mathrm{sc}^\delta(r)^2=1 \text{  for } r\ge 0 \eqsp.
\end{aligned}
\end{equation} 
Further, we assume that there exists $\epsilon_0>0$ such that for any $\delta\leq \epsilon_0$, $\mathrm{rc}^\delta$ satisfies
\begin{equation} \label{eq:condition_rc_sc2}
\begin{aligned}
\mathrm{rc}^\delta(r)\geq \frac{\|\gamma\|_{\Lip}}{2\|\gamma\|_\infty}r &&\text{ for any }r\in(0,\delta)\eqsp,
\end{aligned}
\end{equation}
where $\|\gamma \|_{\Lip}<\infty$ denotes the Lipschitz norm of $\gamma$. 
This assumption is satisfied for example if $\mathrm{rc}^\delta(r)=\sin((\pi/2\delta)r)\1_{r < \delta}+\1_{r\geq \delta}$ and $\mathrm{sc}^\delta(r)=\cos((\pi/2\delta)r)\1_{r < \delta}$ with $\delta\leq \epsilon_0=2\|\gamma\|_\infty/\|\gamma\|_{\Lip}$. 

Let $(B_t^1)_{t\geq 0}$ and $(B_t^2)_{t\geq 0}$ be two $d$-dimensional Brownian motions. We define the coupling $(\bar{X}_t^\delta,\bar{Y}_t^\delta)_{t\geq 0}$ as a process in $\mathbb{R}^{2d}$ satisfying the following nonlinear stochastic differential equation
\begin{equation}\label{eq:coupling_non_linearSDE_approx}
\begin{aligned} 
\rmd \bar{X}_t^\delta&=b*\barmu^\delta_t(\bar{X}_t^\delta)\rmd t+ \mathrm{rc}^\delta(\bar{r}_t^\delta)\rmd B_t^1+\mathrm{sc}^\delta(\bar{r}_t^\delta)\rmd B_t^2\eqsp, && \barmu_t^\delta =\Law(\bar{X}_t^\delta)\eqsp ,
\\ \rmd \bar{Y}_t^\delta&=b*\bar{\nu}^\delta_t(\bar{Y}_t^\delta)\rmd t+ \mathrm{rc}^\delta(\bar{r}_t^\delta)(\Id-2\bar{e}_t^\delta  (\bar{e}_t^\delta)^T ) \rmd B_t^1+\mathrm{sc}^\delta(\bar{r}_t^\delta)\rmd B_t^2\eqsp, && \bar{\nu}_t^\delta =\Law(\bar{Y}_t^\delta)  
\end{aligned}
\end{equation}
with initial condition $(\bar{X}_0^\delta,\bar{Y}_0^\delta)=(x_0,y_0)$. Here we set $\bar{Z}_t^\delta=\bar{X}_t^\delta-\bar{Y}_t^\delta$, $\bar{r}_t^\delta=|\bar{Z}_t^\delta|$ and $\bar{e}_t^\delta=\bar{Z}_t^\delta/\bar{r}_t^\delta$ if $\bar{r}_t^\delta\neq 0$. For $\bar{r}_t^\delta=0$, $\bar{e}_t^\delta$ is some arbitrary unit vector, whose exact choice is irrelevant since $\mathrm{rc}^\delta(0)=0$. 
We note that a refection coupling is obtained if $\mathrm{rc}^{\delta}=1$, whereas a synchronous coupling is obtained if $\mathrm{sc}^{\delta}=0$. This indicates the name of the functions $\mathrm{rc}$ and $\mathrm{sc}$, respectively.

\begin{theorem}
  \label{theo:1}
  Assume \Cref{ass:decomp_W}. Let $\bar{\mu}_0$ and $\bar{\nu}_0$ be probability measures on $(\mathbb{R}^d,\mathcal{B}(\mathbb{R}^d))$ satisfying \Cref{ass:init_distr}.
  Then, $(\bar{X}_t,\bar{Y}_t)_{t \geq 0}$ is a subsequential limit in distribution as $\delta\to 0$ of $\{(\bar{X}_t^{\delta},\bar{Y}_t^{\delta})_{t \geq 0} \, : \, \delta >0\}$ where $(\bar{X}_t)_{t \geq 0}$ and $(\bar{Y}_t)_{t \geq 0}$ are solutions of \eqref{eq:nonlinearSDE} with  initial distribution $\bar{\mu}_0$ and $\bar{\nu}_0$. 
  Further, there exists a process $(r_t)_{t \geq 0}$ defined on the same probability space as $(\bar{X}_t,\bar{Y}_t)_{t\geq 0}$ satisfying 
 for any $t \geq 0$,  $  \absLigne{\bar{X}_t - \bar{Y}_t} \leq r_t$ almost surely and 
  which is a weak solution of
  \begin{equation}
    \label{eq:non_linear_sticky_sde}
\rmd r_t=(\bar{b}(r_t)+2 \|\gamma\|_\infty \PP(r_t >0))\rmd t+2 \1_{(0,\infty)}(r_t) \rmd \tilde{W}_t \eqsp,
\end{equation}
where $(\tilde{W}_t)_{t \geq 0}$ is a one-dimensional Brownian motion. 
\end{theorem}
\begin{proof}
The proof is postponed to \Cref{sec:proof_theo1}.
\end{proof}

Therefore, next we study sticky nonlinear SDEs given by \eqref{eq:one-dim_stickydiff}. 

\section{Nonlinear SDEs with sticky boundaries} \label{sec:non-linear_sticky_diff}
Consider nonlinear SDEs with a sticky boundary at $0$ of the form 
\begin{align} \label{eq:one-dimSDE_general}
\rmd r_t=(\tilde{b}(r_t)+P_t(g))\rmd t+2\1_{(0,\infty)}(r_t)\rmd W_t\eqsp, && P_t=\Law(r_t) \eqsp ,
\end{align}
where $\tilde{b}:[0,\infty)\to\mathbb{R}$ is some continuous function and $P_t(g)=\int_{\mathbb{R}_+}g(r)P_t(\rmd r)$ for some measurable function $g:[0,\infty)\to \mathbb{R}$.

In this section we establish existence, uniqueness in law and comparison results for solutions of \eqref{eq:one-dim_stickydiff}.
Consider a filtered probability space $(\Omega, \mathcal{A}, (\mathcal{F}_t)_{t\ge 0},P)$ and a probability measure $\mu$ on $\mathbb{R}_+$. 
We call an $(\mathcal{F}_t)_{t\ge 0}$ adapted process $(r_t,W_t)_{t\geq 0}$ a \textit{weak solution} of \eqref{eq:one-dimSDE_general} 
with initial distribution $\mu$ if the following holds: $\mu=P\circ r_0^{-1}$, the process $(W_t)_{t\geq 0}$ is a one-dimensional $(\mathcal{F}_t)_{t\ge 0}$ Brownian motion w.r.t. $P$, the process $(r_t)_{t\geq 0}$ is non-negative and continuous, and satisfies almost-surely
\begin{align*}
r_t-r_0=\int_0^t \Big(\tilde{b}(r_s)+P_s(g)\Big)\rmd s+\int_0^t 2\cdot \1_{(0,\infty)}(r_s)\rmd W_s\eqsp, \qquad \text{ for $t\in\mathbb{R}_+$}\eqsp.
\end{align*}

Note that the sticky nonlinear SDE given in \eqref{eq:one-dim_stickydiff} is a special case of \eqref{eq:one-dimSDE_general} with $g(r)=a\1_{(0,\infty)}(r)$
since $\PP(r_t >0)=\int_{\mathbb{R}_+}\1_{(0,\infty)}(y) P_t(\rmd y)$ with $P_t=P\circ r_t^{-1}$.

\subsection{Existence, uniqueness in law, and a comparison result} \label{sec:existence}

Let $\mathbb{W}=\mathcal{C}(\mathbb{R}_+,\mathbb{R})$ be the space of continuous functions endowed with the topology of uniform convergence on compact sets, and let $\mathcal{B}(\mathbb{W})$ be the corresponding Borel $\sigma$-algebra.
Suppose $(r_t, W_t)_{t\geq 0}$ is a solution of \eqref{eq:one-dimSDE_general} on $(\Omega,\mathcal{A},P)$, then we denote by $\PP=P\circ r^{-1}$ its law 
on $(\mathbb{W},\mathcal{B}(\mathbb{W}))$. 
We say that \textit{uniqueness in law} holds for \eqref{eq:one-dimSDE_general} if for any two solutions $(r_t^1)_{t\geq 0}$ and $(r_t^2)_{t\geq 0}$ of \eqref{eq:one-dimSDE_general} with the same initial law, the distributions of $(r_t^1)_{t\geq 0}$ and $(r_t^2)_{t\geq 0}$ on $(\mathbb{W},\mathcal{B}(\mathbb{W}))$ are equal. 

We impose the following assumptions on  $\tilde{b}$, $g$ and the initial condition $\mu$:
\begin{assumptionH} 
\label{H1b}  $\tilde{b}$ is a Lipschitz continuous function with Lipschitz constant $\tilde{L}$ and $\tilde{b}(0)=0$.
\end{assumptionH}
\begin{assumptionH}
\label{H1g}   $g$ is a left-continuous, non-negative, non-decreasing and bounded function.
\end{assumptionH}
\begin{assumptionH}
\label{H2}
There exists $p>2$ such that the $p$-th order moment of the law $\mu$ is finite.
\end{assumptionH}
Note that for \eqref{eq:one-dim_stickydiff}, the condition \Cref{H1g} is satisfied if $a$ is a positive constant.
It follows from \Cref{H1b} and \Cref{H1g} that there is a constant $C<\infty$ such that for all $r\in\mathbb{R}_+$, the following linear growth condition  holds,
\begin{align} \label{eq:nonlinear_lineargrowthcondition}
 \tilde{b}(r)+\sup_{ p\in\mathcal{P}(\mathbb{R}_+)}p(g)\leq C(1+|r|)\eqsp.
\end{align}

In order to get a solution to \eqref{eq:one-dimSDE_general} on $\mathbb{R}_+$ we extend the function $\tilde{b}$ to $\mathbb{R}$ by setting $\tilde{b}(r)=0$ for $r<0$.
Note that any solution $(r_t)_{t\geq 0}$ with initial distribution supported on $\mathbb{R}_+$ satisfies almost surely $r_t\geq 0$ for all $t\geq 0$. 
This follows from the It\={o}-Tanaka formula applied to $F(r)=\1_{(-\infty,0)}(r) r$, cf. \cite[Chapter 6, Theorem 1.2 and Theorem 1.7]{ReYo99}. Indeed
\begin{align*}
\1_{(-\infty,0)}(r_t)r_t&=\1_{(-\infty,0)}(r_0)r_0+\int_0^t \1_{(-\infty,0)}(r_s)\rmd r_s - \frac{1}{2}\ell_t^{0-}(r)
\\ &=\int_0^t \1_{(-\infty,0)}(r_s)(\tilde{b}(r_s)+P_s(g))\rmd s+\int_0^t \1_{(-\infty,0)}2 \1_{(0,\infty)}(r_s)\rmd W_s- \frac{1}{2}\ell_t^{0-}(r)
\\ & =\int_0^t\1_{(-\infty,0)}(r_s)P_s(g)\rmd s \ge 0\eqsp, 
\end{align*}
where $\ell_t^{0-}(r)$ is the left local time at $0$, which is given by $\ell_t^{0-}(r)=\lim_{\epsilon \downarrow 0} \epsilon^{-1}\int_0^t\1_{\{-\epsilon\le r_s\le 0\}}\rmd [r]_s$ and which vanishes, since $\rmd [r]_s=\1_{(0,\infty)}(r_s)\rmd s$. 

Existence and uniqueness in law of \eqref{eq:one-dimSDE_general} is a direct consequence of a stronger result that we now introduce. To study existence and uniqueness and to compare two solutions of \eqref{eq:one-dimSDE_general} with different drifts, we establish existence of a synchronous coupling of two copies of \eqref{eq:one-dimSDE_general},
\begin{equation} \label{eq:two_onedim_stickydiff}
\begin{aligned}
\rmd r_t&=(\tilde{b}(r_t)+P_t(g))\rmd t+2\1_{(0,\infty)}(r_t)\rmd W_t\eqsp, 
\\ \rmd s_t&=(\hat{b}(s_t)+\hat{P_t}(h))\rmd t+2\1_{(0,\infty)}(s_t)\rmd W_t\eqsp, \qquad \text{Law}(r_0,s_0)=\eta\eqsp,
\end{aligned}
\end{equation}
where $P_t=P\circ r_t^{-1}$, $\hat{P}_t=P\circ s_t^{-1}$,  $(W_t)_{t\geq 0}$ is a Brownian motion and where $\eta\in\Gamma(\mu, \nu)$ for $\mu, \nu\in\mathcal{P}(\mathbb{R}_+)$.

\begin{theorem} \label{thm:existence_comparison}
Suppose that $(\tilde{b},g)$ and $(\hat{b},h)$ satisfy \Cref{H1b} and \Cref{H1g}. Let $\eta\in\Gamma(\mu,\nu)$ where the probability measures $\mu$ and $\nu$ on $\mathbb{R}_+$ satisfy \Cref{H2}.
Then there exists a weak solution $(r_t,s_t)_{t\geq 0}$ of the sticky stochastic differential equation \eqref{eq:two_onedim_stickydiff} with initial distribution $\eta$ defined on a probability space $(\Omega, \mathcal{A},P)$ with values in $(\mathbb{W}\times\mathbb{W},\mathcal{B}(\mathbb{W})\otimes\mathcal{B}(\mathbb{W}))$.
If additionally,
\begin{align*}
&\tilde{b}(r)\leq \hat{b}(r)\quad and \quad g(r)\leq h(r) && \text{ for any } r\in\mathbb{R}_+, 
 \text{ and }
\\ & P[r_0\leq s_0]=1, 
\end{align*}
then $P[r_t\leq s_t \text{ for all } t\geq 0]=1$. 
\end{theorem}

\begin{proof}
  The proof is postponed to \Cref{sec:proof_nonlinear_existence}.
\end{proof}

\begin{remark} 
 We note that by the comparison result we can deduce uniqueness in law for the solution of \eqref{eq:one-dimSDE_general}.
\end{remark}

\subsection{Invariant measures and phase transition for \texorpdfstring{\eqref{eq:one-dim_stickydiff}}{}}
\label{subsection_stationarydistr}
Under the following conditions on the drift function $\tilde{b}$ we exhibit a phase transition phenomenon for the model \eqref{eq:one-dim_stickydiff}, where as compared to \eqref{eq:one-dimSDE_general} we focus on the case $P_t(g)=a \PP[r_t>0]$.

\begin{theorem} \label{thm:one_dim_SDE_statdistr_general}
Suppose \Cref{H1b} holds and $\limsup_{r\to \infty}(r^{-1}\tilde{b}(r))<0$.
Then, the Dirac measure at $0$, $\delta_0$, is an invariant probability measure for \eqref{eq:one-dim_stickydiff}. If there exists $p\in(0,1)$ solving 
\begin{align} \label{eq:defintion_p}
(2/a)=(1-p)I(a,p)
\end{align}
with
\begin{align} \label{eq:definition_I(a,p)}
I(a,p)=\int_0^\infty \exp\Big(\frac{1}{2}apx+\frac{1}{2}\int_0^x \tilde{b}(r)\rmd r\Big)\rmd x\eqsp,
\end{align}
then the probability measure 
$\pi$ on $[0,\infty)$  given by 
\begin{align} \label{eq:definition_stationarydistr}
\pi(\rmd x)\propto\Big(\frac{2}{ap}\delta_0(\rmd x)+\exp\Big(\frac{1}{2}apx+\frac{1}{2}\int_0^x \tilde{b}(r)\rmd r\Big) \lambda_{(0,\infty)}(\rmd x)\Big)
\end{align}
is another invariant probability measure for \eqref{eq:one-dim_stickydiff}.
\end{theorem}

\begin{proof}
  The proof is postponed to \Cref{sec:proof_nonlinear_invmeas}.
\end{proof}

In our next result we specify a necessary and sufficient condition for the existence of a solution of \eqref{eq:defintion_p}.
\begin{proposition} \label{thm:one_dim_SDE_stationarydistr}
Suppose that $\tilde{b}(r)$ in \eqref{eq:one-dim_stickydiff} is of the form $\tilde{b}(r)=-\tilde{L}r$ with constant a $\tilde{L}>0$. 
If $a/\sqrt{\tilde{L}} > 2/\sqrt{\pi}$, then there exists a unique $\hat{p}$ solving \eqref{eq:definition_I(a,p)}. In particular, the Dirac measure $\delta_0$ and the measure $\pi$ given in \eqref{eq:definition_stationarydistr} with $\hat{p}$ are invariant measures for \eqref{eq:one-dim_stickydiff}.
On the other hand, if $a/\sqrt{\tilde{L}}\leq 2/\sqrt{\pi}$, then there exists no $\hat{p}$ solving \eqref{eq:definition_I(a,p)}. 
\end{proposition}
\begin{proof}
  The proof is postponed to \Cref{sec:proof_nonlinear_invmeas}.
\end{proof}

\subsection{Convergence for sticky nonlinear SDEs of the form \texorpdfstring{\eqref{eq:one-dim_stickydiff}}{}}
Under \Cref{H1b} and the following additional assumption we establish geometric convergence in Wasserstein distance for the marginal law of the solution $r_t$ of \eqref{eq:one-dim_stickydiff} to the Dirac measure at $0$:
\begin{assumptionH} \label{A2}
It holds $\limsup_{r\to \infty}(r^{-1}\tilde{b}(r))<0$ and
$a\leq (2\int_0^{\tilde{R}_1}\exp\big(\frac{1}{2}\int_0^s \tilde{b}(u)_+\rmd u\big)ds)^{-1}$
with $\tilde{R}_0, \tilde{R}_1 $ defined by
\begin{align}
\tilde{R}_0&=\inf\{s\in\mathbb{R}_+: \tilde{b}(r)\leq0 \ \forall r\geq s\} \eqsp \qquad \text{and} \label{eq:definition_R_0}
\\ \tilde{R}_1&=\inf\{s\geq \tilde{R}_0: -\frac{s}{r}(s-\tilde{R}_0) \tilde{b}(r)\geq 4 \ \forall r\geq s\}\eqsp. \label{eq:definition_R_1}
\end{align}
\end{assumptionH}

\begin{theorem} \label{thm:one-dim_stickydiff}
Suppose \Cref{H1b} and \Cref{A2} holds. Then, the Dirac measure at $0$, $\delta_0$, is the unique invariant probability measure of \eqref{eq:one-dim_stickydiff}. Moreover if $(r_s)_{s\geq 0}$ is a solution of \eqref{eq:one-dim_stickydiff} with $r_0$ distributed with respect to an arbitrary probability measure $\mu$ on $(\mathbb{R}_+,\mathcal{B}(\mathbb{R}_+))$, it holds for all $t\geq 0 $, 
\begin{align} \label{eq:stickydiff_estimateinthm}
\mathbb{E}[f(r_t)]\leq \rme^{-ct}\mathbb{E}[f(r_0)]\eqsp,
\end{align} 
where $f$ and $c$ are given by \eqref{eq:definition_f} and \eqref{eq:definition_c} with $a$ and $\tilde{b}$ given in \eqref{eq:one-dim_stickydiff} and $\tilde{R}_0$ and $\tilde{R}_1$ given in \eqref{eq:definition_R_0} and \eqref{eq:definition_R_1}. 
\end{theorem} 

\begin{proof}
  The proof is postponed to \Cref{sec:proof_nonlinear_convergence}.
\end{proof}

\section{Uniform in time propagation of chaos}\label{sec:propachaos}

To prove uniform in time propagation of chaos,  we consider the $L^1$ Wasserstein distance with respect to the cost function $\bar{f}_N\circ \pi:\mathbb{R}^{Nd}\times\mathbb{R}^{Nd}\to\mathbb{R}_+$ 
with $\pi$ given in \eqref{eq:proj_pi}, 
and $\barf_N$ given by
\begin{equation}
            \label{eq:def_l1_f_N}
       \barf_N((x^{i,N})_{i=1}^N,(y^{i,N})_{i=1}^N)=\frac{1}{N}\sum_{i=1}^N f\parenthese{\abs{x^i-y^i }} \eqsp, 
     \end{equation}
with $f: \rset_+ \to \rset_+$ defined in \eqref{eq:definition_f}. This distance is denoted by $\mathcal{W}_{f,N}$. Note that $\bar{f}_N$ is equivalent to $l^1$ defined in \eqref{eq:def_l1}.

We note that since $\pi$ defines a projection from $\mathbb{R}^{Nd}$ to the hyperplane $\msh_N\subset \mathbb{R}^{Nd}$ given in \eqref{eq:def_H_N}, for $\hat{\mu}$ and $\hat{\nu}$ on $\msh_N$, $\mathcal{W}_{f,N}(\hat{\mu},\hat{\nu})$ coincides with the Wasserstein distance given by 
\begin{align} \label{eq:projectedWassersteindist}
\hat{\mathcal{W}}_{f,N}(\hat{\mu},\hat{\nu})=\inf_{\xi\in\Gamma(\hat{\mu},\hat{\nu})}\int_{\msh_N \times \msh_N} \barf_N (x,y) \xi(\rmd x \rmd y)
\end{align}
and $\mathcal{W}_{l^1\circ \pi}(\hat{\mu},\hat{\nu})=\hat{\mathcal{W}}_{l^1}(\hat{\mu},\hat{\nu})$, where $\bar{f}_N$ and $l^1$ are given in \eqref{eq:def_l1_f_N} and \eqref{eq:def_l1}, respectively, and where $\hat{\mathcal{W}}_{l^1}(\hat{\mu},\hat{\nu})$ is defined as in \eqref{eq:projectedWassersteindist} with respect to the distance $l^1$.

\begin{theorem}[Uniform in time propagation of chaos]\label{thm:timepropagation_of_chaos} Let $N\in\mathbb{N}$ and assume \Cref{ass:decomp_W} and \Cref{ass:gamma_bound}. Let $\barmu_0$ and $\nu_0$ be probability measures on $(\mathbb{R}^d,\mathcal{B}(\mathbb{R}^d))$ satisfying \Cref{ass:init_distr}. For $t \geq 0$, denote by $\barmu_t$ and $\nu_t^N$ the law of $\bar{X}_t$ and  $\{X_t^{i,N}\}_{i=1}^N$ where $(\bar{X}_s)_{s \geq 0}$ and $(\{X_s^{i,N}\}_{i=1}^N)_{s \geq 0}$ are solutions of  \eqref{eq:nonlinearSDE} and \eqref{eq:meanfield}, respectively, with initial distributions $\barmu_0$ and $\nu_0^{\otimes N}$. 
Then for all $t\geq 0$,
\begin{align*}
\mathcal{W}_{f,N}(\barmu_t^{\otimes N},\nu_t^N)&\leq \rme^{-\tilde{c} t}\mathcal{W}_{f,N}(\barmu_0^{\otimes N},\nu_0^{\otimes N})+\tilde{C}\tilde{c}^{-1}N^{-1/2}\eqsp,
\\ \mathcal{W}_{l^1\circ \pi}(\barmu_t^{\otimes N},\nu_t^N)&\leq M_1 \rme^{-\tilde{c} t}\mathcal{W}_{l^1\circ \pi}(\barmu_0^{\otimes N},\nu_0^{\otimes N})+M_1\tilde{C}\tilde{c}^{-1}N^{-1/2}\eqsp,
\end{align*}
where $f$ is defined by \eqref{eq:definition_f}, $M_1$ by \eqref{eq:def_M_1}, $\tilde{c}$ by \eqref{eq:definition_tildec}
and $\tilde{C}$ is a finite constant depending on $\|\gamma\|_\infty$, $L$ and the second moment of $\barmu_0$ and given in \eqref{eq:tildeC}. 
\end{theorem}
\begin{proof}
  The proof is postponed to \Cref{sec:proof_thm_propag}.
\end{proof}

\begin{remark}
  Denote by $\mu_t^N$ and $\nu_t^N$ the distribution of $\{X_t^{i,N}\}_{i=1}^N$ and $\{Y_t^{i,N}\}_{i=1}^N$ where the two processes $(\{X_s^{i,N}\}_{i=1}^N)_{s \geq 0}$ and $(\{Y_s^{i,N}\}_{i=1}^N)_{s \geq 0}$ are solutions of \eqref{eq:meanfield} with initial probability distributions $\mu_0^{N},\nu_0^{N}\in\mathcal{P}(\mathbb{R}^{Nd})$, respectively, with finite forth moment. 
An easy inspection and adaptation of the proof of \Cref{thm:timepropagation_of_chaos} show that if \Cref{ass:decomp_W} holds, then 
\begin{equation*}
\begin{aligned}
  \mathcal{W}_{f,N}(\mu_t^N,\nu_t^N )\leq \rme^{-\tilde{c} t}\mathcal{W}_{f,N}(\mu_0^{\otimes N},\nu_0^{\otimes N})\eqsp, \qquad\mathcal{W}_{l^1\circ \pi}(\mu_t^N,\nu_t^N)\leq 2 M_1 \rme^{-\tilde{c} t}\mathcal{W}_{l^1\circ \pi}(\mu_0^{\otimes N},\nu_0^{\otimes N}) \eqsp,
  \end{aligned}
\end{equation*}
where $f$, $\tilde{c}$ and $M_1$ are defined as in \Cref{thm:timepropagation_of_chaos}.

\end{remark}

\section{System of N sticky SDEs} \label{sec:system_non-linear_sticky_diff}

Consider a systerm of $N$ one-dimensional SDEs with sticky boundaries at $0$ given by 
\begin{equation}\label{eq:N_onedimSDE}
\rmd r_t^i=\Big(\tilde{b}(r_t^i)+\frac{1}{N}\sum_{j=1}^N g(r_t^j)\Big)\rmd t+2 \1_{(0,\infty)}(r_t^i)\rmd W_t^i\eqsp, \qquad i=1,\ldots,N.
\end{equation} 
The results on existence, uniqueness and the comparison theorem for solutions of sticky nonlinear SDEs mostly carry directly over to a solution of \eqref{eq:N_onedimSDE} and are applied to prove propagation of chaos in \Cref{thm:timepropagation_of_chaos}.

Let $\mu$ be a probability distribution on $\mathbb{R}_+$.
For $N\in\mathbb{N}$, $(\{r_t^i,W_t^i\}_{i=1}^N)_{t\geq 0}$ is a weak solution on the filtered probability space $(\Omega, \mathcal{A},(\mathcal{F}_t)_{t\ge 0},P)$ of \eqref{eq:N_onedimSDE}
with initial distribution $\mu^{\otimes N}$
if the following hold: $\mu^{\otimes N}=P\circ (\{r_0\}_{i=1}^N)^{-1}$, $(\{W_t\}_{i=1}^N)_{t\geq 0}$ is a $N$-dimensional $(\mathcal{F}_t)_{t\ge 0}$ Brownian motion w.r.t. $P$, 
the process $(r_t^i)_{t\geq 0}$ is non-negative, continuous and satisfies almost surely for any $i\in\{1,\ldots,N\}$ and $t\in\mathbb{R}_+$,
\begin{align*}
r_t^i-r_0^i&=\int_0^t\Big(\tilde{b}(r_s^i)+\frac{1}{N}\sum_{j=1}^N g(r_s^j)\Big)\rmd s+\int_0^t 2\1_{(0,\infty)}(r_s^i) \rmd W_s^i\eqsp.
\end{align*}

To show existence and uniqueness in law of a weak solution $(\{r_t^i,W_t^i\}_{i=1}^N)_{t\geq 0}$, we suppose \Cref{H1b} and \Cref{H1g} for $\tilde{b}$ and $g$. 

It follows that there exists a constant $C<\infty$ such that for all $\{r^i\}_{i=1}^N\in\mathbb{R}_+^N$, it holds $\sum_{i=1}^N |\tilde{b}(r^i)|+|g(r^i)|\leq C(1+\sum_{i=1}^N|r^i|)$, and a possible solution $(\{r_t^i\}_{i=1}^N)_{t\geq 0}$ is non-explosive.
If the initial distribution is supported on $\mathbb{R}_+^N$, then in the same line as for the nonlinear SDE in \Cref{sec:existence}, the solution $(\{r_t^i\}_{i=1}^N)_{t\geq 0}$ satisfies $r^i_t>0$ almost surely for any $i=1,\ldots,N$ and $t\geq 0$ by \Cref{H1b} and \Cref{H1g}.

Existence and uniqueness in law of \eqref{eq:N_onedimSDE} is a direct consequence of a stronger result that we now introduce. To study existence and uniqueness and to compare two solutions of \eqref{eq:N_onedimSDE} with different drifts, we establish existence of a synchronous coupling of two copies of \eqref{eq:N_onedimSDE},
\begin{equation} \label{eq:N_onedimSDE_coupling}
\begin{aligned}
&\rmd r_t^i=\Big(\tilde{b}(r_t^i)+\frac{1}{N}\sum_{j=1}^Ng(r_t^j)\Big)\rmd t+2 \1_{(0,\infty)}(r_t^i) \rmd W_t^i \eqsp,
\\ &\rmd s_t^i=\Big(\hat{b}(s_t^i)+\frac{1}{N}\sum_{j=1}^N h(s_t^j)\Big)\rmd t+2 \1_{(0,\infty)}(s_t^i) \rmd W_t^i\eqsp, 
\\ &  \Law(r_0^i,s_0^i)=\eta\eqsp,
\end{aligned}\qquad \text{for $i\in\{1,\ldots,N\}$}
\end{equation}
where $(\{W_t^i\}_{i=1}^N)_{t\geq 0}$ are $N$ \iid $1$-dimensional Brownian motions and where $\eta\in\Gamma(\mu, \nu)$ for $\mu, \nu\in\mathcal{P}(\mathbb{R}_+)$. 

Let $\mathbb{W}^{N}=\mathcal{C}(\mathbb{R}_+,\mathbb{R}^{N})$ be the space of continuous functions from $\mathbb{R}_+$ to $\mathbb{R}^{N}$ endowed with the topology of uniform convergence on compact sets, and let $\mathcal{B}(\mathbb{W}^{N})$ denote its Borel $\sigma$-Algebra. 

\begin{theorem} \label{thm:existence_comparison_Nparticles}
 Assume that $(\tilde{b},g)$ and $(\hat{b},h)$ satisfy \Cref{H1b} and \Cref{H1g}. Let $\eta\in\Gamma(\mu,\nu)$ where $\mu$ and $\nu$ are the probability measure on $\mathbb{R}_+$ satisfying \Cref{H2}.
Then there exists a weak solution $(\{r^i_t,s^i_t\}_{i=1}^N)_{t\geq 0}$ of the sticky stochastic differential equation \eqref{eq:N_onedimSDE_coupling} with initial distribution $\eta^{\otimes N}$ defined on a probability space $(\Omega, \mathcal{A},P)$ with values in $\mathbb{W}^N\times\mathbb{W}^N$.
If additionally, 
\begin{align*}
&\tilde{b}(r)\leq \hat{b}(r) \quad and \quad g(r)\leq h(r)\eqsp,  && \text{for any } r\in\mathbb{R}_+ \eqsp,
\\ & 
P[r_0^{i}\leq s_0^{i} \text{ for all } i=1,\ldots,N]=1\eqsp,
\end{align*}
then $P[r_t^i\leq s_t^i \text{ for all } t\geq 0\text{ and } i=1,\ldots,N]=1$. 
\end{theorem}

\begin{proof}
  The proof is postponed to \Cref{sec:proof_meanfield_existence}.
\end{proof}

\begin{remark}
 We note that by the comparison result we can deduce uniqueness in law for the solution of \eqref{eq:N_onedimSDE}.
\end{remark}

\section{Proofs} \label{sec:proofs}

Before proving the statements of Section~\ref{sec:main_contraction_results}-\ref{sec:system_non-linear_sticky_diff}, let us give an overview of the proofs.
The first subsection gives the definition of the underlying distance function $f$ used in \Cref{thm:nonlinearSDE}, \Cref{thm:one-dim_stickydiff} and \Cref{thm:timepropagation_of_chaos}.
\Cref{sec:proof_nonlinearSDE} and \Cref{sec:proofs_nonlinearSDE} provide proofs for the convergence result for the nonlinear SDE (\Cref{thm:nonlinearSDE}) using the sticky coupling approach and the results for the sticky nonlinear SDE (\Cref{thm:one-dim_stickydiff}). 
Note that both \Cref{thm:nonlinearSDE} and \Cref{thm:one-dim_stickydiff} use the auxiliary Lemmata \ref{lemma:modification_watanabe_nonlinear}-\ref{lem:existence_stickySDE_step2}, where a comparison result and an approximation in two steps of the sticky nonlinear SDE are given. The existence of a solution to the sticky nonlinear SDE and a comparison result are essential to show contraction in this approach.

In \Cref{sec:proof_thm_propag} and \Cref{sec:proof_meanfield_existence} the proofs for the propagation of chaos for the mean-field particle system and for the system of sticky SDEs are given.
Note that the techniques to prove the result for the particle systems and the system of $N$ sticky SDEs are partially similar to the nonlinear case. In particular, the proofs of \Cref{thm:timepropagation_of_chaos} and \Cref{thm:existence_comparison_Nparticles} and its auxiliary Lemmata \ref{lem:r_t^idelta}, \ref{lemma:comparison_propofchaos}-\ref{lem:existence_NstickySDE_step2} have a similar structure as the ones of \Cref{theo:1} and \Cref{thm:existence_comparison} and its auxiliary Lemmata \ref{lemma:SDEr_t^delta}-\ref{lem:existence_stickySDE_step2}, respectively.

\subsection{Definition of the metrics}

In \Cref{thm:nonlinearSDE}, \Cref{thm:one-dim_stickydiff} and \Cref{thm:timepropagation_of_chaos} we consider Wasserstein distances based on a carefully designed concave function $f:\mathbb{R}_+\to\mathbb{R}_+$ that we now define. In addition we derive useful properties of this function that will be used in our proofs of \Cref{thm:nonlinearSDE}, \Cref{thm:timepropagation_of_chaos} and \Cref{thm:one-dim_stickydiff}.
Let $a\in\mathbb{R}_+$ and $\tilde{b}:\mathbb{R}_+\to\mathbb{R}$ be such that \Cref{A2} is satisfied with $\tilde{R}_0$ and $\tilde{R}_1$ defined in \eqref{eq:definition_R_0}. 
We define
\begin{align*}
\varphi(r)&=\exp\parenthese{-\int_0^r \{\tilde{b}(s)_+/2\}\rmd s} \eqsp, \qquad \Phi(r)=\int_0^r \varphi(s)\rmd s\eqsp, && \text{ and }
\\ g(r)&=1-\frac{c}{2}\int_0^{r\wedge \tilde{R}_1}\{\Phi(s)/\varphi(s)\}\rmd s-\frac{a}{2}\int_0^{r \wedge \tilde{R}_1} \{1/\varphi(s)\}\rmd s \eqsp,
\end{align*}
where
\begin{align} \label{eq:definition_c}
c=\parenthese{2\int_0^{\tilde{R}_1}  \{\Phi(s)/\varphi(s)\}\rmd s}^{-1},
\end{align}
and $\tilde{R}_1$ is given in \eqref{eq:definition_R_1}.
It holds $\varphi(r)=\varphi(\tilde{R}_0)$ for $ r\geq \tilde{R}_0$ with $\tilde{R}_0$ given in \eqref{eq:definition_R_0}, $g(r)=g(\tilde{R}_1)\in[1/2,3/4]$ for $r\geq \tilde{R}_1$ and  $g(r)\in[1/2,1]$ for all $r\in \mathbb{R}_+$ by \eqref{eq:definition_c} and \Cref{A2}.  
We define the increasing function $f:[0,\infty)\to [0,\infty)$ by
\begin{align} \label{eq:definition_f}
f(t)=\int_0^t \varphi(r)g(r) \rmd r \eqsp.
\end{align}
The construction is adapted from the function $f$ given in \cite{Eb16}.
Here, the function $g$ has an extra term. As we see later in the proof of \Cref{thm:nonlinearSDE} and \Cref{thm:one-dim_stickydiff}, this term has the purpose to control the term $a\PP[r_t>0]$.
We observe that $f$ is concave, since $\varphi$ and $g$ are decreasing.
Since for all $r\in \mathbb{R}_+$
\begin{align} \label{eq:norm_equivalence}
\varphi(\tilde{R}_0) r/2\leq \Phi(r)/2\leq f(r)\leq \Phi(r)\leq r\eqsp,
\end{align}
 $(x,y)\mapsto f(|x-y|)$ defines a distance on $\mathbb{R}^d$ equivalent to the Euclidean distance on $\mathbb{R}^d$. 
 
Moreover, $f$ satisfies
\begin{align}
2f''(0)=-\tilde{b}(0)_+-a = -a\eqsp, \label{eq:condition_f_1}
\end{align}
and 
\begin{align}
 2f''(r)\leq 2f''(0)-f'(r)\tilde{b}(r)-cf(r)\eqsp,\qquad \text{for all $r\in\mathbb{R}_+ \backslash \{\tilde{R}_1\}$}\eqsp. \label{eq:condition_f_2}
\end{align}
Indeed by construction of $f$, $f''(r)=-\tilde{b}(r)_+f'(r)/2-c\Phi(r)/2-a/2$ for $0\leq r< \tilde{R}_1$ and so \eqref{eq:condition_f_2} holds for $0\leq r< \tilde{R}_1$ by \eqref{eq:norm_equivalence}.
To show \eqref{eq:condition_f_2} for $r>\tilde{R}_1$ note that $f''(r)=0$ and $f'(r)\geq \varphi(\tilde{R}_0)/2$ hold for $r>\tilde{R}_1$. Hence, by the definition \eqref{eq:definition_R_1} of $\tilde{R}_1$, for $r>\tilde{R}_1$,
\begin{align} \label{eq:proof_thm1_estimate}
f''(r)+f'(r)\tilde{b}(r)/2\leq\varphi(\tilde{R}_0)\tilde{b}(r)/4\leq -(\tilde{R}_1(\tilde{R}_1-\tilde{R}_0))^{-1}\varphi(\tilde{R}_0)r\eqsp.
\end{align}
Since $\varphi(r)=\varphi(\tilde{R}_0)$ for $r\geq \tilde{R}_0$, it holds $\Phi(r)=\Phi(\tilde{R}_0)+(r-\tilde{R}_0)\varphi(\tilde{R}_0)$
for $r\geq \tilde{R}_0$. Further, it holds $\Phi(R_0)\geq \tilde{R}_0\varphi(\tilde{R}_0)$ since $\varphi$ is decreasing for $r\leq \tilde{R}_0$.
Hence,
\begin{align}
\frac{r}{\tilde{R}_1}&=\frac{(r-\tilde{R}_1)(\Phi(\tilde{R}_0)+(\tilde{R}_1-\tilde{R}_0)\varphi(\tilde{R}_0))}{\tilde{R}_1\Phi(\tilde{R}_1)}+1\geq \frac{(r-\tilde{R}_1)\tilde{R}_1\varphi(\tilde{R}_0)}{\tilde{R}_1\Phi(\tilde{R}_1)}+1 
=\frac{\Phi(r)}{\Phi(\tilde{R}_1)}\eqsp. \label{eq:thm1_estimate1}
\end{align}
Furthermore, we have
\begin{align}
\int_{\tilde{R}_0}^{\tilde{R}_1} \{\Phi(s)/\varphi(s)\}\rmd s&=\int_{\tilde{R}_0}^{\tilde{R}_1} \frac{\Phi(\tilde{R}_0)+(s-\tilde{R}_0)\varphi(\tilde{R}_0)}{\varphi(\tilde{R}_0)}\rmd s \nonumber
\\ & = (\tilde{R}_1-\tilde{R}_0)\frac{\Phi(\tilde{R}_0)}{\varphi(\tilde{R}_0)}+\frac{1}{2}(\tilde{R}_1-\tilde{R}_0)^2 \geq  \frac{1}{2}(\tilde{R}_1-\tilde{R}_0)\frac{\Phi(\tilde{R}_1)}{\varphi(\tilde{R}_0)}\eqsp.  \label{eq:thm1_estimate2}
\end{align}
We insert \eqref{eq:thm1_estimate1} and \eqref{eq:thm1_estimate2} in \eqref{eq:proof_thm1_estimate} and use \eqref{eq:definition_c} to obtain
\begin{align} \label{eq:f_computation}
f''(r)+f'(r)\tilde{b}(r)/2&\leq -\Phi(r)\Phi(\tilde{R}_1)^{-1}(\tilde{R}_1-\tilde{R}_0)^{-1}\varphi(\tilde{R}_0)
\\ & \leq -\frac{\Phi(r)}{2\int_{\tilde{R}_0}^{\tilde{R}_1} \{\Phi(s)/\varphi(s)\}\rmd s}
\leq -\frac{cf(r)}{2}-\frac{c\Phi(r)}{2}\eqsp.
\end{align}
By \Cref{A2} and \eqref{eq:definition_c}, we get
\begin{align*}
-\frac{c\Phi(r)}{2}\leq -\frac{\Phi(\tilde{R}_1)}{4\int_0^{\tilde{R}_1} \{\Phi(s)/\varphi(s)\}\rmd s}\leq -\frac{1}{4\int_0^{\tilde{R}_1}\{1/\varphi(s)\}\rmd s}\leq -\frac{a}{2}=f''(0)\eqsp.
\end{align*}
Combining this estimate with \eqref{eq:f_computation} gives \eqref{eq:condition_f_2} for $r>\tilde{R}_1$. 
Hence, the choice of the underlying function $f$ for the Wasserstein distance ensures \eqref{eq:condition_f_1} and \eqref{eq:condition_f_2}. These properties guarantee that the term $a\PP[r_t>0]$ is controlled in \eqref{eq:one-dim_stickydiff} 
and contraction with rate $c$ is obtained in \Cref{thm:nonlinearSDE}, \Cref{thm:one-dim_stickydiff} and \Cref{thm:timepropagation_of_chaos}.

\subsection{Proof of  \texorpdfstring{\Cref{sec:main_contraction_results}}{}}\label{sec:proof_nonlinearSDE}

First, we prove \Cref{thm:nonlinearSDE} by using \Cref{theo:1} and properties of the carefully constructed function $f$ before we show \Cref{theo:1}. To prove that the dominating process $r_t$ exists we make use of the result of the sticky nonlinear SDE which are proven in \Cref{sec:proof_nonlinear_existence}. 

\subsubsection{Proof of \texorpdfstring{\Cref{thm:nonlinearSDE}}{}} \label{sec:proof_thm_nonlinear}

\begin{proof}[Proof of \Cref{thm:nonlinearSDE}]
  We consider the process $(\bar{X}_t,\bar{Y}_t,r_t)_{t\geq 0}$ defined in \Cref{theo:1} and  satisfying $|\bar{X}_t-\bar{Y}_t|\leq r_t$ for any
  $ t\geq 0$, and  $(r_t)_{t\geq 0}$ is a weak solution of   \eqref{eq:non_linear_sticky_sde}.
Set $a=2\|\gamma\|_\infty$ and $\tilde{b}(r)=\bar{b}(r)$. 
With this notation, \Cref{ass:decomp_W} and \Cref{ass:gamma_bound} imply \Cref{A2} and $\tilde{R}_0=R_0$ and $\tilde{R}_1=R_1$ by \eqref{eq:definition_R0}, \eqref{eq:definition_R1}, \eqref{eq:definition_R_0} and \eqref{eq:definition_R_1}.
By It\=o-Tanaka formula, cf. \cite[Chapter 6, Theorem 1.1]{ReYo99}, using that $f'$ is absolutely continuous, we have,
\begin{align*}
\rmd f(r_t)&\leq f'(r_t)(\bar{b}(r_t)+2\|\gamma\|_\infty \PP(r_t>0))\rmd t +2f''(r_t)\1_{(0,\infty)}(r_t)\rmd t 
\\ & +f'(r_t)2\1_{(0,\infty)}(r_t)\rmd W_t\eqsp.
\end{align*}
Taking expectation we obtain by \eqref{eq:condition_f_1} and \eqref{eq:condition_f_2}
\begin{align*}
\frac{\rmd }{\rmd t}\mathbb{E}[f(r_t)]\leq \mathbb{E}[f'(r_t)\tilde{b}(r_t)_++ 2(f''(r_t)-f''(0))]+\mathbb{E}[(a+2f''(0))\1_{r_t>0}]\leq -\tilde{c}\mathbb{E}[f(r_t)]\eqsp,
\end{align*}
where $\tilde{c}$ is given by \eqref{eq:definition_tildec}.
Therefore by Gr\"onwall's lemma, 
\begin{align*}
\mathbb{E}[f(|\bar{X}_t-\bar{Y}_t|)]\leq \mathbb{E}[f(r_t)]\leq \rme^{-\tilde{c}t} \mathbb{E}[f(r_0)]=\rme^{-\tilde{c}t}\mathbb{E}[f(|\bar{X}_0-\bar{Y}_0|)]\eqsp.
\end{align*}
Hence, it holds 
\begin{align*}
\mathcal{W}_f(\barmu_t,\bar{\nu}_t)\leq \mathbb{E}[f(|\bar{X}_t-\bar{Y}_t|)]\leq \rme^{-\tilde{c}t}\int_{\mathbb{R}^d\times\mathbb{R}^d} f(|x-y|) \xi(\rmd x \rmd y)
\end{align*}
for an arbitrary coupling $\xi\in\Gamma(\mu_0,\nu_0)$. Taking the infimum over all couplings $\xi\in\Gamma(\mu_0,\nu_0)$, we obtain the first inequality of \eqref{eq:thm_nonlinearSDE}.
By \eqref{eq:norm_equivalence}, we get the second inequality of \eqref{eq:thm_nonlinearSDE}.
\end{proof}

\subsubsection{Proof of \texorpdfstring{\Cref{theo:1}}{}} \label{sec:proof_theo1}
Note that the nonlinear SDE \eqref{eq:coupling_non_linearSDE_approx} has Lipschitz continuous coefficients. The existence and the uniqueness of the coupling $(\bar{X}_t^\delta,\bar{Y}_t^\delta)_{t\geq 0}$ follows from \cite[Theorem 2.2]{Me96}. By Levy's characterization, $(\bar{X}_t^\delta,\bar{Y}_t^\delta)_{t\geq 0}$ is indeed a coupling of two copies of solutions of \eqref{eq:nonlinearSDE}. Further, we remark that $W_t^\delta=\int_0^t (\bar{e}_s^\delta)^T \rmd B_s^1$ is a one-dimensional Brownian motion.
In the next step, we analyse $|\bar{X}_t^\delta-\bar{Y}_t^\delta|$. 

\begin{lemma} \label{lemma:SDEr_t^delta}
Suppose that the conditions \Cref{ass:decomp_W} and \Cref{ass:init_distr} are satisfied. Then, it holds 
for any $\epsilon<\epsilon_0$, where $\epsilon_0$ is given by \eqref{eq:condition_rc_sc2}, setting   $\bar{r}_t^\delta=|\bar{X}_t^\delta-\bar{Y}_t^\delta|$
\begin{align}
\rmd \bar{r}_t^\delta &= \Big(-L \bar{r}_t^\delta+\Big\langle{\bar{e}_t^\delta},\int_{\mathbb{R^d}}\int_{\mathbb{R}^d}\gamma(\bar{X}_t^\delta-x)-\gamma(\bar{Y}_t^\delta-y)\mu_t^\delta(\rmd x)\nu_t^\delta(\rmd y)\Big\rangle \Big) \rmd t+2 \mathrm{rc}^\delta(\bar{r}_t^\delta)\rmd W_t^\delta \label{eq:one-dimSDE_approx_estimate1}
\\ & \leq \Big(\bar{b}(\bar{r}_t^\delta)+2\|\gamma\|_\infty \int_{\mathbb{R}^{d}} \int_{\mathbb{R}^{d}} \mathrm{rc}^\epsilon(|x-y|)  \barmu_t^\delta(\rmd x) \bar{\nu}_t^\delta(\rmd y)\Big) \rmd t+2 \mathrm{rc}^\delta(\bar{r}_t^\delta)\rmd W_t^\delta\eqsp, \label{eq:one-dimSDE_approx_estimate2}
\end{align}
almost surely for all $t\geq 0$,
where $\barmu_t^\delta$ and $\bar{\nu}_t^\delta$ are the laws of $\bar{X}_t^\delta$ and $\bar{Y}_t^\delta$, respectively.
\end{lemma}

\begin{proof}
Using \eqref{eq:coupling_non_linearSDE_approx}, \Cref{ass:decomp_W} and \Cref{ass:init_distr}, the stochastic differential equation of the process $((\bar{r}_t^\delta)^2)_{t\geq 0}$ is given by
\begin{align*}
\rmd ((\bar{r}_t^\delta)^2)&=2 \Big\langle Z_t^\delta,-LZ_t^\delta +  \int_{\mathbb{R}^d}\int_{\mathbb{R}^d} \gamma(\bar{X}_t^\delta-x)-\gamma(\bar{Y}_t^\delta-y)  \bar{\mu}_t^\delta(\rmd x) \bar{\nu}_t^\delta(\rmd y) \Big\rangle \rmd t 
\\ & + 4 \mathrm{rc}^\delta(\bar{r}_t^\delta)^2\rmd t+ 4\mathrm{rc}^\delta (\bar{r}_t^\delta)\langle Z_t^\delta, e_t^\delta\rangle \rmd W_t^\delta\eqsp.
\end{align*}
For ${{\varepsilon}}>0$ we define as in \cite[Lemma 8]{EbZi19} a $\mathcal{C}^2$ approximation of the square root by
\begin{align*}
S_{{\varepsilon}}(r)=\begin{cases}(-1/8){{\varepsilon}}^{-3/2}r^2+(3/4){{\varepsilon}}^{-1/2}r+(3/8){{\varepsilon}}^{1/2} & \text{for } r<{{\varepsilon}} \\ \sqrt{r} & \text{otherwise}\eqsp. \end{cases}
\end{align*}
Then, by It\=o's formula, 
\begin{align*}
&\rmd S_{{\varepsilon}}((\bar{r}_t^\delta)^2)=S_{{\varepsilon}}'((\bar{r}_t^\delta)^2)\rmd (\bar{r}_t^\delta)^2+\frac{1}{2}S_{{\varepsilon}}''((\bar{r}_t^\delta)^2)\rmd [(\bar{r}^\delta)^2]_t
\\ & = 2 S_{{\varepsilon}}'((\bar{r}_t^\delta)^2) \Big\langle Z_t^\delta,-LZ_t^\delta + \int_{\mathbb{R}^d}\int_{\mathbb{R}^d} \gamma(\bar{X}_t^\delta-x)-\gamma(\bar{Y}_t^\delta-y) \barmu_t^\delta(\rmd x) \bar{\nu}_t^\delta(\rmd y) \Big\rangle \rmd t
\\ &  + S_{{\varepsilon}}'((\bar{r}_t^\delta)^2) 4 \mathrm{rc}^\delta(\bar{r}_t^\delta)^2\rmd t+ S_{{\varepsilon}}'((\bar{r}_t^\delta)^2)4 \mathrm{rc}^\delta (\bar{r}_t^\delta)\langle Z_t^\delta, e_t^\delta\rangle \rmd W_t^\delta
+ 8 S_{{\varepsilon}}''((\bar{r}_t^\delta)^2)(\mathrm{rc}^\delta(\bar{r}_t^\delta))^2 (\bar{r}_t^\delta)^2 \rmd t\eqsp.
\end{align*} 
We take the limit ${{\varepsilon}} \to 0 $. 
Then 
$\lim_{{\varepsilon}\to 0} S_{{\varepsilon}}'(r)=(1/2)r^{-1/2}$ and $\lim_{{\varepsilon}\to 0} S_{{\varepsilon}}''(r)=-(1/4)r^{-3/2}$ for $r > 0$. 
Since $\sup_{0\leq r\leq \varepsilon}|S_{{\varepsilon}}'(r)|\lesssim {{\varepsilon}}^{-1/2}$, $\sup_{0\leq r\leq \epsilon}|S_{\bar{\varepsilon}}''(r)|\lesssim {\bar{\varepsilon}}^{-3/2}$ and $\mathrm{rc}^\delta$ is Lipschitz continuous with $\mathrm{rc}^\delta(0)=0$, we apply Lebesgue's dominated convergence theorem to show convergence for the integrals with respect to time $t$. More precisely, we note that the integrand $(4S_{\varepsilon}'((\bar{r}_t^\delta)^2)+8S_{\varepsilon}''((\bar{r}_t^\delta)^2))\mathrm{rc}^\delta(\bar{r}_t^\delta))^2(\bar{r}_t^\delta)^2$ is dominated by $3\varepsilon^{1/2}\|\mathrm{rc}^\delta\|_{\Lip}$. 
For any $\varepsilon<\varepsilon_0$ for fixed $\varepsilon_0>0$, the integrand $2 S_{{\varepsilon}}'((\bar{r}_t^\delta)^2) \langle Z_t^\delta,-LZ_t^\delta + \int_{\mathbb{R}^d}\int_{\mathbb{R}^d} (\gamma(\bar{X}_t^\delta-x)-\gamma(\bar{Y}_t^\delta-y))  \barmu_t^\delta(\rmd x) \bar{\nu}_t^\delta(\rmd y)\rangle$ is dominated by $(3/2)(L\max(\varepsilon_0^{(1/2)},\bar{r}_t^\delta)+2\|\gamma\|_\infty)$.

For the stochastic integral it holds $|S_{\varepsilon}'((\bar{r}_t^\delta)^2)4\mathrm{rc}^\delta(\bar{r}_t^\delta) \bar{r}_t^\delta|\leq 3$. Hence, the stochastic integral converges along a subsequence almost surely, to $\int_0^t2 \mathrm{rc}^\delta(\bar{r}_s^\delta) \rmd W_s^\delta$, see \cite[Chapter 4, Theorem 2.12]{ReYo99}. 
Hence, we obtain \eqref{eq:one-dimSDE_approx_estimate1}. 
Since \eqref{eq:kappa} implies $\langle x-y, \gamma(x-\tilde{x})-\gamma(y-\tilde{x})\rangle \le \kappa(|x-y|)|x-y|^2$ for all $x,y,\tilde{x}\in\mathbb{R}^d$, we obtain by \Cref{ass:decomp_W} and \eqref{eq:condition_rc_sc2} for $\epsilon<\epsilon_0$ 
\begin{align*}
\Big\langle{\bar{e}_t^\delta}&,\int_{\mathbb{R}^d}\int_{\mathbb{R}^d}(\gamma(\bar{X}_t^\delta-x)-\gamma(\bar{Y}_t^\delta-y))\mu_t^\delta(\rmd x)\nu_t^\delta(\rmd y)\Big\rangle 
\\ & \leq \Big\langle{\bar{e}_t^\delta},\int_{\mathbb{R}^d}\int_{\mathbb{R}^d}(\gamma(\bar{X}_t^\delta-x)-\gamma(\bar{Y}_t^\delta-x)+\gamma(\bar{Y}_t^\delta-x)-\gamma(\bar{Y}_t^\delta-y))\mu_t^\delta(\rmd x)\nu_t^\delta(\rmd y)\Big\rangle
\\ & \leq \kappa(\bar{r}_t^\delta)\bar{r}_t^\delta+\int_{\mathbb{R}^d}\int_{\mathbb{R}^d} 2\|\gamma\|_\infty \mathrm{rc}^\epsilon(|x-y|) \mu_t^\delta(\rmd x)\nu_t^\delta(\rmd y)\eqsp,
\end{align*}
and hence \eqref{eq:one-dimSDE_approx_estimate2} holds.
\end{proof}
We define a one-dimensional process $(r_t^{\delta,\epsilon})_{t\geq 0}$ by
\begin{align} \label{eq:nonlinearonedimSDE_approx}
\rmd r_t^{\delta,\epsilon} =\Big(\bar{b}(r_t^{\delta,\epsilon})+2 \|\gamma\|_\infty \int_{\mathbb{R}_+} \mathrm{rc}^{\epsilon}(u) P_t^{\delta,\epsilon}(\rmd u) \Big)\rmd t+2 \mathrm{rc}^\delta( r_t^{\delta,\epsilon}) \rmd W_t^\delta
\end{align}
with initial condition $r_0^{\delta,\epsilon} =\bar{r}_0^\delta$, $P_t^{\delta,\epsilon}=\Law(r_t^{\delta,\epsilon})$ and $W_t^\delta=\int_0^t (\bar{e}_s^\delta)^T \rmd B_s^1$. This process will allow us to control the distance of $\bar{X}_t^\delta$ and $\bar{Y}_t^\delta$. 

By \cite[Theorem 2.2]{Me96}, under \Cref{ass:decomp_W} and \Cref{ass:init_distr}, $(U_t^{\delta,\epsilon})_{t\geq 0}=(\bar{X}_t^\delta,\bar{Y}_t^\delta,r_t^{\delta,\epsilon})_{t\geq 0}$ exists and is unique, where $(\bar{X}_t^\delta,\bar{Y}_t^\delta)_{t\geq 0}$ solves uniquely \eqref{eq:coupling_non_linearSDE_approx}, $(\bar{r}_t^\delta)_{t\geq 0}$ and $(r_t^{\delta,\epsilon})_{t\geq 0}$ solve uniquely \eqref{eq:one-dimSDE_approx_estimate1} and \eqref{eq:nonlinearonedimSDE_approx}, respectively, with $W_t^\delta=\int_0^t (\bar{e}_s^\delta)^T \rmd B_s^1$.

\begin{lemma} \label{lem:comparision_r_t^delta} 
Assume \Cref{ass:decomp_W} and \Cref{ass:init_distr}.
Then, $|\bar{X}_t^\delta-\bar{Y}_t^\delta|=\bar{r}_t^\delta\leq r_t^{\delta,\epsilon}$, almost surely for all $t$ and $\epsilon<\epsilon_0$.
\end{lemma}

\begin{proof} 
Note that $(\bar{r}_t^\delta)_{t\geq 0}$ and $(r_t^{\delta,\epsilon})_{t\geq 0}$ have the same initial distribution and are driven by the same noise. Since the drift of $(\bar{r}_t^\delta)_{t\geq 0}$ is smaller than the drift of $(r_t^{\delta,\epsilon})_{t\geq 0}$ for $\epsilon<\epsilon_0$, the result follows by \Cref{lemma:modification_watanabe_nonlinear}.

\end{proof}

\begin{proof}[Proof of \Cref{theo:1}]
We consider the nonlinear process $(U_t^{\delta,\epsilon})_{t\geq 0}=(\bar{X}_t^\delta,\bar{Y}_t^\delta,r_t^{\delta,\epsilon})_{t\geq 0}$ on $\mathbb{R}^{2d+1}$ for each $\epsilon,\delta>0$. We denote by $\PP^{\delta,\epsilon}$ the law of $U^{\delta,\epsilon}$ on the space $\mathcal{C}(\mathbb{R}_+,\mathbb{R}^{2d+1})$. We define by $\mathbf{X},\mathbf{Y}:\mathcal{C}(\mathbb{R}_+,\mathbb{R}^{2d+1})\to \mathcal{C}(\mathbb{R}_+,\mathbb{R}^{d})$ and $\mathbf{r}:\mathcal{C}(\mathbb{R}_+,\mathbb{R}^{2d+1})\to \mathcal{C}(\mathbb{R}_+,\mathbb{R})$ the canonical projections onto the first $d$ components, onto the second $d$ components and onto the last component, respectively. 
By \Cref{ass:decomp_W} and \Cref{ass:init_distr} following the same line as the proof of \Cref{lem:existence_stickySDE_step1}, see  \eqref{eq:pthmoment_difference1}, it holds for each $T>0$ 
\begin{align} \label{eq:boundU}
E[|U_{t_2}^{\delta,\epsilon}-U_{t_1}^{\delta,\epsilon}|^4]\leq C|t_2-t_1|^2 \qquad \text{for $t_1,t_2\in[0,T]$}\eqsp,
\end{align}
 for some constant $C$ depending on $T$, $L$, $\|\gamma\|_{\Lip}$, $\|\gamma\|_\infty$ and on the fourth moment of $\mu_0$ and $\nu_0$.
As in \Cref{lem:existence_stickySDE_step1} the law $\PP_T^{\delta,\epsilon}$ of  $(U_t^{\delta,\epsilon})_{0\leq t\leq T}$ on $\mathcal{C}([0,T],\mathbb{R}^{2d+1})$ is tight for each $T>0$ by \cite[Corollary 14.9]{Ka02} and for each $\epsilon>0$ there exists a subsequence $\delta_n\to 0$ such that $(\PP^{\delta_n,\epsilon}_T)_{n\in\mathbb{N}}$ on $\mathcal{C}([0,T],\mathbb{R}^{2d+1})$ converge to a measure $\PP^\epsilon_T$ on $\mathcal{C}([0,T],\mathbb{R}^{2d+1})$.
By a diagonalization argument and since $\{\PP^\epsilon_T: T\geq 0\}$ is a consistent family, cf. \cite[Theorem 5.16]{Ka02}, there exists a probability measure $\PP^\epsilon$ on $\mathcal{C}(\mathbb{R}_+,\mathbb{R}^{2d+1})$ such that for all $\epsilon$ there exists a subsequence $\delta_n$ such that $(\PP^{\delta_n,\epsilon})_{n\in\mathbb{N}}$ converges along this subsequence to $\PP^\epsilon$. 
As in the proof of \Cref{lem:existence_stickySDE_step2} we repeat this argument for the family of measures $(\PP^\epsilon)_{\epsilon>0}$. Hence, there exists a subsequence $\epsilon_m\to 0$ such that $(\PP^{\epsilon_m})_{m\in\mathbb{N}}$ converges to a measure $\PP$.
Let $(\bar{X}_t,\bar{Y}_t,r_t)_{t\geq 0}$ be some process on $\mathbb{R}^{2d+1}$ with distribution $\PP$ on $(\bar{\Omega},\bar{\mathcal{F}}, \bar{P})$.

Since $(\bar{X}_t^\delta)_{t\geq 0}$ and $(\bar{Y}_t^\delta)_{t\geq 0}$  are solutions of \eqref{eq:nonlinearSDE} which are unique in law, we have that for any $\epsilon, \delta>0$, $\PP^{\delta,\epsilon}\circ \mathbf{X}^{-1}=\PP\circ\mathbf{X}^{-1}$ and $\PP^{\delta,\epsilon}\circ \mathbf{Y}^{-1}=\PP\circ\mathbf{Y}^{-1}$.
And therefore $(\bar{X}_t)_{t\geq 0}$ and $(\bar{Y}_t)_{t\geq 0}$ are solutions of \eqref{eq:nonlinearSDE} as well with the same initial condition. 
Hence $\PP\circ(\mathbf{X},\mathbf{Y})^{-1}$ is a coupling of two copies of \eqref{eq:nonlinearSDE}.

Similarly to the proof of \Cref{lem:existence_stickySDE_step1} and \Cref{lem:existence_stickySDE_step2} 
there exist an extended probability space and a one-dimensional Brownian motion
$(W_t)_{t\geq 0}$ such that $(r_t,W_t)_{t\geq 0}$ is a solution to
\begin{align*}
\rmd r_t=(\bar{b}(r_t)+2\|\gamma\|_\infty \PP(r_t >0))\rmd t+2 \1_{(0,\infty)}(r_t) \rmd W_t\eqsp. 
\end{align*} 

In addition, the statement of \Cref{lem:comparision_r_t^delta} carries over to the limiting process $(r_t)_{t\geq 0}$, i.e., $|\bar{X}_t-\bar{Y}_t|\leq r_t$ for all $t\geq 0$, 
since by the weak convergence along the subsequences $(\delta_n)_{n\in\mathbb{N}}$ and $(\epsilon_m)_{m\in\mathbb{N}}$ and the Portmanteau theorem,
$P(|\bar{X}_t-\bar{Y}_t|\leq r_t)\geq\limsup_{m\to \infty}\limsup_{n\to \infty} P(|\bar{X}_t^{\delta_n}-\bar{Y}_t^{\delta_n}|\leq r_t^{\delta_n,\epsilon_m})=1$.

\end{proof}

\subsection{Proof of \texorpdfstring{\Cref{sec:non-linear_sticky_diff}}{}} \label{sec:proofs_nonlinearSDE}

First, we introduce a family of nonlinear SDE whose drift and diffusion coefficient are Lipschitz continuous approximations of the drift and diffusion coefficient of \eqref{eq:two_onedim_stickydiff}. \Cref{thm:existence_comparison} is shown by proving a comparison result for nonlinear SDEs, taking in two steps the limit of the approximations and identifying the limit with the solution of \eqref{eq:two_onedim_stickydiff}. 
Then, \Cref{thm:one_dim_SDE_statdistr_general} and \Cref{thm:one-dim_stickydiff} are shown where we make use of the careful construction of the function $f$.

\subsubsection{Proof of \texorpdfstring{\Cref{thm:existence_comparison}}{}} \label{sec:proof_nonlinear_existence}

We show \Cref{thm:existence_comparison} 
via a family of stochastic differential equations, indexed by $n,m\in\mathbb{N}$, with Lipschitz continuous coefficients,
\begin{equation}\label{eq:1d_stickydiff_coupling}
\begin{aligned} 
\rmd r_t^{n,m}&=(\tilde{b}(r_t^{n,m})+P_t^{n,m}(g^m))\rmd t+2 \theta^n(r_t^{n,m}) \rmd W_t  
\\ \rmd s_t^{n,m}&=(\hat{b}(s_t^{n,m})+\hat{P_t}^{n,m}(h^m))\rmd t+2\theta^n(s_t^{n,m}) \rmd W_t\eqsp, &&\text{Law}(r_0^{n,m},s_0^{n,m})=\eta_{n,m}\eqsp,
\end{aligned}
\end{equation} 
where $P_t^{n,m}=\Law (r_t^{n,m})$, $\hat{P}^{n,m}_t=\Law (s_t^{n,m})$, $P_t^{n,m}(g^m)=\int_{\mathbb{R}_+} g^m(x)P_t^{n,m}(\rmd x)$ and $\hat{P}^{n,m}_t(h^m)=\int_{\mathbb{R}_+} h^m(x)\hat{P}^{n,m}_t(\rmd x)$ for some measurable functions $(g^m)_{m\in\mathbb{N}}$ and $(h^m)_{m\in\mathbb{N}}$, and where $\eta_{n,m}\in\Gamma(\mu_{n,m}, \nu_{n,m})$ for $\mu_{n,m}, \nu_{n,m}\in\mathcal{P}(\mathbb{R}_+)$.
We identify the weak limit for $n\to \infty$ as solution of a family of stochastic differential equations, indexed by $m\in\mathbb{N}$, given by 
\begin{equation}\label{eq:1d_stickydiff_coupling2}
\begin{aligned} 
\rmd r_t^{m}&=(\tilde{b}(r_t^{m})+P_t^{m}(g^m))\rmd t+2 \1_{(0,\infty)}(r_t^{m}) \rmd W_t  
\\ \rmd s_t^{m}&=(\hat{b}(s_t^{m})+\hat{P_t}^{m}(h^m))\rmd t+2\1_{(0,\infty)}(s_t^{m}) \rmd W_t\eqsp, &&\text{Law}(r_0^{m},s_0^{m})=\eta_m \eqsp.
\end{aligned}
\end{equation}
with $P_t^{m}=\Law (r_t^{m})$ and  $\hat{P}^{m}_t=\Law (s_t^{m})$, and where $\eta_{m}\in\Gamma(\mu_{m}, \nu_{m})$ for $\mu_{m}, \nu_{m}\in\mathcal{P}(\mathbb{R}_+)$.
Taking the limit $m\to\infty$, we show in the next step that the solution of \eqref{eq:1d_stickydiff_coupling2} converges to a solution of \eqref{eq:two_onedim_stickydiff}.

We assume for $(g^m)_{m\in\mathbb{N}}$, $(h^m)_{m\in\mathbb{N}}$, $(\theta^n)_{n\in\mathbb{N}}$ and the initial distributions:  
\begin{assumptionH} \label{H5}
$(g^m)_{m\in\mathbb{N}}$ and $(h^m)_{m\in\mathbb{N}}$ are sequences of non-decreasing non-negative uniformly bounded Lipschitz continuous functions such that for all $r\geq 0$, $g^m(r)\leq g^{m+1}(r)$ and $h^m(r)\leq h^{m+1}(r)$ and 
$\lim_{m \to \plusinfty} g^m(r) = g(r)$ and $\lim_{m \to \plusinfty} h^m(r) = h(r)$ where $g$, $h$ are left-continuous non-negative non-decreasing bounded functions. 
In addition, there exists $K_m<\infty$ for any $m$ such that for all $r,s\in\mathbb{R}$ 
\begin{align*}
|g^m(r)-g^m(s)|\leq K_m |r-s| \qquad \text{and} \qquad |h^m(r)-h^m(s)|\leq K_m |r-s|\eqsp.
\end{align*}

\end{assumptionH}
\begin{assumptionH} \label{H4} 
 $(\theta^n)_{n\in\mathbb{N}}$ is a sequence of Lipschitz continuous functions from $\mathbb{R}_+$ to $[0,1]$ with $\theta^n(0)=0$, $\theta^n(r)=1$ for all $r\geq 1/n$ and $\theta^n(r)>0$ for all $r> 0$.
\end{assumptionH}
\begin{assumptionH} \label{H3} $(\mu_{n,m})_{m,n\in\mathbb{N}}$, $(\nu_{n,m})_{m,n\in\mathbb{N}}$, $(\mu_{m})_{m\in\mathbb{N}}$, $(\nu_{m})_{m\in\mathbb{N}}$ are families of probability distributions on $\mathbb{R}_+$ and $(\eta_{n,m})_{n,m\in\mathbb{N}}$,  $(\eta_{m})_{m\in\mathbb{N}}$ families of probability distributions on $\mathbb{R}_+^2$ such that for any $n,m\in\mathbb{N}$ $\eta_{n,m}\in\Gamma(\mu_{n,m},\nu_{n,m})$ and $\eta_{m}\in\Gamma(\mu_{m},\nu_{m})$ and for any $m\in\mathbb{N}$, $(\eta_{n,m})_{n\in\mathbb{N}}$  converges weakly to $\eta_m$ and $(\eta_m)_{m\in\mathbb{N}}$ converges weakly to $\eta$.
Further, the $p$-th order moments of $(\mu_{n,m})_{n,m\in\mathbb{N}}$, $(\nu_{n,m})_{n,m\in\mathbb{N}}$, $(\mu_{m})_{m\in\mathbb{N}}$ and $(\nu_{m})_{m\in\mathbb{N}}$ are uniformly bounded for $p>2$ given in \Cref{H2}.
\end{assumptionH}
Note that by \Cref{H5} for any non-decreasing sequence $(u_m)_{m\in\mathbb{N}}$, which converges to $u\in\mathbb{R}_+$, $g^m(u_m)$ and $h^m(u_m)$ converge to $g(u)$ and $h(u)$, respectively. More precisely, it holds for for all $m\in\mathbb{N}$, $g^m(u_m)- g(u)\leq 0$ and for $m\geq n$, $g^m(u_m)\geq g^m(u_n)$ and therefore, $\lim_{m\to\infty} g^m(u_n)-g(u)\geq \lim_{n\to\infty}\lim_{m\to\infty} =\lim_{n\to\infty}g(u_n)-g(u)=0$ by left-continuity of $g$. Hence, $\lim_{m\to\infty}g^m(u_m)-g(u)=0$ and analogously  $\lim_{m\to\infty}h^m(u_m)-h(u)=0$.
By \Cref{H5}, $\Gamma=\max(\|h\|_\infty,\|g\|_\infty)$ is a uniform upper bound of $(g^m)_{m\in\mathbb{N}}$ and $(h^m)_{m\in\mathbb{N}}$.  

Consider a probability space $(\Omega_0, \mathcal{A}_0, Q)$ and a one-dimensional Brownian motion $(W_t)_{t\geq 0}$. Under \Cref{H5}, \Cref{H4} and \Cref{H3}, for all $m,n\in\mathbb{N}$, there exists random variables $r^{n,m}, s^{n,m}:\Omega_0\to \mathbb{W}$ for each $n,m$ such that $(r^{n,m}_t,s^{n,m}_t)_{t\geq 0}$ is a unique strong solution to \eqref{eq:1d_stickydiff_coupling} associated to $(W_t)_{t\geq 0}$ by \cite[Theorem 2.2]{Me96}. We denote by $\PP^{n,m}=Q\circ (r^{n,m},s^{n,m})^{-1}$ the corresponding distribution on $\mathbb{W}\times \mathbb{W}$.

Before studying the two limits $n,m\to\infty$ and proving \Cref{thm:existence_comparison}, we state a modification of the comparison theorem by Ikeda and Watanabe to compare two solutions of \eqref{eq:1d_stickydiff_coupling}, cf. \cite[Section VI, Theorem 1.1]{IkWa89}.

\begin{lemma}\label{lemma:modification_watanabe_nonlinear}
Let $(r_t^{n,m},s_t^{n,m})_{t\geq 0}$ be a solution of \eqref{eq:1d_stickydiff_coupling} for fixed $n,m\in \mathbb{N}$. Assume \Cref{H1b}, \Cref{H5} and \Cref{H4}. If $Q[r_0^{n,m}\leq s_0^{n,m}]=1$, $\tilde{b}(r)\leq \hat{b}(r)$ and $g^m(r)\leq h^m(r)$ for any $r\in\mathbb{R}_+$, then
\begin{align} \label{eq:comparison}
Q[r_t^{n,m}\leq s_t^{n,m} \text{ for all } t\geq 0]=1\eqsp. 
\end{align}
\end{lemma}

\begin{proof}
For simplicity, we drop the dependence on $n,m$ in $(r_t^{n,m})$ and $(s_t^{n,m})$. Denote by $\rho$ the Lipschitz constant of $\theta^n$. 
Let $(a_k)_{k\in\mathbb{N}}$ be a decreasing sequence, $1>a_1>a_2>\ldots>a_k>\ldots >0$, such that $\int_{a_1}^1\rho^{-2}x^{-1}\rmd x=1$, $\int_{a_2}^{a_1}\rho^{-2}x^{-1}\rmd x=2$,$\ldots$, $\int_{a_k}^{a_{k-1}}\rho^{-2}x^{-1}\rmd x=k$.
We choose a sequence $\Psi_k(u)$, $k=1,2,\ldots$, of continuous functions such that its support is contained in $(a_k,a_{k-1})$, $\int_{a_k}^{a_{k-1}}\Psi_k(u)\rmd u=1$ and $0\leq \Psi_k(u)\leq 2/k\cdot\rho^{-2} u^{-2}$. Such a function exists. We set 
\begin{align*}
\varphi_k(x)=\begin{cases} \int_0^x \rmd y \int_0^y \Psi_k(u)\rmd u & \text{ if } x\geq 0, \\ 0\eqsp & \text{ if } x<0\eqsp. \end{cases}
\end{align*} 
Note that for any $k\in\mathbb{N}$, $\varphi_k\in\mathcal{C}^2(\mathbb{R}_+)$, $|\varphi'_k(x)|\leq 1$, $\varphi_k(x)\to x_+$ as $k\uparrow\infty$ and $\varphi'_k(x)\uparrow \1_{(0,\infty)}(x)$. Applying It\=o's formula to $\varphi_k(r_t-s_t)$, we obtain
\begin{align*}
\varphi_k(r_t-s_t)=\varphi_k(r_0-s_0)+I_1(k)+I_2(k)+I_3(k)\eqsp,
\end{align*}
where
\begin{align*}
I_1(k)&=\int_0^t\varphi_k'(r_u-s_u)[\theta^n(r_u)-\theta^n(s_u)]\rmd B_u\eqsp,
\\ I_2(k)&=\int_0^t \varphi_k'(r_u-s_u)[\tilde{b}(r_u)-\hat{b}(s_u)+P_u(g^m)-\hat{P}_u(h^m)] \rmd u \eqsp,
\\ I_3(k)&=\frac{1}{2}\int_0^t \varphi_k''(r_u-s_u)[\theta^n(r_u)-\theta^n(s_u)]^2\rmd u \eqsp,
\end{align*}
with $P_u=Q\circ r_u^{-1}$ and $\hat{P}_u=Q\circ s_u^{-1}$.
It holds by boundedness and Lipschitz continuity of $\theta^n$
\begin{align*}
\mathbb{E}[I_1(k)]=0\eqsp, \quad \text{ and }\quad \mathbb{E}[I_3(k)]\leq \frac{1}{2}\mathbb{E}\Big[\int_0^t \varphi_k''(r_u-s_u)\rho^2|r_u-s_u|^2\rmd u\Big]\leq \frac{t}{k} \eqsp.
\end{align*}
We note that by \Cref{H5} $\mathbb{E}[(g^m(r_u)-h^m(s_u))\1_{r_u-s_u<0}]\leq 0$ and
\begin{align} \label{eq:comparison_lemma_nonlineardriftpart}
\mathbb{E}[(g^m(r_u)-h^m(s_u))\1_{r_u-s_u\geq 0}]&\leq \mathbb{E}[(g^m(r_u)-g^m(s_u)+g^m(s_u)-h^m(s_u))\1_{r_u-s_u\geq 0}] \nonumber
\\ & \leq \mathbb{E}[(g^m(r_u)-g^m(s_u))\1_{r_u-s_u\geq 0}] \nonumber
\\ & \leq K_m\mathbb{E}[|r_u-s_u|\1_{r_u-s_u\geq 0}]
\end{align}
by Lipschitz continuity of $g^m$, by $g^m(r)\leq h^m(r)$ and since $g^m$ and $h^m$ are non-decreasing. Hence for $I_2$, we obtain
\begin{align*}
I_2(k&)= \int_0^t \varphi_k'(r_u-s_u)[\tilde{b}(r_u)-\hat{b}(r_u)+\hat{b}(r_u)-\hat{b}(s_u)]\rmd u
\\ & +\int_0^t \varphi_k'(r_u-s_u)\Big(\mathbb{E}[(g^m(r_u)-h^m(s_u))\1_{r_u-s_u\geq 0}]+\mathbb{E}[(g^m(r_u)-h^m(s_u))\1_{r_u-s_u<0}]\Big) \rmd u
\\ & \leq \int_0^t \varphi_k'(r_u-s_u)\tilde{L}|r_u-s_u|\rmd u+ \int_0^t \varphi_k'(r_u-s_u) K_m\mathbb{E}[|r_u-s_u|\1_{r_u-s_u\geq 0}] \rmd u\eqsp.
\end{align*} 
Taking the limit $k\to \infty$ and using that $\mathbb{E}[r_0-s_0]=0$, we obtain
\begin{align} \label{eq:modificationwatanabe_nonlinear}
\mathbb{E}[(r_t-s_t)_+]\leq \tilde{L}\mathbb{E}\Big[\int_0^t(r_u-s_u)_+ \rmd u\Big]+K_m\mathbb{E}\Big[\int_0^t  \1_{(0,\infty)}(r_u-s_u)\mathbb{E}[(r_u-s_u)_+] \rmd u\Big] \eqsp,
\end{align}
 by the monotone convergence theorem and since $(\varphi_k')_{k\in\mathbb{N}}$ is a monotone increasing sequence which converges pointwise to $\1_{(0,\infty)}(x)$.
Assume there exists $t^*=\inf\{t\geq 0 : \mathbb{E}[(r_t-s_t)_+]>0\}<\infty$. Then, $\int_0^{t^*}\mathbb{E}[(r_u-s_u)_+ ]\rmd u>0$ or $\int_0^{t^*}\mathbb{E}[ \1_{(0,\infty)}(r_u-s_u)]\mathbb{E}[(r_u-s_u)_+] \rmd u>0$. By definition of $t^*$, $\mathbb{E}[(r_u-s_u)_+ ]=0$ for all $u<t^*$ and hence both terms are zero.  This contradicts the definition of $t^*$. Hence, \eqref{eq:comparison} holds. 
\end{proof}

Next, we show that the distribution of the solution of \eqref{eq:1d_stickydiff_coupling} converges as $n\to\infty$.

\begin{lemma} \label{lem:existence_stickySDE_step1}
Assume that $\tilde{b}$, $\hat{b}$, $g$ and $h$ satisfy \Cref{H1b} and \Cref{H1g}.
Let $\eta\in\Gamma(\mu,\nu)$ where the probability measures $\mu$ and $\nu$ on $\mathbb{R}_+$ satisfy \Cref{H2}.
Assume that $(g^m)_{m\in\mathbb{N}}$, $(h^m)_{m\in\mathbb{N}}$, $(\theta^n)_{n\in\mathbb{N}}$, $(\mu_{n,m})_{m,n\in\mathbb{N}}$, $(\nu_{n,m})_{m,n\in\mathbb{N}}$ and $(\eta_{n,m})_{m,n\in\mathbb{N}}$ satisfy condition \Cref{H5}, \Cref{H3} and \Cref{H4}. Then for any $m \in \nset$, there exists a random variable $(r^m,s^m)$ defined on some probability space $(\Omega^m,\mathcal{A}^m,P^m)$ with values in $\mathbb{W}\times\mathbb{W}$, such that $(r_t^m,s_t^m)_{t\geq 0}$ is a weak solution of the stochastic differential equation \eqref{eq:1d_stickydiff_coupling2}.
More precisely, for all $m\in\mathbb{N}$ the sequence of laws $Q\circ (r^{n,m},s^{n,m})^{-1}$ converges weakly to the distribution $P^m\circ (r^m,s^m)^{-1}$. If additionally,
\begin{align*} 
& \tilde{b}(r)\leq \hat{b}(r) \quad \text{and} \quad g^m(r)\leq h^m(r) \eqsp,  &&\text{for any } r\in\mathbb{R}_+ \text{ and }
\\ & Q[r_0^{n,m}\leq s_0^{n,m}]=1 &&\text{for any } n,m\in\mathbb{N},
\end{align*}
then $P^m[r_t^{m}\leq s_t^{m} \text{ for all } t\geq 0]=1$.
\end{lemma}

\begin{proof}
Fix $m\in\mathbb{N}$. The proof is divided in three parts. First we show tightness of the sequences of probability measures. Then we identify the limit of the sequence of stochastic processes. Finally, we compare the two limiting processes. \\
\textbf{Tightness:}
We show that  the sequence of probability measures  $(\mathbb{P}^{n,m})_{n\in\mathbb{N}}$ on $(\mathbb{W}\times \mathbb{W},\mathcal{B}(\mathbb{W})\otimes \mathcal{B}(\mathbb{W}))$ is tight by applying Kolmogorov's continuity theorem. 
Consider $p>2$ such that the $p$-th moment in \Cref{H2} and \Cref{H3} are uniformly bounded.
Fix $T>0$. Then the $p$-th moment of $r_t^{n,m}$ for $t<T$ can be bounded using It\=o's formula,
\begin{align*}
\rmd |r_t^{n,m}|^p &\leq p|r_t^{n,m}|^{p-2}\langle r_t^{n,m},(\tilde{b}(r_t^{n,m})+P_t^{n,m}(g^m))\rangle \rmd t + 2\theta^n(r_t^{n,m}) p|r_t^{n,m}|^{p-2} r_t^{n,m}\rmd W_t
\\ & +p(p-1)|r_t^{n,m}|^{p-2}2 \theta^n(r_t^{n})^2\rmd t
\\ & \leq p\Big(|r_t^{n,m}|^p \tilde{L}+\Gamma |r_t^{n,m}|^{p-1}+2(p-1)|r_t^{n,m}|^{p-2}\Big) \rmd t  + 2\theta^n(r_t^{n,m}) p(r_t^{n,m})^{p-1}\rmd W_t
\\ & \leq p\Big(\tilde{L}+\Gamma+2(p-1)\Big)|r_t^{n,m}|^p\rmd t + p(\Gamma+2(p-1))\rmd t + 2\theta^n(r_t^{n,m}) p(r_t^{n,m})^{p-1}\rmd W_t\eqsp,
\end{align*}
where $\Gamma=\max(\|g\|_\infty,\|h\|_{\infty})$.
Taking expectation yields
\begin{align*}
\frac{\rmd}{\rmd t}\mathbb{E}[ |r_t^{n,m}|^p]\leq p\Big(\tilde{L}+\Gamma+2(p-1)\Big)\mathbb{E}|r_t^{n,m}|^p+p(\Gamma+2(p-1))\eqsp.
\end{align*}
Then by Gronwall's lemma
\begin{align} \label{eq:pthmomentbound}
\sup_{t\in[0,T]}\mathbb{E}[|r_t^{n,m}|^p]\leq \rme^{p(\tilde{L}+\Gamma+2(p-1))T} (\mathbb{E}[|r_0^{n,m}|^p]+Tp(\Gamma+2(p-1)))<C_p<\infty \eqsp,
\end{align}
where $C_p$ depends on $T$ and the $p$-th moment of the initial distribution, which is finite by \Cref{H4}.
Similarly, it holds $\sup_{t\in[0,T]} \mathbb{E}[|s_t^{n,m}|^p]<C_p$ for $t\leq T$.
Using this moment bound, it holds for all $t_1,t_2\in[0,T]$ by \Cref{H1b}, \Cref{H5} and \Cref{H4},
\begin{align*}
\mathbb{E}[|r_{t_2}^{n,m}&-r_{t_1}^{n,m}|^p]\leq C_1(p)\Big(\mathbb{E}[|\int_{t_1}^{t_2} \tilde{b}(r_u^{n,m}) + P^{n,m}_u(g^m)\rmd u |^p]+\mathbb{E}[|\int_{t_1}^{t_2}2 \theta^n(r_u^{n,m})\rmd W_u|^p]\Big)
\\ & \leq C_2(p)\Big(\Big(\mathbb{E}\Big[\frac{\tilde{L}^p}{|t_2-t_1|}\int_{t_1}^{t_2}|r_u^{n,m}|^p\rmd u\Big]+\Gamma^p\Big)|t_2-t_1|^p+\mathbb{E}[|\int_{t_1}^{t_2}2\theta^n(r_u^{n,m})\rmd u|^{p/2}]\Big)
\\ & \leq C_2(p)\Big(\Big(\frac{\tilde{L}^p}{|t_2-t_1|}\int_{t_1}^{t_2}\mathbb{E}[|r_u^{n,m}|^p]\rmd u +\Gamma^p\Big)|t_2-t_1|^p+2^{p/2}|t_2-t_1|^{p/2}\Big)
\\ & \leq C_3(p,T,\tilde{L},\Gamma,C_p)|t_2-t_1|^{p/2}\eqsp,
\end{align*}
where $C_i(\cdot)$ are constants depending on the stated argument and which are independent of $n,m$. Note that in the second step, we used Burkholder-Davis-Gundy inequality, see \cite[Chapter IV, Theorem 48]{Pr04}.
It holds similarly, $\mathbb{E}[|s_{t_2}^{n,m}-s_{t_1}^{n,m}|^p]\leq C_3(p,T,\tilde{L},\Gamma,C_p)|t_2-t_1|^{p/2}$. Hence, 
\begin{align} \label{eq:pthmoment_difference1}
\mathbb{E}[|(r_{t_2}^{n,m},s_{t_2}^{n,m})-(r_{t_1}^{n,m},s_{t_1}^{n,m})|^p]\leq C_4(p,T,\tilde{L},\Gamma,C_p)|t_2-t_1|^{p/2}
\end{align}
for all $t_1,t_2\in[0,T]$. 
Hence, by Kolmogorov's continuity criterion, cf. \cite[Corollary 14.9]{Ka02},  
there exists a constant $\tilde{C}$ depending on $p$ and $\gamma$ such that
\begin{align} \label{eq:pthmoment_difference}
\mathbb{E}\Big[ [(r^{n,m},s^{n,m})]_\gamma^p\Big]\leq \tilde{C}\cdot C_4(p,T,\tilde{L},\Gamma,C_p) \eqsp,
\end{align}
where $[\cdot]_\gamma^p$ is given by $[x]_\gamma=\sup_{t_1,t_2\in[0,T]}\frac{|x(t_1)-x(t_2)|}{|t_1-t_2|^\gamma}$
and $(r_t^{n,m},s_t^{n,m})_{n\in\mathbb{N},t\geq 0}$ is tight in $\mathcal{C}([0,T],\mathbb{R}^2)$.
Hence, for each $T>0$ there exists a subsequence $n_k\to \infty$ and a probability measure $\mathbb{P}^m_T$ on $\mathcal{C}([0,T],\mathbb{R}^2)$. Since $\{\PP^m_T\}_T$ is a consistent family, there exists by \cite[Theorem 5.16]{Ka02} a probability measure $\PP^m$ on
$(\mathbb{W}\times\mathbb{W}, \mathcal{B}(\mathbb{W})\otimes\mathcal{B}(\mathbb{W}))$ such that there is a subsequence $(n_k)_{k\in\mathbb{N}}$ such that $\PP^{n_k,m}$ converges along this subsequence to $\PP^m$. Note that here we can take by a diagonalization argument the same subsequence $(n_k)_{k\in\mathbb{N}}$ for all $m$.

\textbf{Characterization of the limit measure:} In the following we drop for simplicity the index $k$ in the subsequence. 
Denote by $(\mathbf{r}_t,\mathbf{s}_t) (\omega)=\omega(t)$ the canonical process on $\mathbb{W}\times\mathbb{W}$.
Since $\PP^{n,m}\circ (\mathbf{r}_0,\mathbf{s}_0)^{-1}=\eta_{n,m}$ converges weakly to $\eta_m$ by \Cref{H3}, it holds $\PP^m\circ (\mathbf{r}_0,\mathbf{s}_0)^{-1}=\eta_m$. We define the maps $M^{n,m},N^{n,m}:\mathbb{W}\times\mathbb{W}\to \mathbb{W}$ by 
\begin{align*} 
M_t^{n,m}=\mathbf{r}_t-\mathbf{r}_0-\int_0^t (\tilde{b}(\mathbf{r}_u)+P_u^n(g^m))\rmd u \text{ and } N_t^{n,m}=\mathbf{s}_t-\mathbf{s}_0-\int_0^t (\hat{b}(\mathbf{s}_u)+\hat{P}_u^n(h^m))\rmd u \eqsp, 
\end{align*}
where $P_u^n=\PP^{n,m}\circ(\mathbf{r}_u)^{-1}$ and $\hat{P}_u^n=\PP^{n,m}\circ(\mathbf{s}_u)^{-1}$.
For each $m,n\in\mathbb{N}$, $(M_t^{n,m},\mathcal{F}_t,\PP^{n,m})$ and $(N_t^{n,m},\mathcal{F}_t,\PP^{n,m})$ are martingales with respect to the canonical filtration $\mathcal{F}_t=\sigma((\mathbf{r}_u,\mathbf{s}_u)_{0\leq u\leq t})$ by It\=o's formula and the moment estimate \eqref{eq:pthmomentbound}. Further the family $(M_t^{n,m},\PP^{n,m})_{n\in\mathbb{N},t\geq 0}$ and $(N_t^{n,m},\PP^{n,m})_{n\in\mathbb{N},t\geq 0}$ are uniformly integrable by Lipschitz continuity of $\tilde{b}$ and $\hat{b}$ and by boundedness of $g^m$ and $h^m$. 
Further, the mappings $M^{n,m}$ and $N^{n,m}$ are continuous in $\mathbb{W}$. We show that $\PP^{n,m}\circ(\mathbf{r},\mathbf{s},M^{n,m},N^{n,m})^{-1}$ converges weakly to $\PP^m\circ(\mathbf{r},\mathbf{s},M^m,N^m)^{-1}$ as $n\to\infty$, where
\begin{align} \label{eq:martingale}
M_t^{m}=\mathbf{r}_t-\mathbf{r}_0-\int_0^t (\tilde{b}(\mathbf{r}_u)+P_u(g^m))\rmd u \quad \text{and} \quad N_t^{m}=\mathbf{s}_t-\mathbf{s}_0-\int_0^t (\hat{b}(\mathbf{s}_u)+\hat{P}_u(h^m))\rmd u \eqsp, 
\end{align}
with $P_u=\PP^{m}\circ\mathbf{r}_u^{-1}$ and $\hat{P}_u=\PP^{m}\circ\mathbf{s}_u^{-1}$.
To show weak convergence to $\PP^m\circ(\mathbf{r},\mathbf{s},M^m,N^m)^{-1}$, we note that $(M^m,N^m)$ is continuous in $\mathbb{W}$ and we consider for a Lipschitz continuous and bounded function $G:\mathbb{W}\to\mathbb{R}$,
\begin{align*}
&\abs{\int_\mathbb{W}G(\omega)\rmd \PP^{n,m}\circ(M^{n,m})^{-1}(\omega)-\int_\mathbb{W}G(\omega)\rmd \PP^{m}\circ(M^{m})^{-1}(\omega)}
\\ &\leq \abs{\int_\mathbb{W}G(\omega)\rmd \PP^{n,m}\circ(M^{n,m})^{-1}(\omega)-\int_\mathbb{W}G(\omega)\rmd \PP^{n,m}\circ(M^{m})^{-1}(\omega)}
\\ & +\abs{\int_\mathbb{W}G(\omega)\rmd \PP^{n,m}\circ(M^{m})^{-1}(\omega)-\int_\mathbb{W}G(\omega)\rmd \PP^{m}\circ(M^{m})^{-1}(\omega)}\eqsp.
\end{align*}
The second term converges to $0$ as $n\to\infty$, since $(M^m)$ is continuous. For the first term it holds 
\begin{align*}
&\abs{\int_\mathbb{W}G(\omega)\rmd \PP^{n,m}\circ(M^{n,m})^{-1}(\omega)-\int_\mathbb{W}G(\omega)\rmd \PP^{n,m}\circ(M^{m})^{-1}(\omega)}
\\ & =\abs{\int_\mathbb{W}(G\circ M^{n,m})(\omega)\rmd \PP^{n,m}(\omega)-\int_\mathbb{W}(G\circ M^{m})(\omega)\rmd \PP^{n,m}(\omega)}
\\ & \leq \|G\|_{\Lip}\sup_{\omega\in\mathbb{W}} d_\mathbb{W}(M^{n,m}(\omega),M^m(\omega)) \eqsp, 
\end{align*}
where $d_\mathbb{W}(f,g)=\sum_{k=1}^\infty\sup_{t\in[0,k]}2^{-k}|f(t)-g(t)|$.
This term converges to $0$ for $n\to\infty$, since for all $T>0$ and $\omega\in\mathbb{W}$, for $n\to\infty$
\begin{align*}
\sup_{t\in [0,T]}|M^{n,m}_t(\omega)-M^{m}_t(\omega)|\leq \int_0^T\abs{(\mathbb{P}^{n,m}\circ \mathbf{r}_s^{-1})(g^m)-(\mathbb{P}^{m}\circ \mathbf{r}_s^{-1})(g^m)}\rmd s\to 0 \eqsp,
\end{align*}
by Lebesgue dominated convergence theorem, since $g$ is bounded.
Hence,
\begin{align*}
&\abs{\int_\mathbb{W}G(\omega)\rmd \PP^{n,m}\circ(M^{n,m})^{-1}(\omega)-\int_\mathbb{W}G(\omega)\rmd \PP^{m}\circ(M^{m})^{-1}(\omega)}\to 0 \quad \text{for } n\to\infty,
\end{align*} 
and similarly for  $(N^{n,m})$, and therefore by the Portmanteau theorem \cite[Theorem 13.16]{klenke2013probability}, weak convergence of $\PP^{n,m}\circ(\mathbf{r},\mathbf{s},M^{n,m},N^{n,m})^{-1}$ to $\PP^m\circ(\mathbf{r},\mathbf{s},M^m,N^m)^{-1}$ holds.

Let $G:\mathbb{W}\to\mathbb{R}_+$ be a $\mathcal{F}_s$-measurable, bounded, non-negative function. By uniformly integrability of $(M_t^{n,m},\PP^{n,m})_{n\in\mathbb{N}, t\geq 0}$, for any $s\leq t$,
\begin{equation}\label{eq:martingale_property}
\begin{aligned} 
\mathbb{E}^m[G(M_t^m-M_s^m)]&=\mathbb{E}^m[G( \int_s^t( \tilde{b}(\mathbf{r}_u)+P_u(g^m))\rmd u)]
\\ & =\lim_{n\to\infty}\mathbb{E}^{n,m}[G( \int_s^t (\tilde{b}(\mathbf{r}_u)+P_u^{n}(g^m))\rmd u)]
\\ & =\lim_{n\to\infty}\mathbb{E}^{n,m}[G(M_t^{n,m}-M_s^{n,m})]=0 \eqsp,
\end{aligned}
\end{equation}
and analogously for $(N_t^{n,m})_{t\geq 0}$ and hence, 
 $(M_t^m,\mathcal{F}_t,\PP^m)$ and $(N_t^m,\mathcal{F}_t,\PP^m)$ are continuous martingales. 
The quadratic variation $([(M^m,N^m)]_t)$ exists $\PP^m$-almost surely. To complete the identification of the limit, it suffices to note that the quadratic variation is given by
\begin{equation} \label{eq:quadraticvariation}
\begin{aligned}
&[M^m]=4\int_0^\cdot \1_{(0,\infty)}(\mathbf{r}_u)\rmd u && \PP^m\text{-almost surely,}
\\  & [N^m]=4\int_0^\cdot \1_{(0,\infty)}(\mathbf{s}_u)\rmd u && \PP^m\text{-almost surely, and}
\\  & [M^m,N^m]=4\int_0^\cdot \1_{(0,\infty)}(\mathbf{r}_u)\1_{(0,\infty)}(\mathbf{s}_u)\rmd u && \PP^m\text{-almost surely,}
\end{aligned}
\end{equation}
which holds following the computations in the proof of \cite[Theorem 22]{EbZi19}. We show that $((M_t^m)^2-4\int_0^t\1_{(0,\infty)}\mathbf{r}_u\rmd u)$ is a sub- and a supermartingale and hence a martingale using a monotone class argument by noting first that for any bounded continuous and non-negative function $G:\mathbb{W}\to\mathbb{R}_+$, 
\begin{align} \label{eq:expGM_t^m}
\mathbb{E}^m[G(M_t^m)^2]=\lim_{n\to\infty}\mathbb{E}^{n,m}[G(M_t^{n,m})^2]
\end{align}
holds using uniform integrability of $((M_t^{n,m})^2,\PP^{n,m})_{n\in\mathbb{N},t\geq 0}$ which holds similarly as above. Note that 
\begin{align} \label{eq:submart1}
\mathbb{E}^m\parentheseDeux{G\int_s^t \1_{(0,\infty)}(\mathbf{r}_u)\rmd u}&\leq \lim_{\epsilon\downarrow 0}\liminf_{n\to\infty}\mathbb{E}^{n,m}\parentheseDeux{G\int_s^t \1_{(\epsilon,\infty)}(\mathbf{r}_u)\rmd u}
\end{align} 
holds by lower semicontinuity of $\omega\to\int_0^{\cdot}\1_{(\epsilon,\infty)}(\omega_s)\rmd s$ for each $\epsilon>0$, Fatou's lemma and the Portmanteau theorem.
For any fixed $\epsilon>0$,
\begin{align} \label{eq:submart2}
\liminf_{n\to\infty}\mathbb{E}^{n,m}\parentheseDeux{G\parenthese{\int_s^t\theta^n(\mathbf{r}_u)^2\rmd u-\int_s^t \1_{(\epsilon,\infty)}(\mathbf{r}_u)\rmd u}}.
\end{align} 
Then by \eqref{eq:expGM_t^m}, \eqref{eq:submart1} and \eqref{eq:submart2}
\begin{align*}
\mathbb{E}^m&\parentheseDeux{G\parenthese{(M_t^m)^2-(M_s^m)^2-4\int_s^t \1_{(0,\infty)}(\mathbf{r}_u)\rmd u}}
\\&\geq \lim_{\epsilon\downarrow 0}\liminf_{n\to\infty}\mathbb{E}^{n,m}\parentheseDeux{G\parenthese{(M_t^{n,m})^2-(M_s^{n,m})^2-4\int_s^t\theta^n(\mathbf{r}_u)^2\rmd u}}=0
\end{align*}
and by a monotone class argument, cf. \cite[Chapter 1, Theorem 8]{Pr04}, $((M_t^m)^2-4\int_0^t\1_{(0,\infty)}(\mathbf{r}_u)\rmd u,\PP^m)$ is a submartingale. To show that it is also a supermartingale we note that $((M_t^m)^2-4t,\PP^m)$ is a supermartingale by \eqref{eq:expGM_t^m}. By the uniqueness of the Doob-Meyer decomposition,  cf. \cite[Chapter 3, Theorem 8]{Pr04}, $t\to[M^m]_t-4t$ is $\PP^m$-almost surely decreasing. Note further, that $(\mathbf{r}_t,\mathcal{F}_t,\PP^m)$ is a continuous semimartingale with $[\mathbf{r}]=[M^m]$. Then by It\=o-Tanaka formula, cf. \cite[Chapter 6, Theorem 1.1]{ReYo99}, 
\begin{align*}
\int_0^t\1_{\{0\}}(\mathbf{r}_u)\rmd [M^m]_u=\int_0^t\1_{\{0\}}(\mathbf{r}_u)\rmd [\mathbf{r}]_u=\int_0^t\1_{\{0\}}(y)\ell_t^y(\mathbf{r})\rmd y=0 \eqsp,
\end{align*}
where $\ell_t^y(\mathbf{r})$ is the local time of $\mathbf{r}$ in $y$. Therefore, for any $0\leq s< t$,
\begin{align*}
[M^m]_t-[M^m]_s=\int_0^t\1_{(0,\infty)}(\mathbf{r}_u)\rmd [M^m]_u\leq4\int_0^t\1_{(0,\infty)}(\mathbf{r}_u)\rmd u
\end{align*}
and hence, for any $\mathcal{F}_s$-measurable, bounded, non-negative function $G:\mathbb{W}\to\mathbb{R}_+$,
\begin{align*}
\mathbb{E}^m\parentheseDeux{G((M_t^m)^2-(M_s^m)^2-4\int_s^t \1_{(0,\infty)}(\mathbf{r}_u)\rmd u)} \leq 0 \eqsp.
\end{align*}
As before, by a monotone class argument, $((M_t^m)^2-4\int_0^t\1_{(0,\infty)}(\mathbf{r}_u)\rmd u,\PP^m)$ is a supermartingale, and hence a martingale.

Hence, we obtain the quadratic variation $[M^m]_t$ given in \eqref{eq:quadraticvariation}. The other characterizations in \eqref{eq:quadraticvariation} follow by analogous arguments. 
Then by a martingale representation theorem, see \cite[Chapter II, Theorem 7.1]{IkWa89}, we conclude, that there are a probability space $(\Omega^m,\mathcal{A}^m,P^m)$ and a Brownian motion motion $W$ and random variables $(r^m,s^m)$ on this space such that $P^m\circ(r^m,s^m)^{-1}=\PP^m\circ(\mathbf{r}^m,\mathbf{s}^m)^{-1}$ and such that $(r^m,s^m,W)$ is a weak solution of \eqref{eq:1d_stickydiff_coupling2}.
Finally, note that we have weak convergence of $Q\circ (r^{n,m},s^{n,m})^{-1}$ to $P^m\circ(r^m,s^m)^{-1}$ not only along a subsequence since the characterization of the limit holds for any subsequence $(n_k)_{k\in\mathbb{N}}$.

\textbf{Comparison of two solutions:} To show $P^m[r_t^{m}\leq s_t^m \text{ for all }t\geq 0]=1$ we note that by \Cref{lemma:modification_watanabe_nonlinear}, $Q[r_t^{n}\leq s_t^{n} \text{ for all } t\geq 0]=1$. The monotonicity carries over to the limit by the Portmanteau theorem for closed sets, since we have weak convergence of $\PP^{n,m}\circ(\mathbf{r},\mathbf{s})^{-1}$ to $\PP^m\circ(\mathbf{r},\mathbf{s})^{-1}$.
\end{proof}

We show in the next step that the distribution of the solution of \eqref{eq:1d_stickydiff_coupling2} converges as $m\to\infty$. 
For each $m\in\mathbb{N}$ let $(\Omega^m,\mathcal{A}^m,P^m)$ be a probability space and random variables $r^m,s^m:\Omega^m\to\mathbb{W}$ such that $(r^m_t,s^m_t)_{t\geq 0}$ is a solution of \eqref{eq:1d_stickydiff_coupling2}. Let $\PP^m=P^m\circ(r^m,s^m)^{-1}$ denote the law on $\mathbb{W}\times\mathbb{W}$. 

\begin{lemma} \label{lem:existence_stickySDE_step2}
Assume that $(\tilde{b},g)$ and $(\hat{b},h)$ satisfy \Cref{H1b} and \Cref{H1g}. 
Let $\eta\in\Gamma(\mu,\nu)$ where the probability measures $\mu$ and $\nu$ on $\mathbb{R}_+$ satisfy \Cref{H2}.
Assume that $(g^m)_{m\in\mathbb{N}}$, $(h^m)_{m\in\mathbb{N}}$, $(\mu_m)_{m\in\mathbb{N}}$, $(\nu_m)_{m\in\mathbb{N}}$ and $(\eta_m)_{m\in\mathbb{N}}$ satisfy conditions \Cref{H5} and \Cref{H3}. 
Then there exists a random variable $(r,s)$ defined on some probability space $(\Omega,\mathcal{A},P)$ with values in $\mathbb{W}\times\mathbb{W}$, such that $(r_t,s_t)_{t\geq 0}$ is a weak solution of the sticky stochastic differential equation \eqref{eq:two_onedim_stickydiff}. Furthermore,  
the sequence of laws $P^{m}\circ (r^{m},s^{m})^{-1}$ converges weakly to the law $P\circ (r,s)^{-1}$. If additionally,
\begin{align*}
& \tilde{b}(r)\leq \hat{b}(r) \eqsp, \qquad g(r)\leq h(r) \quad \text{and} \qquad g^m(r)\leq h^m(r) && \text{ for any } r\in\mathbb{R}_+ \text{, and }
\\ &
 P^m[r_0^m\leq s_0^m]=1 &&\text{for any } m\in\mathbb{N}
\end{align*}
then $P[r_t\leq s_t \text{ for all } t\geq 0]=1$.
\end{lemma}

\begin{proof}
The proof is structured as the proof of \Cref{lem:existence_stickySDE_step1}. First analogously to the proof of \eqref{eq:pthmomentbound} we show under \Cref{H1b}, \Cref{H5} and \Cref{H3}, 
\begin{align} \label{eq:pthmomentbound2}
\sup_{t\in[0,T]}\mathbb{E}[|r_t^{m}|^p]<\infty \eqsp.
\end{align}
Tightness of the sequence of probability measures $(\PP^{m})_{m\in\mathbb{N}}$ on $(\mathbb{W}\times\mathbb{W},\mathcal{B}(\mathbb{W})\otimes\mathcal{B}(\mathbb{W}))$ holds adapting the steps of the proof of \Cref{lem:existence_stickySDE_step1} to \eqref{eq:1d_stickydiff_coupling2}. Note that \eqref{eq:pthmomentbound} and \eqref{eq:pthmoment_difference1} hold analogously for $(r_t^m,s_t^m)_{m\in\mathbb{N}}$ by \Cref{H1b}, \Cref{H5} and \Cref{H3}. Hence by Kolmogorov's continuity criterion, cf. \cite[Corollary 14.9]{Ka02}, we can deduce that there exists a probability measure $\PP$ on $(\mathbb{W}\times\mathbb{W},\mathcal{B}(\mathbb{W})\otimes\mathcal{B}(\mathbb{W}))$ such that there is a subsequence $(m_k)_{k\in\mathbb{N}}$ along which $\PP^{m_k}$ converge towards $\PP$.
To characterize the limit, we first note that by Skorokhod representation theorem, cf. \cite[Chapter 1, Theorem 6.7]{Bi99}, without loss of generality we can assume that $(r^m,s^m)$ are defined on a common probability space $(\Omega,\mathcal{A},P)$ with expectation $E$ and converge almost surely to $(r,s)$ with distribution $\PP$.  
By \Cref{H5}, $P_t^m(g^m)=E[g^m(r_t^m)]$ and the monotone convergence theorem, $P_t^{m}(g^m)$ converges to $P_t(g)$ for any $t\geq0$.
Then, by Lebesgue convergence theorem it holds almost surely for all $t\geq 0$
\begin{align} \label{eq:limit_drift}
\lim_{m\to\infty}\int_0^t\Big(\tilde{b}(r_t^m)+P^m_u(g^m)\Big)\rmd u=\int_0^t\Big(\tilde{b}(r_t)+P_u(g)\Big)\rmd u \eqsp,
\end{align}
where $P^m_u=P\circ (r_u^m)^{-1}$ and $P_u=P\circ (r_u)^{-1}$. A similar statement holds for $(s_t)_{t\geq 0}$.

Consider the mappings ${M}^m,{N}^m:\mathbb{W}\times\mathbb{W}\to \mathbb{W}$ given by \eqref{eq:martingale}.
Then for all $m\in\mathbb{N}$, $({M}_t^m,\mathcal{F}_t,\PP^m)$ and $({N}_t^m,\mathcal{F}_t,\PP^m)$ are martingales with respect to the canonical filtration $\mathcal{F}_t=\sigma((\mathbf{r}_u,\mathbf{s}_u)_{0\leq u\leq t})$. Further the family $({M}_t^m,\PP^m)_{m\in\mathbb{N},t\geq 0}$ and $({N}_t^m,\PP^m)_{m\in\mathbb{N},t\geq 0}$ are uniformly integrable by \eqref{eq:pthmomentbound2}. In the same line as in the proof of \Cref{lem:existence_stickySDE_step1} and by \eqref{eq:limit_drift}, $\PP^m\circ(\mathbf{r},\mathbf{s},{M}^m,{N}^m)$ converges weakly to $\PP\circ(\mathbf{r},\mathbf{s},{M},{N})$ where
\begin{align*}
{M}_t=\mathbf{r}_t-\mathbf{r}_0-\int_0^t(\tilde{b}(\mathbf{r}_u)+P_u(g))\rmd u \qquad \text{and} \qquad {N}_t=\mathbf{s}_t-\mathbf{s}_0-\int_0^t(\hat{b}(\mathbf{s}_u)+
\hat{P}_u(h))\rmd u \eqsp.
\end{align*} 
Let $G:\mathbb{W}\to\mathbb{R}_+$ be a $\mathcal{F}_s$-measurable bounded, non-negative function. By uniform integrability, for any $s\leq t$,
\begin{align*}
\mathbb{E}[G(M_t-M_s)]&=\mathbb{E}[G( \int_s^t( \tilde{b}(\mathbf{r}_u)+P_u(g))\rmd u)]=\lim_{m\to\infty}\mathbb{E}^m[G( \int_s^t (\tilde{b}(\mathbf{r}_u)+P_u(g^m))\rmd u)]
\\ & =\lim_{m\to\infty}\mathbb{E}^m[G(M_t^m-M_s^m)]=0 \eqsp,
\end{align*}
and analogously for $(N_t)_{t\geq 0}$. Hence, $({M}_t,\mathcal{F}_t,\PP)$ and $({N}_t,\mathcal{F}_t,\PP)$ are martingales. 
Further, the quadratic variation $([({M},{N})]_t)$ exists $\PP$-almost surely and is given by \eqref{eq:quadraticvariation} $\PP$-almost surely, which holds following the computations in the proof of \Cref{lem:existence_stickySDE_step1}. As in \Cref{lem:existence_stickySDE_step1}, we conclude by a martingale representation theorem that there are a probability space $(\Omega,\mathcal{A},P)$ and a Brownian motion $W$ and random variables $(r,s)$ on this space such that $P\circ(r,s)^{-1}=\PP\circ(\mathbf{r},\mathbf{s})^{-1}$ and such that $(r,s,W)$ is a weak solution of \eqref{eq:two_onedim_stickydiff}. Note that the limit identification holds for all subsequences $(m_k)_{k\in\mathbb{N}}$ and hence $P^m\circ(r^m,s^m)^{-1}$ converges weakly to $P\circ(r,s)^{-1}$ for $m\to\infty$.
The monotonicity $P^m[r_t^m\leq s_t^m \text{ for all } t\geq 0]=1$ carries over to the limit by Portmanteau theorem, since $\PP^m\circ(\mathbf{r},\mathbf{s})^{-1}$ converges weakly to $\PP\circ(\mathbf{r},\mathbf{s})^{-1}$. 
\end{proof}

\begin{proof}[Proof of \Cref{thm:existence_comparison}] The proof  is a direct consequence of \Cref{lem:existence_stickySDE_step1} and \Cref{lem:existence_stickySDE_step2}.
\end{proof}

\subsubsection{Proof of \texorpdfstring{\Cref{thm:one_dim_SDE_statdistr_general}}{}} \label{sec:proof_nonlinear_invmeas}

\begin{proof}[Proof of \Cref{thm:one_dim_SDE_statdistr_general}] 
Note that the Dirac at $0$, $\delta_0$, is by definition an invariant measure of $(r_t)_{t\geq 0}$ solving \eqref{eq:one-dim_stickydiff}.
Assume that the process starts from an invariant probability measure $\pi$, hence $\PP(r_t >0)= p=\pi((0,\infty))$ for any $t\geq 0$.  
Note that for $p=0$ the drift vanishes. If the initial measure is the Dirac measure in $0$, $\delta_0$, then the diffusion coefficient disappears. Hence, $\Law(r_t)=\delta_0$ for any $t\geq 0$. It remains to investigate the case $p\neq 0$. Here, we are in the regime of \cite[Lemma 24]{EbZi19} where an invariant measure is of the form \eqref{eq:definition_stationarydistr}. Since $p =\PP(r_t >0)$, the invariant measure $\pi$ satisfies additionally the necessary condition
\begin{align} \label{eq:necessarycondition_stationarydistr}
p=\pi((0,\infty))=\frac{I(a,p)}{2/(ap)+I(a,p)}
\end{align}
with $I(a,p)$ given in \eqref{eq:definition_I(a,p)}. For $p\neq 0$, this expression is equivalent to \eqref{eq:defintion_p}. 
\end{proof}

\begin{proof}[Proof of \Cref{thm:one_dim_SDE_stationarydistr}]
By \Cref{thm:one_dim_SDE_statdistr_general}, it suffices to study the solutions of \eqref{eq:defintion_p}. 
By \eqref{eq:definition_I(a,p)} and since $\tilde{b}(r)=-\tilde{L}r$, it holds for $\hat{I}(a,p)=(1-p)I(a,p)$,
\begin{align} \label{eq:I(a,p)_representation}
\hat{I}(a,p)=\Big(\sqrt{\frac{\pi}{2}}+\int_0^{\frac{ap}{\sqrt{2\tilde{L}}}}\exp(-x^2/2)\rmd x\Big)\sqrt{\frac{2}{\tilde{L}}}\exp\Big(\frac{a^2p^2}{4\tilde{L}}\Big)(1-p) \eqsp.
\end{align}
In the case $a/\sqrt{\tilde{L}}\leq 2/\sqrt{\pi}$, $\hat{I}(a,0)=\sqrt{\pi/\tilde{L}}$ by \eqref{eq:I(a,p)_representation}. 
Further, by $1+x\leq \rme^x$ and $a/\sqrt{\tilde{L}}\leq 2/\sqrt{\pi}$,
\begin{align*}
\Big(\sqrt{\frac{\pi}{2}}+\int_0^{\frac{ap}{\sqrt{2\tilde{L}}}}\rme^{-\frac{x^2}{2}}\rmd x\Big)(1-p)\rme^{\frac{a^2p^2}{4\tilde{L}}}
 & \leq \sqrt{\frac{\pi}{2}}\Big(1+\sqrt{\frac{2}{\pi}}\int_0^{\frac{ap}{\sqrt{2L}}}\rme^{-\frac{x^2}{2}}\rmd x\Big)\rme^{-p}\rme^{\frac{p^2}{\pi}}
\\ & \leq \sqrt{\frac{\pi}{2}}\Big(1+\frac{2p}{\pi}\Big)\rme^{-p}\rme^{\frac{p^2}{\pi}}
\leq \sqrt{\frac{\pi}{2}}\rme^{p(\frac{3}{\pi}-1)}< \sqrt{\frac{\pi}{2}}
\end{align*}
for $p\in(0,1]$. Hence, $\hat{I}(a,p)<\hat{I}(a,0)$ by \eqref{eq:I(a,p)_representation}.
Therefore, $\hat{I}(a,p)<\hat{I}(a,0)\leq\frac{2}{a}$ for all $p\in(0,1]$ and so $\delta_0$ is the unique invariant probability measure for $a/\sqrt{\tilde{L}}\leq 2/\sqrt{\pi}$. 

To show that for $a/{\sqrt{\tilde{L}}}> 2/{\sqrt{\pi}}$, there exists a unique $p$ solving \eqref{eq:defintion_p}, 
we note that $\hat{I}(a,p)$ is continuous with $\hat{I}(a,0)>2/a$ and $\hat{I}(a,1)=0$. By the mean value theorem, there exists at least one $p\in(0,1)$ satisfying \eqref{eq:defintion_p}. In the following we drop the dependence on $a$ in $I(a,p)$ and $\hat{I}(a,p)$. We show uniqueness of the solution $p$ by contradiction. Assume that $p_1<p_2$ are the two smallest solutions of \eqref{eq:defintion_p}. Hence, it holds either $\hat{I}'(p_1)<0$ or $\hat{I}'(p)=0$ for $p_1$. Note that the derivative is given by
\begin{align}
\hat{I}'(p_i)&=-I(p_i)+(1-p_i)I'(p_i)
 =-I(p_i)+(1-p_i)\Big(p_i\frac{a^2}{2\tilde{L}}I(p_i)+\frac{a}{\tilde{L}}\Big) \nonumber
\\ & =-\frac{2}{a(1-p_i)}+(1-p_i)\frac{a}{\tilde{L}}\Big(\frac{p_i}{1-p_i}+1\Big)=-\frac{2}{a(1-p_i)}+\frac{a}{\tilde{L}} \eqsp. \label{eq:I_derivative}
\end{align}
Then, for $p_2>p_1$, it holds
\begin{align*}
\hat{I}'(p_2)=-\frac{2}{a(1-p_2)}+\frac{a}{\tilde{L}}<-\frac{2}{a(1-p_1)}+\frac{a}{\tilde{L}}=\hat{I}'(p_1)\leq 0 \eqsp.
\end{align*}
If $\hat{I}'(p_1)<0$, it holds $\hat{I}'(p_2)<0$ which contradicts that $p_1$ and $p_2$ are the two smallest solutions. In the second case, when $\hat{I}'(p_1)=0$, we note that the second derivative of $\hat{I}(p)$ at $p_1$ is given by 
\begin{align*}
\hat{I}''(p_1)&=-2I'(p_1)+(1-p_1)I''(p_1)
\\& = \Big(-2+(1-p_1)\frac{a^2p_1}{2\tilde{L}}\Big)\Big(I(p_1)\frac{a^2p_1}{2\tilde{L}}+\frac{a}{\tilde{L}}\Big)+(1-p_1)I(p_1)\frac{a^2}{2\tilde{L}}
\\& =\Big(-2+(1-p_1)\frac{a^2p_1}{2\tilde{L}}\Big)\frac{a}{\tilde{L}(1-p_1)}+\frac{a}{\tilde{L}}
=-\frac{a}{\tilde{L}(1-p_1)}<0 \eqsp.
\end{align*}
Hence, in this case there is a maximum at $p_1$, which contradicts that $p_1$ is the smallest solution.
Thus, there exists a unique solution $p_1$ of \eqref{eq:defintion_p} for $a/\sqrt{\tilde{L}}>2/\sqrt{\pi}$.

\end{proof}

\subsubsection{Proof of \texorpdfstring{\Cref{thm:one-dim_stickydiff}}{}} \label{sec:proof_nonlinear_convergence}
 
\begin{proof}[Proof of \Cref{thm:one-dim_stickydiff}]

To show \eqref{eq:stickydiff_estimateinthm} we extend the function $f$ to a concave function on $\mathbb{R}$ by setting $f(x)=x$ for $x<0$. Note that $f$ is continuously differentiable and $f'$ is absolutely continuous and bounded. 
Using It\=o-Tanaka formula, c.f. \cite[Chapter 6, Theorem 1.1]{ReYo99} we obtain 
\begin{align*}
\rmd f(r_t)= f'(r_t)(\tilde{b}(r_t)+a\PP(r_t >0))\rmd t+2 f''(r_t)\1_{(0,\infty)}(r_t)\rmd t+\rmd M_t \eqsp,
\end{align*}
where $M_t=2 \int_0^t f'(r_s) 1_{(0,\infty)}(r_s) \rmd B_s$ is a martingale. Taking expectation, we get 
\begin{align*}
\frac{d}{dt}\mathbb{E}[f(r_t)]&=\mathbb{E}[ f'(r_t)(\tilde{b}(r_t)+a\PP(r_t >0))]+2\mathbb{E}[f''(r_t)\1_{(0,\infty)}(r_t)]
\\ & =\mathbb{E}[f'(r_t)\tilde{b}(r_t)+2(f''(r_t)-f''(0))]+\mathbb{E}[af'(r_t)+2f''(0)]\PP(r_t >0)
\\ & \leq -c\mathbb{E}[f(r_t)] \eqsp,  
\end{align*}
where the last step holds by \eqref{eq:condition_f_1} and \eqref{eq:condition_f_2}. By applying Gronwall's lemma, we obtain \eqref{eq:stickydiff_estimateinthm}.

\end{proof}

\subsection{Proof of \texorpdfstring{\Cref{sec:propachaos}}{}} \label{sec:proof_thm_propag}

The proof of \Cref{thm:timepropagation_of_chaos} works in the same line as the proof of \Cref{thm:nonlinearSDE} and \Cref{theo:1}. Additionally, the difference between the nonlinear SDE and the mean-field system is bounded in \Cref{lemma:bound_A_t^idelta}  for which a uniform in time bound for the second moment of the process $(\bar{X}_t)_{t\geq 0}$ solving \eqref{eq:nonlinearSDE} is needed and which is given first.

\begin{lemma} \label{lem:moment_bound}
Let $(\bar{X}_t)_{t\geq 0}$ be a solution of \eqref{eq:nonlinearSDE} with $\mathbb{E}[|\bar{X}_0|^2]<\infty$. Assume \Cref{ass:decomp_W}. Then there exists $C\in(0,\infty)$ depending on $d$, $W$ and the second moment of $\bar{X}_0$ such that
\begin{align} \label{eq:moment_bound}
C=\sup_{t\geq 0 }\mathbb{E}[|\bar{X}_t|^2]<\infty\eqsp.
\end{align}
\end{lemma}

The proof relies on standard techniques (see e.g., \cite[Lemma 8]{DuEbGuZi20}) and is added for completeness. 

\begin{proof} [Proof of \Cref{lem:moment_bound}] By It\=o's formula, it holds
\begin{align*}
\frac{1}{2}\rmd  |\bar{X}_t|^2=\langle \bar{X}_t,b*\barmu_t (\bar{X}_t)\rangle \rmd t+\bar{X}_t^T \rmd B_t+\frac{1}{2}d \ \rmd t\eqsp.
\end{align*}
Taking expectation and using symmetry, we get
\begin{align*}
\frac{\rmd }{ \rmd t}\mathbb{E}[|\bar{X}_t|^2]&=  \mathbb{E}[\langle\bar{X}_t-\tilde{X}_t,b(\bar{X}_t-\tilde{X}_t\rangle]+d
\\ & = -\mathbb{E}[\langle \bar{X}_t-\tilde{X}_t,L(\bar{X}_t-\tilde{X}_t)-\gamma(\bar{X}_t-\tilde{X}_t)\rangle \1_{|\bar{X}_t-\tilde{X}_t|>R_0}]
\\ & \indent -\mathbb{E}[\langle \bar{X}_t-\tilde{X}_t,L(\bar{X}_t-\tilde{X}_t)- \gamma(\bar{X}_t-\tilde{X}_t)\rangle \1_{|\bar{X}_t-\tilde{X}_t|\leq R_0}]+d
\\ & \leq  \mathbb{E}[|\bar{X}_t|^2(-2L+ \kappa(|\bar{X}_t-\tilde{X}_t|) \1_{|\bar{X}_t-\tilde{X}_t|>R_0})]+\|\gamma\|_\infty R_0+d\eqsp. 
\end{align*}
Hence by definition \eqref{eq:definition_R0} of $R_0$ and by Gronwall's lemma we obtain the result \eqref{eq:moment_bound}.
\end{proof}

Let $N\in\mathbb{N}$. We construct a sticky coupling of $N$ \iid~realizations of solutions $(\{\bar{X}_t^i\}_{i=1}^N)_{t\geq 0}$ to \eqref{eq:nonlinearSDE} and of the solution $(\{Y_t^i\}_{i=1}^N)_{t\geq 0}$ to the mean field particle system \eqref{eq:meanfield}. Then, we consider a weak limit  for $\delta\to 0$ of Markovian couplings which are constructed similar as in \Cref{sec:main_contraction_results}. Let $\mathrm{rc}^{\delta}$, $\mathrm{sc}^{\delta}$ satisfy \eqref{eq:condition_rc_sc} and \eqref{eq:condition_rc_sc2}. The coupling $(\{\bar{X}_t^{i,\delta},{Y}^{i,\delta}\}_{i=1}^N)_{t\geq 0}$ is defined as process in $\mathbb{R}^{2Nd}$ satisfying a system of SDEs given by
\begin{equation}\label{eq:propofchaos_coupling_approx}
\begin{aligned} 
\rmd \bar{X}_t^{i,\delta}&=b*\barmu^\delta_t(\bar{X}_t^{i,\delta}) \rmd t
 +\mathrm{rc}^{\delta}(\tilde{r}_t^{i,\delta})\rmd B_t^{i,1}  +\mathrm{sc}^{\delta}(\tilde{r}_t^{i,\delta})\rmd B_t^{i,2}  
\\ \rmd {Y}_t^{i,\delta}&=\frac{1}{N}\sum_{j=1}^N b(Y_t^{i,\delta}-Y_t^{j,\delta})\rmd t 
+\mathrm{rc}^{\delta}(\tilde{r}_t^{i,\delta})(\Id-2\tilde{e}_t^{i,\delta}(\tilde{e}_t^{i,\delta})^T) \rmd B_t^{i,1} +\mathrm{sc}^{\delta}(\tilde{r}_t^{i,\delta})\rmd B_t^{i,2}\eqsp,
\end{aligned}
\end{equation}
where $\Law(\{\bar{X}_0^{i,\delta},Y_0^{i,0}\}_{i=1}^N)=\bar{\mu}_0^{\otimes N}\otimes \nu_0^{\otimes N}$, and where $(\{B_t^{i,1}\}_{i=1}^N)_{t\geq 0},(\{B_t^{i,2}\}_{i=1}^N)_{t\geq 0}$ are \iid~$d$-dimensional standard Brownian motions.
 We set $\tilde{X}_t^{i,\delta}=\bar{X}_t^{i,\delta}-\frac{1}{N}\sum_{j=1}^N\bar{X}_t^{j,\delta}$, $\tilde{Y}_t^{i,\delta}=Y_t^{i,\delta}-\frac{1}{N}\sum_{j=1}^N Y_t^{j,\delta}$, $\tilde{Z}_t^{i,\delta}=\tilde{X}_t^{i,\delta}-\tilde{Y}_t^{i,\delta}$, $\tilde{r}_t^{i,\delta}=|\tilde{Z}_t^{i,\delta}|$ and $\tilde{e}_t^{i,\delta}=\tilde{Z}_t^{i,\delta}/\tilde{r}_t^{i,\delta}$ for $\tilde{r}_t^{i,\delta}\neq 0$. The value $\tilde{e}_t^{i,\delta}$ for $\tilde{r}_t^{i,\delta}=0$ is irrelevant as $\mathrm{rc}^{i,\delta}(0)=0$. By Levy's characterization $(\{\bar{X}_t^{i,\delta},{Y}^{i,\delta}_t\}_{i=1}^N)_{t\geq 0}$ is indeed a coupling of \eqref{eq:nonlinearSDE} and \eqref{eq:meanfield}. Existence and uniqueness of the coupling given in \eqref{eq:propofchaos_coupling_approx} hold by \cite[Theorem 2.2]{Me96}. In the next step we analyse $\tilde{r}_t^{i,\delta}$.

\begin{lemma} \label{lem:r_t^idelta}
Assume \Cref{ass:decomp_W} holds. Then, for $\epsilon<\epsilon_0$, where $\epsilon_0$ is given in \eqref{eq:condition_rc_sc2}, and for any $i \in\{1,\ldots,N\}$, it holds almost surely,
\begin{align} \label{eq:r_tdelta}
\rmd \tilde{r}_t^{i,\delta}&=-L\tilde{r}_t^{i,\delta} \rmd t+\langle \tilde{e}_t^{i,\delta},\frac{1}{N}\sum_{j=1}^N  \gamma(\tilde{X}_t^{i,\delta}-\tilde{X}_t^{j,\delta})-\gamma(\tilde{Y}_t^{i,\delta}-\tilde{Y}_t^{j,\delta})\rangle \rmd t  \nonumber
\\ & + 2\sqrt{1+\frac{1}{N}} \mathrm{rc}^{\delta}(\tilde{r}_t^{i,\delta})\rmd W_t^{i,\delta}  + \Big\langle \tilde{e}_t^{i,\delta},\Theta_t^{i,\delta}+\frac{1}{N}\sum_{k=1}^N\Theta_t^{k,\delta}\Big\rangle \rmd t
\\ & \leq \Big( \bar{b}(\tilde{r}_t^{i,\delta}) 
 + 2\|\gamma\|_\infty \frac{1}{N}\sum_{j=1}^N\mathrm{rc}^\epsilon(\tilde{r}_t^{j,\delta})\Big) \rmd t  + 2 \sqrt{1+\frac{1}{N}} \mathrm{rc}^{\delta}(\tilde{r}_t^{i,\delta})\rmd W_t^{i,\delta} \nonumber
 \\ & +\Big( A_t^{i,\delta} + \frac{1}{N}\sum_{k=1}^NA_t^{k,\delta}\Big)\rmd t\eqsp. \nonumber
\end{align}
with $\Theta_t^{i,\delta}=b*\barmu_t^\delta(\bar{X}_t^{i,\delta})-\frac{1}{N}\sum_{j=1}^N b (\bar{X}_t^{i,\delta}-\bar{X}_t^{j,\delta})$
and 
\begin{align} \label{eq:A_t^i}
A_t^{i,\delta}=\Big|\Theta_t^{i,\delta}\Big|=\Big|b*\barmu_t^\delta(\bar{X}_t^{i,\delta})- \frac{1}{N}\sum_{j=1}^N b (\bar{X}_t^{i,\delta}-\bar{X}_t^{j,\delta})\Big|
\end{align}
and where $(\{W_t^{i,\delta}\}_{i=1}^N)_{t\geq 0}$ are $N$ one-dimensional Brownian motions given by
\begin{align} \label{eq:def_BM_W_t^i}
W_t^{i,\delta}=\sqrt{\frac{N}{N+1}}\parenthese{\int_0^t (\tilde{e}_s^{i,\delta})^T \rmd B_s^{i,1}+\frac{1}{N}\sum_{j=1}^N\int_0^t (\tilde{e}_s^{j,\delta})^T\rmd B_s^{j,1}}\eqsp, \quad i=1,\ldots, N.
\end{align}
\end{lemma}
\begin{proof}
By \eqref{eq:propofchaos_coupling_approx} and since $\gamma$ is anti-symmetric, it holds by It\=o's formula for any $i \in \{1,\ldots,N\}$,
\begin{align*}
\rmd (\tilde{r}_t^{i,\delta})^2&=-2L(\tilde{r}_t^{i,\delta})^2 \rmd t+2\langle \tilde{Z}_t^{i,\delta},\frac{1}{N}\sum_{j=1}^N \gamma(\tilde{X}_t^{i,\delta}-\tilde{X}_t^{j,\delta})-\gamma(\tilde{Y}_t^{i,\delta}-\tilde{Y}_t^{j,\delta})\rangle \rmd t
\\ & +4 \Big(1+\frac{1}{N}\Big) \mathrm{rc}^{\delta}(\tilde{r}_t^{i,\delta})^2\rmd t+ 4 \sqrt{1+\frac{1}{N}} \mathrm{rc}^{\delta}(\tilde{r}_t^{i,\delta})\langle \tilde{Z}_t^{i,\delta},\tilde{e}_t^{i,\delta}\rangle \rmd W_t^{i,\delta} 
\\ & + 2\langle \tilde{Z}_t^{i,\delta},b*\barmu_t^{\delta} (\bar{X}_t^{i,\delta})-\frac{1}{N}\sum_{j=1}^N b(\bar{X}_t^{i,\delta}-\bar{X}_t^{j,\delta})\rangle \rmd t
\\ & +2 \langle \tilde{Z}_t^{i,\delta},-\frac{1}{N}\sum_{k=1}^N \Big(b*\barmu_t^{\delta} (\bar{X}_t^{k,\delta})-\frac{1}{N}\sum_{j=1}^N b(\bar{X}_t^{k,\delta}-\bar{X}_t^{j,\delta})\Big)\rangle \rmd t\eqsp.
\end{align*}
where $(\{W_t^i\}_{i=1}^N)_{t\geq 0}$ are $N$ \iid one-dimensional Brownian motions given by \eqref{eq:def_BM_W_t^i}.
Note that the prefactor $(N/(N+1))^{1/2}$ ensures that the quadratic variation satisfies $[W^i]_t=t$ for $t\geq 0$, and hence $(\{W_t^i\}_{i=1}^N)_{t\geq 0}$ are Brownian motions. This definition of $(\{W_t^i\}_{i=1}^N)_{t\geq 0}$ leads to $(1+1/N)^{1/2}$ in the diffusion term of the SDE.
Applying the $\mathcal{C}^2$ approximation of the square root used in the proof of \Cref{lemma:SDEr_t^delta} and taking $\varepsilon\to 0$ in the approximation yields the stochastic differential equations of $(\{\tilde{r}_t^{i,\delta}\}_{i=1}^N)_{t\geq 0}$. We obtain its upper bound for $\epsilon<\epsilon_0$ by \Cref{ass:decomp_W} and \eqref{eq:condition_rc_sc2} similarly to the proof of \Cref{lemma:SDEr_t^delta}. 
\end{proof}

Next, we state a bound for \eqref{eq:A_t^i}. The result and the proof are adapted from \cite[Theorem 2]{DuEbGuZi20}. 

\begin{lemma}\label{lemma:bound_A_t^idelta} 
Under the same assumption as in \Cref{lemma:comparison_propofchaos}, 
it holds for any  $i=1,\ldots,N$ 
\begin{align*}
E\Big[|A_t^{i,\delta}|^2\Big]\leq C_1N^{-1} \text{ and } E\Big[A_t^{i,\delta}\Big]\leq C_2 N^{-1/2}\eqsp,
\end{align*}
where $A_t^{i,\delta}$ is given in \eqref{eq:A_t^i} and $C_1$ and $C_2$ are constants depending on $\|\gamma\|_\infty$, $L$ and $C$ given in \Cref{lem:moment_bound}.
\end{lemma}
\begin{proof} By \Cref{ass:init_distr}, it holds $\mathbb{E}[|\bar{X}_0^{i,\delta}|^2]<\infty$ for $i=1,\ldots,N$. 
Note that given $\bar{X}_t^{i,\delta}$, $\bar{X}_t^{j,\delta}$ are \iid with law $\barmu_t^{\delta}$ for all $j\neq i $. Hence,
\begin{align*}
\mathbb{E}[b(\bar{X}_t^{i,\delta}-\bar{X}_t^{j,\delta})|\bar{X}_t^{i,\delta}]=b*\barmu_t^\delta(\bar{X}_t^{i,\delta})\eqsp.
\end{align*} 
Since $\gamma$ is anti-symmetric, $b(0)=0$,  and we have
\begin{align*}
\mathbb{E}&\Big[|b*\barmu_t^\delta(\bar{X}_t^{i,\delta})-\frac{1}{N-1}\sum_{j=1}^Nb(\bar{X}_t^{i,\delta}-\bar{X}_t^{j,\delta})|^2\Big|\bar{X}_t^{i,\delta}\Big]
\\ & = \mathbb{E}\Big[|\frac{1}{N-1}\sum_{j=1}^N\mathbb{E}[b(\bar{X}_t^{i,\delta}-\bar{X}_t^{j,\delta})|\bar{X}_t^{i,\delta}] -\frac{1}{N-1}\sum_{j=1}^Nb(\bar{X}_t^{i,\delta}-\bar{X}_t^{j,\delta})|^2\Big|\bar{X}_t^{i,\delta}\Big]
\\ & =\frac{1}{N-1}\Var_{\barmu_t^\delta}(b(\bar{X}_t^{i,\delta}-\cdot))\eqsp.
\end{align*}
By \eqref{eq:nabla_W}, \Cref{ass:decomp_W}, \Cref{ass:init_distr} and \Cref{lem:moment_bound}, we obtain 
\begin{align*}
\Var_{\barmu_t^\delta}(b(\bar{X}_t^{i,\delta}-\cdot))&=\int_{\mathbb{R}^d}\Big|\Big(-L(\bar{X}_t^{i,\delta}-x)+\int_{\mathbb{R}^d} L(\bar{X}_t^{i,\delta}-\tilde{x})\barmu_t^\delta(\rmd \tilde{x})\Big)
\\ & \indent +\Big(\gamma(\bar{X}_t^{i,\delta}-x)-\int_{\mathbb{R}^d} \gamma(\tilde{X}_t^{i,\delta}-\tilde{x})\barmu_t^\delta(\rmd \tilde{x})\Big)\Big|^2 \barmu_t^\delta(\rmd x)
\\ & =\int_{\mathbb{R}^d}\Big|Lx+\Big(\gamma(\bar{X}_t^{i,\delta}-x)-\int_{\mathbb{R}^d} \gamma(\tilde{X}_t^{i,\delta}-\tilde{x})\barmu_t^\delta(\rmd \tilde{x})\Big)\Big|^2 \barmu_t^\delta(\rmd x)
\\ & \leq 2L^2\int_{\mathbb{R}^d} |x|^2\barmu_t^\delta(\rmd x)+8\|\gamma\|_\infty^2\leq 2L^2C^2+8\|\gamma\|_\infty^2\eqsp.
\end{align*}
By the Cauchy-Schwarz inequality, we have
\begin{align*}
\mathbb{E}[(A_t^{i,\delta})^2]&\leq 2\mathbb{E}\Big[|b*\barmu_t(\bar{X}_t^{i,\delta})-\frac{1}{N-1}\sum_{j=1}^Nb(\bar{X}_t^{i,\delta}-\bar{X}_t^{j,\delta})|^2\Big]
\\ & +2\Big(\frac{1}{N-1}-\frac{1}{N}\Big)^2 \mathbb{E}\Big[|\sum_{j=1}^Nb(\bar{X}_t^{i,\delta}-\bar{X}_t^{j,\delta})|^2\Big]
\\ & \leq 2\frac{1}{N-1}\mathbb{E}[\Var_{\barmu_t^\delta}(b(\bar{X}_t^{i,\delta}-\cdot))]+\frac{1}{N^2(N-1)}\mathbb{E}\Big[\sum_{j=1}^N|b(X_t^{i,\delta}-X_t^{j,\delta})|^2\Big]
\\ & \leq \frac{4 L^2}{N-1}C+\frac{16\|\gamma\|_\infty^2}{N-1} +\frac{1}{N^2}\Big(8CL^2+4\|\gamma\|^2_\infty\Big)
\\ & \leq N^{-1}C_1<\infty\eqsp,
\end{align*}
where $C_1$ depends on $\|\gamma\|_\infty$, $L$ and the second moment bound $C$.
Similarly, it holds
\begin{align*}
\mathbb{E}[A_t^{i,\delta}]&\leq \mathbb{E}\Big[|b*\barmu_t^\delta(\bar{X}_t^{i,\delta})-\frac{1}{N-1}\sum_{j=1}^Nb(\bar{X}_t^{i,\delta}-\bar{X}_t^{j,\delta})|\Big]
\\ & +\Big(\frac{1}{N-1}-\frac{1}{N}\Big) \sum_{j=1}^N \mathbb{E}\Big[|b(\bar{X}_t^{i,\delta}-\bar{X}_t^{j,\delta})|\Big]
\\ & \leq \frac{\sqrt{2} L}{\sqrt{N-1}}C^{1/2}+\frac{\sqrt{8}\|\gamma\|_\infty}{\sqrt{N-1}} +\frac{1}{N}\Big(\sqrt{2}C^{1/2}L+\|\gamma\|_\infty\Big)
\\ & \leq N^{-1/2}C_2<\infty\eqsp,
\end{align*}
where $C_2=2LC^{1/2}+4\|\gamma\|_\infty +(\sqrt{2}C^{1/2}+\|\gamma\|_\infty)$. 
\end{proof}

To control $(\{\tilde{r}_t^{i,\delta}\}_{i=1}^N)_{t\geq 0}$, we consider $(\{r_t^{i,\delta,\epsilon}\}_{i=1}^N)_{t\geq 0}$ given as solution of 
\begin{equation}\label{eq:r_t^iepsdelta_propofchaos}
\begin{aligned} 
\rmd r_t^{i,\delta,\epsilon}&=\bar{b}(r_t^{i,\delta,\epsilon}) \rmd t
+ \frac{1}{N}\sum_{j=1}^N  2\|\gamma\|_\infty \mathrm{rc}^{\epsilon}(r_t^{j,\delta,\epsilon})\rmd t 
 +\Big( A_t^{i,\delta}+\frac{1}{N}\sum_{k=1}^N A_t^{k,\delta}\Big)\rmd t
 \\& + 2\sqrt{1+\frac{1}{N}} \mathrm{rc}^{\delta}(r_t^{i,\delta,\epsilon})\rmd W_t^{i,\delta} 
\end{aligned}
\end{equation} 
with initial condition $r_0^{i,\delta,\epsilon}=\tilde{r}_0^{i,\delta}$ for all $i=1,\ldots,N$, $A_t^{i,\delta}$ given in \eqref{eq:A_t^i} and $W_t^{i,\delta}$ given in \eqref{eq:def_BM_W_t^i}. 

By \cite[Theorem 2.2]{Me96}, under \Cref{ass:decomp_W} and \Cref{ass:init_distr}, $(\{U_t^{i,\delta,\epsilon}\}_{i=1}^N)_{t\geq 0}=(\{\bar{X}_t^{i,\delta},Y_t^{i,\delta},r_t^{i,\delta,\epsilon}\}_{i=1}^N)_{t\geq 0}$ exists and is unique, where $(\{\bar{X}_t^{i,\delta},\bar{Y}_t^{i,\delta}\}_{i=1}^N)_{t\geq 0}$ solves uniquely \eqref{eq:propofchaos_coupling_approx}, $(\{\bar{r}_t^{i\delta}\}_{i=1}^N)_{t\geq 0}$ and $(\{r_t^{i,\delta,\epsilon}\}_{i=1}	N)_{t\geq 0}$ solve uniquely \eqref{eq:r_tdelta} and \eqref{eq:r_t^iepsdelta_propofchaos}, respectively, with $(\{W_t^{i,\delta}\}_{i=1}^N)_{t\geq 0}$ given by \eqref{eq:def_BM_W_t^i}.

\begin{lemma} \label{lemma:comparison_propofchaos}
Assume \Cref{ass:decomp_W} and \Cref{ass:init_distr}.
Then for any $i=1,\ldots,N$, $|\bar{X}_t^{i,\delta}-Y_t^{i,\delta}-\frac{1}{N}\sum_j(\bar{X}_t^{j,\delta}-Y_t^{j,\delta})|=\tilde{r}_t^{i,\delta}\leq r_t^{i,\delta,\epsilon}$, almost surely for all $t\geq 0$ and $\epsilon<\epsilon_0$.
\end{lemma}

\begin{proof} 
Note, that both processes $(\{\tilde{r}_t^{i,\delta}\}_{i=1}^N)_{t \geq 0}$ and $(\{r_t^{i,\delta,\epsilon}\}_{i=1}^N)_{t\geq 0}$ have the same initial condition and are driven by the same noise.
Since the drift for $(\{r_t^{i,\delta,\epsilon}\}_{i=1}^N)_{t\geq 0}$ is larger than the drift for $(\{\tilde{r}_t^{i,\delta}\}_{i=1}^N)_{t\geq 0}$ for $\epsilon<\epsilon_0$ by \eqref{eq:condition_rc_sc2}, we can conclude $\tilde{r}_t^{i,\delta}\leq r_t^{i,\delta,\epsilon}$ almost surely for all $t\geq 0$, $\epsilon<\epsilon_0$ and $i=1,\ldots N$ by \Cref{lemma:modification_watanabe_meanfield}. 
\end{proof}

\begin{proof}[Proof of \Cref{thm:timepropagation_of_chaos}]
Consider the process $(\{U_t^{i,\delta,\epsilon}\}_{i=1}^N)_{t\geq 0}=(\{\bar{X}_t^{i,\delta},Y_t^{i,\delta},r_t^{i,\delta,\epsilon}\}_{i=1}^N)_{t\geq 0}$ on $\mathbb{R}^{N(2d+1)}$ for each $\epsilon,\delta>0$. We denote by $\mathbb{P}^{\delta,\epsilon}$ the law of $\{U^{\delta,\epsilon}\}_{i=1}^N$ on $\mathcal{C}(\mathbb{R}_+,\mathbb{R}^{N(2d+1)})$. We define the canonical projections ${\mathbf{X}},{\mathbf{Y}},\mathbf{r}$ onto the first $N d$, second $N d$ and last $N$ components.

By \Cref{ass:decomp_W} and \Cref{ass:init_distr} it holds in the same line as in the proof of  \Cref{lem:existence_NstickySDE_step1} 
for each $T>0$
\begin{align}
E[|\{U_{t_2}^{i,\delta,\epsilon}-U_{t_1}^{i,\delta,\epsilon}\}_{i=1}^N|^4]\leq C|t_2-t_1|^2 \qquad \text{for $t_1,t_2\in[0,T]$,}
\end{align}
 for some constant $C$ depending on $T$, $L$, $\|\gamma\|_{\Lip}$, $\|\gamma\|_\infty$, $N$ and on the fourth moment of $\mu_0$ and $\nu_0$.
Note that we used here that the additional drift terms $(A_t^{i,\delta})_{t\geq 0}$ occurring in the SDE of $(\{r_t^{i,\delta,\epsilon}\}_{i=1}^N)_{t\geq 0}$ are Lipschitz continuous in $(\{\bar{X}_t^{i,\delta}\}_{i=1}^N)_{t\geq 0}$.
Then as in the proofs of \Cref{lem:existence_NstickySDE_step1} and \Cref{lem:existence_NstickySDE_step2}, $\PP^{\delta,\epsilon}$ is tight and converges weakly along a subsequence to a measure $\PP$ by Kolmogorov's continuity criterion, cf. \cite[Corollary 14.9]{Ka02}.

As in \Cref{lem:existence_NstickySDE_step1} the law $\PP_T^{\delta,\epsilon}$ of  $(\{U_t^{i,\delta,\epsilon}\}_{i=1}^N)_{0\leq t\leq T}$ on $\mathcal{C}([0,T],\mathbb{R}^{N(2d+1)})$ is tight for each $T>0$ by \cite[Corollary 14.9]{Ka02} and for each $\epsilon>0$ there exists a subsequence $\delta_n\to 0$ such that $(\PP^{\delta_n,\epsilon}_T)_{n\in\mathbb{N}}$ on $\mathcal{C}([0,T],\mathbb{R}^{N(2d+1)})$ converge to a measure $\PP^\epsilon_T$ on $\mathcal{C}([0,T],\mathbb{R}^{N(2d+1)})$.
By a diagonalization argument and since $\{\PP^\epsilon_T: T\geq 0\}$ is a consistent family, cf. \cite[Theorem 5.16]{Ka02}, there exists a probability measure $\PP^\epsilon$ on $\mathcal{C}(\mathbb{R}_+,\mathbb{R}^{N(2d+1)})$ such that for all $\epsilon$ there exists a subsequence $\delta_n$ such that $(\PP^{\delta_n,\epsilon})_{n\in\mathbb{N}}$ converges along this subsequence to $\PP^\epsilon$. 
As in the proof of \Cref{lem:existence_NstickySDE_step2} we repeat this argument for the family of measures $(\PP^\epsilon)_{\epsilon>0}$. Hence, there exists a subsequence $\epsilon_m\to 0$ such that $(\PP^{\epsilon_m})_{m\in\mathbb{N}}$ converges to a measure $\PP$.
Let $(\{\bar{X}_t^i,{Y}_t^i,r_t^i\}_{i=1}^N)_{t\geq 0}$ be some process on $\mathbb{R}^{N(2d+1)}$ with distribution $\PP$ on $(\bar{\Omega},\bar{\mathcal{F}}, \bar{P})$.

Since $(\{\bar{X}_t^{i,\delta}\}_{i=1}^N)_{t\geq 0}$ and $(\{{Y}_t^{i,\delta}\}_{i=1}^N)_{t\geq 0}$ are solutions that are unique in law, we have that for any $\delta,\epsilon>0$, $\PP^{\delta,\epsilon}\circ\mathbf{X}^{-1}=\PP\circ\mathbf{X}^{-1}$ and $\PP^{\delta,\epsilon}\circ\mathbf{Y}^{-1}=\PP\circ\mathbf{Y}^{-1}$. Hence, $\PP\circ({\mathbf{X}},{\mathbf{Y}})^{-1}$ is a coupling of \eqref{eq:nonlinearSDE} and \eqref{eq:meanfield}.

Similarly to the proof of \Cref{lem:existence_NstickySDE_step1} and \Cref{lem:existence_NstickySDE_step2} there exist an extended underlying probability space and $N$ \iid one-dimensional Brownian motion $(\{W_t^i\}_{i=1}^N)_{t\geq 0}$ such that $(\{r_t^i,W_t^i\}_{i=1}^N)_{t\geq 0}$ is a solution of 
\begin{align*}
\rmd r_t^i&= \bar{b}(r_t^{i}) \rmd t
+ \frac{1}{N}\sum_{j=1}^N 2\|\gamma\|_\infty \1_{(0,\infty)}(r_t^j)\rmd t +\Big(A_t^i+\frac{1}{N}\sum_{k=1}^N A_t^k\Big)\rmd t
\\&+2\sqrt{1+\frac{1}{N}} \1_{(0,\infty)}(r_t^i)\rmd W_t^i\eqsp,
\end{align*} 
where $A_t^i=|b*\barmu_t(\bar{X}_t^i)- \frac{1}{N}\sum_{j=1}^N b(\bar{X}_t^i-\bar{X}_t^j)|$. 

In addition, the statement of \Cref{lemma:comparison_propofchaos} carries over to the limiting process $(\{r_t^i\}_{i=1}^N)_{t\geq 0}$, 
 since by the weak convergence along the subsequences $(\delta_n)_{n\in\mathbb{N}}$ and $(\epsilon_m)_{m\in\mathbb{N}}$ and the Portmanteau theorem,
$P(|\tilde{X}^i_t-\tilde{Y}^i_t|\leq r_t^i \text{ for } i=1,\ldots,N)\geq\limsup_{m\to \infty}\limsup_{n\to \infty} P(|\tilde{X}_t^{i,\delta_n}-\tilde{Y}_t^{i,\delta_n}|\leq r_t^{i,\delta_n,\epsilon_m} \text{ for } i=1,\ldots,N)=1$, where $\tilde{X}^i_t=\bar{X}_t^i-(1/N)\sum_{j=1}^N\bar{X}_t^j$ and $\tilde{Y}^i_t=\bar{X}_t^i-(1/N)\sum_{j=1}^N\bar{Y}_t^j$ for all $t\geq 0$ and $i=1,\ldots,N$.

Using It\=o-Tanaka formula, c.f. \cite[Chapter 6, Theorem 1.1]{ReYo99}, and $f'$ is absolutely continuous, we obtain for $f$ defined in \eqref{eq:definition_f} with $\tilde{b}(r)=(\kappa(r)-L)r$ and $a=2\|\gamma\|_\infty$,
\begin{align*}
\rmd \Big(\frac{1}{N}\sum_{i=1}^N f(r_t^i)\Big)&=\frac{1}{N}\sum_{i=1}^N\Big(\bar{b}(r_t^i)f'(r_t^i)+f''(r_t^i)2\frac{N+1}{N} \1_{(0,\infty)}(r_t^i)\Big)\rmd t
\\ &  +\frac{1}{N^2}\sum_{i=1}^N\sum_{j=1}^N2f'(r_t^i)\|\gamma\|_\infty \1_{(0,\infty)}(r_t^j)\rmd t
\\ & +\frac{1}{N}\sum_{i=1}^N f'(r_t^i)2\sqrt{1+\frac{1}{N}} \1_{(0,\infty)}(r_t^i)\rmd W_t^i
 +\frac{1}{N}\sum_{i=1}^N f'(r_t^i)\Big(A_t^{i}+\frac{1}{N}\sum_{k=1}^NA_t^k\Big)\rmd t\eqsp.
\end{align*}
Taking expectation, we get using $f'(r)\leq 1$ for all $r\geq 0$, 
\begin{equation}\label{eq:expectation_ito}
\begin{aligned} 
\frac{\rmd }{\rmd t}\mathbb{E}\Big[\frac{1}{N}\sum_{i=1}^N f(r_t^i)\Big]&\leq \frac{1}{N}\sum_{i=1}^N\Big\{\mathbb{E}\Big[\bar{b}(r_t^i)f'(r_t^i)+2\frac{N+1}{N}(f''(r_t^i)-f''(0))\Big] 
\\ & + \mathbb{E}\Big[2\Big(\|\gamma\|_\infty+\frac{N+1}{N} f''(0)\Big)\1_{(0,\infty)}(r_t^i)\Big]+\mathbb{E}\Big[ 2A_t^{i}\Big]\Big\}\eqsp.
\end{aligned}
\end{equation}
By \eqref{eq:condition_f_1} and \eqref{eq:condition_f_2}, the first two terms are bounded by $-\tilde{c}\frac{1}{N}\sum_i f(r_t^i)$ with $\tilde{c}$ given in \eqref{eq:definition_tildec}. 

By \Cref{lemma:bound_A_t^idelta} the last term in \eqref{eq:expectation_ito} is bounded by
\begin{align*}
2E[A_t^i]\leq \tilde{C}N^{-1/2}\eqsp,
\end{align*}
where 
\begin{align} \label{eq:tildeC}
\tilde{C}=2C_2=4LC^{1/2}+8\|\gamma\|_\infty+2(\sqrt{2}C^{1/2}L+\|\gamma\|_\infty)\eqsp.
\end{align}
Hence, we obtain
\begin{align*}
\frac{d}{dt}\mathbb{E}\Big[\frac{1}{N}\sum_i f(r_t^i)\Big]&\leq -\tilde{c} \frac{1}{N}\sum_i\mathbb{E}[f(r_t^i)]+\tilde{C}N^{-1/2}
\end{align*}
for $t\geq 0$ which leads by Gr\"onwall's lemma to
\begin{align*}
\mathbb{E}\Big[\frac{1}{N}\sum_i f(r_t^i)\Big]\leq \rme^{-\tilde{c} t} \mathbb{E}\Big[\frac{1}{N}\sum_i f(r_0^i)\Big]+\frac{1}{\tilde{c}}\tilde{C}N^{-1/2}\eqsp.
\end{align*}
For an arbitrary coupling $\xi\in\Gamma(\bar{\mu}_0^{\otimes N},\nu_0^{\otimes N})$, we have
\begin{align*}
\mathcal{W}_{f,N}((\barmu_t)^{\otimes N},\nu_t^N)\leq \rme^{-\tilde{c} t} \int_{\mathbb{R}^{2Nd}} \frac{1}{N}\sum_{i=1}^N f \parenthese{\abs{x^i-y^i-\frac{1}{N}\sum_{j=1}^N(x^j-y^j)}} \xi(\rmd x \rmd y)+\frac{\tilde{C}}{\tilde{c}N^{1/2}}\eqsp,
\end{align*}
as $\mathbb{E}[f(r_0^i)]\leq \int_{\mathbb{R}^{2Nd}} \frac{1}{N}\sum_{i=1}^Nf(|x^i-y^i-\frac{1}{N}\sum_{j=1}^N(x^j-y^j)|) \xi(\rmd x \rmd y)$.
Taking the infimum over all couplings $\xi\in\Gamma(\bar{\mu}_0^{\otimes N},\nu_0^{\otimes})$ gives the first bound. By \eqref{eq:norm_equivalence}, the second bound follows.
\end{proof}

\subsection{Proof of \texorpdfstring{\Cref{sec:system_non-linear_sticky_diff}}{}} \label{sec:proof_meanfield_existence}

Analogously to the proof of \Cref{thm:existence_comparison}, we introduce approximations for the system of sticky SDEs and prove \Cref{thm:existence_comparison_Nparticles} using a comparison result given in \Cref{lemma:modification_watanabe_meanfield} and via taking the limit of the approximation of the system of sticky SDEs in two steps and identifying the limit with the solution of \eqref{eq:N_onedimSDE_coupling}.

As for the nonlinear case we show \Cref{thm:existence_comparison_Nparticles} 
via a family of stochastic differential equations, with Lipschitz continuous coefficients,
\begin{equation}\label{eq:Nd_stickydiff_coupling}
\begin{aligned}
&\rmd r_t^{i,n,m}=\Big(\tilde{b}(r_t^{i,n,m})+\frac{1}{N}\sum_{j=1}^Ng^m(r_t^{j,n,m})\Big)\rmd t+2\theta^n(r_t^{i,n,m}) \rmd W_t^i 
\\ & \rmd s_t^{i,n,m}=\Big(\hat{b}(s_t^{i,n,m})+\frac{1}{N}\sum_{j=1}^Nh^m(s_t^{j,n,m})\Big)\rmd t+2\theta^n(s_t^{i,n,m}) \rmd W_t^i 
\\ & \Law(r_0^{i,n,m},s_0^{i,n,m})=\eta_{n,m}\eqsp, \qquad i\in \{1,\ldots,N\}\eqsp,
\end{aligned}
\end{equation}
where $\eta_{n,m}\in\Gamma(\mu_{n,m},\nu_{n,m})$.
Under \Cref{H1b}, \Cref{H1g}, \Cref{H5}, \Cref{H4} and \Cref{H3} we identify the weak limit of $(\{r_t^{i,n,m},s_t^{i,n,m}\}_{i=1,n,m\in\mathbb{N}}^N)_{t\geq 0}$ solving \eqref{eq:Nd_stickydiff_coupling} for $n\to\infty$ by $(\{r_t^{i,m},s_t^{i,m}\}_{i=1,m\in\mathbb{N}}^N)_{t\geq 0}$ solving the family of SDEs given by  
\begin{equation}\label{eq:Nd_sticydiff_coupling2}
\begin{aligned} 
&\rmd r_t^{i,m}=\Big(\tilde{b}(r_t^{i,m})+\frac{1}{N}\sum_{j=1}^Ng^m(r_t^{j,m})\Big)\rmd t+2 \1_{(0,\infty)}(r_t^{i,m}) \rmd W_t^i \eqsp,  
\\ &\rmd s_t^{i,m}=\Big(\hat{b}(s_t^{i,m})+\frac{1}{N}\sum_{j=1}^Nh^m(s_t^{j,m})\Big)\rmd t+2 \1_{(0,\infty)}(s_t^{i,m}) \rmd W_t^i \eqsp, 
\\ & \Law(r_0^{i,m},s_0^{i,m})=\eta_m \eqsp, \qquad i\in \{1,\ldots,N\}\eqsp, 
\end{aligned} 
\end{equation}
where $\eta_m\in\Gamma(\mu_m,\nu_m)$.

Taking the limit $m\to \infty$, we obtain \eqref{eq:N_onedimSDE_coupling} as the weak limit of \eqref{eq:Nd_sticydiff_coupling2}.
In the case $g(r)=\1_{(0,\infty)}(r)$, we can choose $g^m=\theta^m$. 

Consider a probability space $(\Omega_0,\mathcal{A}_0,Q)$ and $N$ \iid $1$-dimensional Brownian motions $(\{W_t^i\}_{i=1}^N)_{t\geq 0}$. Note that under \Cref{H1b}--\Cref{H3}, there are random variables $\{r^{i,n,m}\}_{i=1}^N,\{s^{i,n,m}\}_{i=1}^N:\Omega_0\to\mathbb{W}^{N}$ for each $n,m$ such that $(\{r^{i,n,m},s^{i,n,m}\}_{i=1}^N)$ is a unique solution to \eqref{eq:Nd_stickydiff_coupling} by \cite[Theorem 2.2]{Me96}. We denote by $\PP^{n,m}=Q\circ(\{r^{i,n,m},s^{i,n,m}\}_{i=1}^N)^{-1}$ the law on $\mathbb{W}^N\times\mathbb{W}^N$.

Before taking the two limits and proving \Cref{thm:existence_comparison_Nparticles}, we introduce a modification of Ikeda and Watanabe's comparison theorem, to compare two solutions of \eqref{eq:Nd_stickydiff_coupling}, cf. \cite[Section VI, Theorem 1.1]{IkWa89}.

\begin{lemma}\label{lemma:modification_watanabe_meanfield}
 Suppose a solution $(\{r_t^{i,n,m},s_t^{i,n,m}\}_{i=1}^N)_{t\geq 0}$ of \eqref{eq:Nd_stickydiff_coupling} is given for fixed $n,m\in \mathbb{N}$. Assume \Cref{H5} for $g^{m}$ and $h^{m}$, \Cref{H1b} for $\tilde{b}$ and $ \hat{b}$, \Cref{H4} for $\theta^n$. If $Q[r_0^{i,n,m}\leq s_0^{i,n,m} \text{ for all } i=1,\ldots,N]=1$, $\tilde{b}(r)\leq \hat{b}(r)$ and $g^m(r)\leq h^m(r)$ for any $r\in\mathbb{R}_+$, then
\begin{align*}
Q[r_t^{i,n,m}\leq s_t^{i,n,m} \text{ for all } t\geq 0 \text{ and } i=1,\ldots,N]=1
\end{align*}
\end{lemma}

\begin{proof}
The proof is similar for each component $i=1,\ldots,N$ to the proof of \Cref{lemma:modification_watanabe_nonlinear}. It holds for the interaction part similarly to \eqref{eq:comparison_lemma_nonlineardriftpart} using the properties of $g^m$ and $h^m$,
\begin{align*}
\frac{1}{N}\sum_{j=1}^N (g^m(r^{j,n,m}_t)-h^m(s^{j,n,m}_t))\leq K_m\frac{1}{N}\sum_{j=1}^N|r^{j,n,m}_t-s^{j,n,m}_t|\1_{(0,\infty)}(r^{j,n,m}_t-s^{j,n,m}_t) \eqsp.
\end{align*}
Hence, we obtain analogously to \eqref{eq:modificationwatanabe_nonlinear},
\begin{align*}
\mathbb{E}[(r_t^{i,n,m}-s_t^{i,n,m})_+]&\leq \tilde{L}\mathbb{E}\Big[\int_0^t(r^{i,n,m}_u-s^{i,n,m}_u)_+ \rmd u\Big]
 +K_m\mathbb{E}\Big[\int_0^t  \frac{1}{N}\sum_{j=1}^N(r_u^{j,n,m}-s_u^{j,n,m})_+ \rmd u\Big]
\end{align*}
for all $i=1,\ldots,N$. 
Assume 
$t^*=\inf\{t\geq 0:\mathbb{E}[(r_t^{i,n,m}- s_t^{i,n,m})_+]>0 \text{ for some } i \}<\infty$. Then, there exists $i\in\{1,\ldots,N\}$ such that $\int_0^{t^*}\mathbb{E}[(r_u^{i,n,m}-s^{i,n,m}_u)_+ ]\rmd u>0$. 
But, by definition of $t^*$, for all $i$, $u<t^*$, $\mathbb{E}[(r^{i,n,m}_u-s^{i,n,m}_u)_+ ]=0$. This contradicts the definition of $t^*$. Hence, $Q[r_t^{i,n,m}\leq s_t^{i,n,m} \text{ for all } i, \ t\geq 0]=1$.
\end{proof}

In the next step, we prove that the distribution of the solution of \eqref{eq:Nd_stickydiff_coupling} converges as $n\to\infty$.

\begin{lemma} \label{lem:existence_NstickySDE_step1}
Assume that \Cref{H1b} and \Cref{H1g} is satisfied for $(\tilde{b},g)$ and $(\hat{b},h)$. Further, let $(\theta^n)_{n\in\mathbb{N}}$, $(g^m)_{m\in\mathbb{N}}$, $(h^m)_{m\in\mathbb{N}}$, $(\mu_{n,m})_{n,m\in\mathbb{N}}$, $(\nu_{n,m})_{n,m\in\mathbb{N}}$ and $(\eta_{n,m})_{n,m\in\mathbb{N}}$ be such that \Cref{H5},  \Cref{H4} and \Cref{H3} hold. 
Let $m\in\mathbb{N}$. Then there exists a random variable $(\{r^{i,m},s^{i,m}\}_{i=1}^N)$ defined on some probability space $(\Omega^m,\mathcal{A}^m,P^m)$ with values in $\mathbb{W}^N\times \mathbb{W}^N$ such that $(\{ r_t^{i,m},s_t^{i,m}\}_{i=1}^N)_{t\geq 0}$ is a weak solution of \eqref{eq:Nd_sticydiff_coupling2}. Moreover, the laws $Q\circ(\{r^{i,n,m},s^{i,n,m}\}_{i=1}^N)^{-1}$ converge weakly 
to $P^m\circ(\{r^{i,m},s^{i,m}\}_{i=1}^N)^{-1}$. If in addition, 
\begin{align*}
& \tilde{b}(r)\leq \hat{b}(r) \quad \text{and} \quad g^m(r)\leq h^m(r)&&  \text{ for any } r\in\mathbb{R}_+,
\\  &Q[r_0^{i,n,m}\leq s_0^{i,n,m}]=1&& \text{ for any } n\in\mathbb{N}, i=1,\ldots,N, 
\end{align*} 
then $P^m[r^{i,m}_t\leq s^{i,m}_t \text{ for all } t\geq0 \text{ and } i\in\{1,\ldots,N\}]=1$. 
\end{lemma}

\begin{proof} Fix $m\in\mathbb{N}$.
The proof is divided in three parts and is similar to the proof of \Cref{lem:existence_stickySDE_step1}. First we show tightness of the sequences of probability measures. Then we identify the limit of the sequence of stochastic processes. Finally, we compare the two limiting processes.

\textbf{Tightness:} 
We show analogously as in the proof of \Cref{lem:existence_stickySDE_step1} that the sequence of probability measures  $(\mathbb{P}^{n,m})_{n\in\mathbb{N}}$ on $(\mathbb{W}^N\times \mathbb{W}^N,\mathcal{B}(\mathbb{W}^N)\otimes \mathcal{B}(\mathbb{W}^N))$ is tight by applying Kolmogorov's continuity theorem. 
We consider $p>2$ such that the $p$-th moment in \Cref{H3} are uniformly bounded.
Fix $T>0$. Then the $p$-th moment of  $r_t^{i,n,m}$ and $s_t^{i,n,m}$ for $t<T$ is bounded using It\=o's formula,
\begin{align*}
\rmd |&r_t^{i,n,m}|^p \leq p|r_t^{i,n,m}|^{p-2}\langle r_t^{i,n,m},(\tilde{b}(r_t^{i,n,m})+\frac{1}{N}\sum_{j=1}^Ng^m(r_t^{j,n,m}))\rangle \rmd t 
 \\ &\quad + 2\theta^n(r_t^{i,n,m}) p|r_t^{n,m}|^{p-2} r_t^{i,n,m}\rmd W_t^i +p(p-1)|r_t^{i,n,m}|^{p-2}2 \theta^n(r_t^{i,n,m})^2\rmd t
\\ & \leq p\Big(|r_t^{i,n,m}|^p \tilde{L}+\Gamma|r_t^{i,n,m}|^{p-1}+2(p-1)|r_t^{i,n,m}|^{p-2}\Big) \rmd t  + 2\theta^n(r_t^{i,n,m}) p(r_t^{i,n,m})^{p-1}\rmd W_t^i
\\ & \leq p\Big(\tilde{L}+\Gamma+2(p-1)\Big)|r_t^{i,n,m}|^p\rmd t + p(\Gamma+2(p-1))\rmd t + 2\theta^n(r_t^{i,n,m}) p(r_t^{n,m})^{p-1}\rmd W_t^i\eqsp.
\end{align*}
Taking expectation yields
\begin{align*}
\frac{\rmd}{\rmd t}\mathbb{E}[ |r_t^{i,n,m}|^p]&\leq p\Big(L+\Gamma+2(p-1)\Big)\mathbb{E}|r_t^{i,n,m}|^p +p(\Gamma+2(p-1))\eqsp.
\end{align*}
Then by Gronwall's lemma
\begin{align} \label{eq:pthmomentbound_Npart}
 \sup_{t\in[0,T]}\mathbb{E}[|r_t^{i,n,m}|^p] \leq \rme^{p(L+\Gamma+2(p-1))T} (\mathbb{E}[|r_0^{i,n,m}|^p]+Tp(\Gamma+2(p-1))
 )<C_p<\infty \eqsp,
\end{align}
where $C_p$ depends on $T$ and the $p$-th moment of the initial distribution, which is by assumption finite.
Similarly, it holds $\sup_{t\in[0,T]} \mathbb{E}[|s_t^{i,n,m}|^p]<C_p$ for $t\leq T$.
Using these moment bounds, it holds for all $t_1,t_2\in[0,T]$ by \Cref{H1b}, \Cref{H5} and \Cref{H4},
\begin{align*}
\mathbb{E}&[|r_{t_2}^{i,n,m}-r_{t_1}^{i,n,m}|^p]
\\ & \leq C_1(p)\Big(\mathbb{E}[|\int_{t_1}^{t_2} \tilde{b}(r_u^{i,n,m}) +\frac{1}{N}\sum_{j=1}^N g^m(r_t^{j,n,m})\rmd u |^p]+\mathbb{E}[|\int_{t_1}^{t_2}2 \theta^n(r_u^{i,n,m})\rmd W_u^i|^p]\Big)
\\ & \leq C_2(p)\Big(\Big(\mathbb{E}\Big[\frac{\tilde{L}^p}{|t_2-t_1|}\int_{t_1}^{t_2}|r_u^{i,n,m}|^p\rmd u\Big]+\Gamma^p\Big)|t_2-t_1|^p+\mathbb{E}[|\int_{t_1}^{t_2}2\theta^n(r_u^{i,n,m})\rmd u|^{p/2}]\Big)
\\ & \leq C_2(p)\Big(\Big(\frac{\tilde{L}^p}{|t_2-t_1|}\int_{t_1}^{t_2}\mathbb{E}[|r_u^{i,n,m}|^p]\rmd u +\Gamma^p\Big)|t_2-t_1|^p+2^{p/2}|t_2-t_1|^{p/2}\Big)
\\ & \leq C_3(p,T,\tilde{L},\Gamma,C_p)|t_2-t_1|^{p/2} \eqsp,
\end{align*}
where $C_k(\cdot)$ are constants depending on the stated arguments, but independent of $n,m$. Note that in the second step, we use Burkholder-Davis-Gundy inequality, see \cite[Chapter IV, Theorem 48]{Pr04}.
It holds similarly, $\mathbb{E}[|s_{t_2}^{i,n,m}-s_{t_1}^{i,n,m}|^p]\leq C_3(p,T,\tilde{L},\Gamma,C_p)|t_2-t_1|^{p/2}$. Hence, 
\begin{align*}
\mathbb{E}[|(\{r_{t_2}^{i,n,m},s_{t_2}^{i,n,m}\}_{i=1}^N)&-(\{r_{t_1}^{i,n,m},s_{t_1}^{i,n,m}\}_{i=1}^N)|^p]
\\ &\leq C_4(p,N)(\sum_{i=1}^N( 
\mathbb{E}[|r_{t_2}^{i,n,m}-r_{t_1}^{i,n,m}|^p]+\mathbb{E}[|s_{t_2}^{i,n,m}-s_{t_1}^{i,n,m}|^p]))
\\ & \leq C_5(p,N,T,\tilde{L},\Gamma,C_p)|t_2-t_1|^{p/2}
\end{align*}
for all $t_1,t_2\in[0,T]$.
Hence, by Kolmogorov's continuity criterion, cf. \cite[Corollary 14.9]{Ka02}, 
there exists a constant $\tilde{C}$ depending on $p$ and $\gamma$ such that\begin{equation} \label{eq:pthmoment_difference_N}
\mathbb{E}\Big[ [(\{r^{i,n,m},s^{i,n,m}\}_{i=1}^N)]_\gamma^p\Big]\leq \tilde{C}\cdot C_5(p,N,T,\tilde{L},\Gamma,C_p) \eqsp.
\end{equation}
where $[\cdot]_\gamma^p$ is defined by $[x]_\gamma=\sup_{t_1,t_2\in[0,T]}\frac{|x(t_1)-x(t_2)|}{|t_1-t_2|^\gamma}$
and $(\{r_t^{i,n,m},s_t^{i,n,m}\}_{i=1}^N)_{n\in\mathbb{N},t\geq 0}$ is tight in $\mathcal{C}([0,T],\mathbb{R}^{2N})$.
Hence, for each $T>0$ there exists a subsequence $n_k\to \infty$ and a probability measure $\mathbb{P}_T$ on $\mathcal{C}([0,T],\mathbb{R}^{2N})$. Since $\{\PP_T^m\}_T$ is a consistent family, there exists by \cite[Theorem 5.16]{Ka02} a probability measure $\PP^m$ on
$(\mathbb{W}^N\times\mathbb{W}^N,\mathcal{B}(\mathbb{W}^N)\otimes\mathcal{B}(\mathbb{W}^N))$ such that $\PP^{n_k,m}$ converges weakly to $\PP^m$. Note that we can take here the same subsequence $(n_k)$ for all $m$ using a diagonalization argument.
\\ \textbf{Characterization of the limit measure:} Denote by $(\{\mathbf{r}_t^i,\mathbf{s}_t^i\}_{i=1}^N)=\omega(t)$ the canonical process on $\mathbb{W}^N\times\mathbb{W}^N$. To characterize the measure $\PP^m$ we first note that $\PP^m\circ(\mathbf{r}_0^i,\mathbf{s}_0^i)^{-1}=\eta_m$ for all $i\in\{1,\ldots,N\}$, since $\PP^{n,m}(\mathbf{r}_0^i,\mathbf{s}_0^i)^{-1}=\eta_{n,m}$ converges weakly to $\eta_m$ by assumption.
We define maps $M^{i,m},N^{i,m}:\mathbb{W}^{N}\times\mathbb{W}^{N}\to\mathbb{W}$ by
\begin{equation} \label{eq:martingaleN}
\begin{aligned}
&M_t^{i,m}=\mathbf{r}_t^i-\mathbf{r}_0^i-\int_0^t\Big(\tilde{b}(\mathbf{r}_u^i)+\frac{1}{N}\sum_{j=1}^Ng^m(\mathbf{r}_u^j)\Big)\rmd u\eqsp , && \text{ and } 
\\ & N_t^{i,m}=\mathbf{s}_t^i-\mathbf{s}_0^i-\int_0^t\Big(\hat{b}(\mathbf{s}_u^i)+\frac{1}{N}\sum_{j=1}^Nh^m(\mathbf{s}_u^j)\Big)\rmd u \eqsp.
\end{aligned}
\end{equation}
For each $n,m\in\mathbb{N}$ and $i=1,\ldots,N$, $(M_t^{i,m},\mathcal{F}_t,\PP^{n,m})$ is a martingale with respect to the filtration $\mathcal{F}_t=\sigma((\mathbf{r}_u^{j},\mathbf{s}_u^{j}):{j=1,\ldots,N,0\leq u\leq t})$. Note that the families $(\{M_t^{i,m}\}_{i=1}^N,\PP^{n,m})_{n\in\mathbb{N},t\geq 0}$ and $(\{N_t^{i,m}\}_{i=1}^N,\PP^{n,m})_{n\in\mathbb{N},t\geq 0}$ are uniformly integrable.
Since the mappings $M^{i,m}$ and $N^{i,m}$ are continuous in $\mathbb{W}$, $\PP^{n,m}\circ(\{\mathbf{r}^i,\mathbf{s}^i,M^{i,m},N^{i,m}\}_{i=1}^N)^{-1}$ converges weakly to $\PP^{m}\circ(\{\mathbf{r}^i,\mathbf{s}^i,M^{i,m},N^{i,m}\}_{i=1}^N)^{-1}$ by the continuous mapping theorem. Then applying the same argument as in \eqref{eq:martingale_property}, $(M^{m,i}_t,\mathcal{F}_t,\PP^m)$ and  $(N^{m,i}_t,\mathcal{F}_t,\PP^m)$ are continuous martingales for all $i=1,\ldots,N$ and the quadratic variation $([\{M^{i,m},N^{i,m}\}_{i=1}^N]_t)_{t\geq 0}$ exists $\PP^m$-almost surely. To complete the identification of the limit, it suffices to identify the quadratic variation.
Similar to the computations in the proof of \Cref{lem:existence_stickySDE_step1}, it holds
\begin{equation} \label{eq:quadraticvariation_N1}
\begin{aligned}
&[M^{i,m}]=4\int_0^\cdot \1_{(0,\infty)}(\mathbf{r}_u^i)\rmd u && \PP^m\text{-almost surely,}
\\  & [N^{i,m}]=4\int_0^\cdot \1_{(0,\infty)}(\mathbf{s}_u^i)\rmd u && \PP^m\text{-almost surely, and}
\\  & [M^{i,m},N^{i,m}]=4\int_0^\cdot \1_{(0,\infty)}(\mathbf{r}_u^i)\1_{(0,\infty)}(\mathbf{s}_u^i)\rmd u && \PP^m\text{-almost surely,}
\end{aligned}
\end{equation}
Further, $[M^{i,m},M^{j,m}]_t=[N^{i,m},N^{j,m}]_t=[M^{i,m},N^{j,m}]_t=0$ $\PP^{n,m}$-almost surely for $i\neq j$ and $(M_t^{i,m}M_t^{j,m},\PP^{n,m})$, $(N_t^{i,m}N_t^{j,m},\PP^{n,m})$ and $(M_t^{i,m}N_t^{j,m},\PP^{n,m})$ are martingales. For any bounded, continuous non-negative function $G:\mathbb{W}\to\mathbb{R}$, it holds
\begin{align*}
\mathbb\mathbb{E}^m[G(M^{i,m}_tM^{j,m}_t-M^{i,m}_sM^{j,m}_s)]=\lim_{n\to\infty}\mathbb\mathbb{E}^{n,m}[G(M_t^{i,m}M_t^{j,m}-M_s^{i,m}M_s^{j,m})]=0 \eqsp,
\end{align*}
respectively, $\mathbb\mathbb{E}^m[G(N^{i,m}_tN^{j,m}_t-N^{i,m}_sN^{j,m}_s)]=0$ and $\mathbb\mathbb{E}^m[G(M^{i,m}_tN^{j,m}_t-M^{i,m}_sN^{j,m}_s)]=0$.
Then 
\begin{align}\label{eq:quadraticvariation_N2}
[M^{i,m},M^{j,m}]=[N^{i,m},N^{j,m}]=[M^{i,m},N^{j,m}]=0&& \PP^m\text{-almost surely, for all } i\neq j \eqsp.
\end{align}
Then by a martingale representation theorem, cf. \cite[Chapter II, Theorem 7.1]{IkWa89},  there is a probability space $(\Omega^m, \mathcal{A}^m,P^m)$ and a Brownian motion $\{W^i\}_{i=1}^N$ and random variables $(\{r^{i,m},s^{i,m}\}_{i=1}^N)$ on this space, such that it holds $P^m\circ(\{r^{i,m},s^{i,m}\}_{i=1}^N)^{-1}=\PP^m\circ(\{\mathbf{r}^i,\mathbf{s}^i\}_{i=1}^N)^{-1}$ and such that $(\{r^{i,m},s^{i,m},W^i\}_{i=1}^N)$ is a weak solution of \eqref{eq:Nd_sticydiff_coupling2}.
\\ \textbf{Comparison of two solutions:} To show $P^m[r^{i,m}_t\leq s^{i,m}_t \text{ for all }t\geq 0 \text{ and } i=1,\ldots,N]=1$  it suffices to note that $P^{n,m}[r_t^{i,n,m}\leq s_t^{i,n,m} \text{ for all } t\geq 0 \text{ and }i=1,\ldots,N]=1$, which holds by \Cref{lemma:modification_watanabe_meanfield}, carries over to the limit by the Portmanteau theorem, since we have weak convergence of $\PP^{n,m}\circ(\{\mathbf{r}^i,\mathbf{s}^i\}_{i=1}^N)^{-1}$ to $\PP^m\circ(\{\mathbf{r}^i,\mathbf{s}^i\}_{i=1}^N)^{-1}$.
\end{proof}

In the next step we show that the distribution of the solution of \eqref{eq:Nd_sticydiff_coupling2} converges as $m\to\infty$.
 Consider a probability space $(\Omega^m,\mathcal{A}^m,P^m)$ for each $m\in\mathbb{N}$ and random variables $\{r^{i,m}\}_{i=1}^N,\{s^{i,m}\}_{i=1}^N:\Omega^m\to\mathbb{W}^N$ such that $(\{r^{i,m}_t,s_t^{i,m}\}_{i=1}^N)_{t\geq 0}$ is a solution to \eqref{eq:Nd_sticydiff_coupling2}. Denote by $\PP^m=P^m\circ(\{r^{i,m},s^{i,m}\}_{i=1}^N)^{-1}$ the law on $\mathbb{W}^N\times\mathbb{W}^N$.

\begin{lemma} \label{lem:existence_NstickySDE_step2}
Assume that \Cref{H1b} and \Cref{H1g} is satisfied for $(\tilde{b},g)$ and $(\hat{b},h)$. 
Let $\eta\in\Gamma(\mu,\nu)$ where the probability measures $\mu$ and $\nu$ on $\mathbb{R}_+$ satisfy \Cref{H2}.
Further, let $(g^m)_{m\in\mathbb{N}}$, $(h^m)_{m\in\mathbb{N}}$, $(\mu_{m})_{m\in\mathbb{N}}$, $(\nu_{m})_{m\in\mathbb{N}}$ and $(\eta_{m})_{m\in\mathbb{N}}$ be such that \Cref{H5} and  \Cref{H3} hold. 
Then there exists a random variable $(\{r^i,s^i\}_{i=1}^N)$ defined on some probability space $(\Omega,\mathcal{A},P)$ with values in $\mathbb{W}^N\times \mathbb{W}^N$ such that $( \{r_t^{i},s_t^{i}\}_{i=1}^N)$ is a weak solution of \eqref{eq:N_onedimSDE_coupling}. Moreover, the laws $P^{m}\circ( \{r^{i,m},s^{i,m}\}_{i=1}^N)^{-1}$ converge weakly to $P\circ(\{r^i,s^i\}_{i=1}^N)^{-1}$. If in addition, 
\begin{align*}
&\tilde{b}(r)\leq \hat{b}(r), \quad g(r)\leq h(r), \quad \text{and} \quad g^m(r)\leq h^m(r)&& \text{ for any } r \in\mathbb{R}_+ \text{, and }
\\ &P^m[r_0^{i,m}\leq s_0^{i,m} \text{ for all } t\geq 0 \text{ and } i\in\{1,\ldots,N\}]=1 
&& \text{ for any } m\in\mathbb{N},
\end{align*} 
then $P[r^i_t\leq s^i_t \text{ for all } t\geq0 \text{ and } i\in\{1,\ldots,N\}]=1$. 
\end{lemma}

\begin{proof}
The proof is structured as the proof of \Cref{lem:existence_NstickySDE_step1}. 
Tightness of the sequence of probability measures $(\PP^{m})_{m\in\mathbb{N}}$ on $(\mathbb{W}^N\times\mathbb{W}^N,\mathcal{B}(\mathbb{W}^N)\otimes\mathcal{B}(\mathbb{W}^N))$ holds adapting the steps of the proof of \Cref{lem:existence_NstickySDE_step1} to \eqref{eq:Nd_sticydiff_coupling2}. Note that \eqref{eq:pthmomentbound_Npart} and \eqref{eq:pthmoment_difference_N} hold analogously for $(\{r_t^{i,m},s_t^{i,m}\}_{i=1}^N)$ by \Cref{H1b}, \Cref{H5} and \Cref{H3}. Hence by Kolmogorov's continuity criterion, cf. \cite[Corollary 14.9]{Ka02},
we can deduce that there exists a probability measure $\PP$ on $(\mathbb{W}^N\times\mathbb{W}^N,\mathcal{B}(\mathbb{W}^N)\otimes\mathcal{B}(\mathbb{W}^N))$ such that there is a subsequence $(m_k)_{k\in\mathbb{N}}$ along which $\PP^{m_k}$ converge towards $\PP$.

To characterize the limit, we first note that by Skorokhod representation theorem, cf. \cite[Chapter 1, Theorem 6.7]{Bi99}, without loss of generality we can assume that $(\{r^{i,m},s^{i,m}\}_{i=1}^N)$ are defined on a common probability space $(\Omega,\mathcal{A},P)$ with expectation $E$ and converge almost surely to $(\{r^i,s^i\}_{i=1}^N)$ with distribution $\PP$. Then, by \Cref{H5} and Lebesgue convergence theorem it holds almost surely for all $t\geq 0$,
\begin{align} \label{eq:limit_drift_N}
\lim_{m\to\infty}\int_0^t \tilde{b}(r_t^{i,m})+\frac{1}{N}\sum_{j=1}^N g^m(r_u^{j,m})\rmd u=\int_0^t \tilde{b}(r_t^i)+\frac{1}{N}\sum_{j=1}^N g^m(r_u^j)\rmd u \eqsp.
\end{align}
Consider the mappings ${M}^{i,m},{N}^{i,m}:\mathbb{W}^N\times\mathbb{W}^N\times\mathcal{P}(\mathbb{W}^N\times\mathbb{W}^N)\to \mathbb{W}$ defined by \eqref{eq:martingaleN}
Then for all $m\in\mathbb{N}$ and $i=1,\ldots,N$, $({M}_t^{i,m},\mathcal{F}_t,\PP^m)$ and $({N}_t^{i,m},\mathcal{F}_t,\PP^m)$ are martingales with respect to the canonical filtration $\mathcal{F}_t=\sigma((\{\mathbf{r}_u^{i},\mathbf{s}_u^{i}\}_{i=1}^N)_{0\leq u\leq t})$. Further  the family $(\{M_t^{i,m}\}_{i=1}^N,\PP^m)_{m\in\mathbb{N},t\geq 0}$ and $(\{N_t^{i,m}\}_{i=1}^N,\PP^m)_{m\in\mathbb{N},t\geq 0}$ are uniformly integrable. In the same line as weak convergence is shown in the proof of \Cref{lem:existence_stickySDE_step1} and by \eqref{eq:limit_drift_N}, $\PP^m\circ(\{\mathbf{r}^i,\mathbf{s}^i,M^{i,m},N^{i,m}\}_{i=1}^N)^{-1}$ converges weakly to $\PP\circ(\{\mathbf{r}^i,\mathbf{s}^i,M^{i},N^{i}\}_{i=1}^N)^{-1}$ where
\begin{align*}
&M_t^{i}=\mathbf{r}_t^i-\mathbf{r}_0^i-\int_0^t\Big(\tilde{b}(\mathbf{r}_u^i)+\frac{1}{N}\sum_{j=1}^Ng(\mathbf{r}_u^j)\Big)\rmd u \eqsp , && \text{ and } 
\\ & N_t^{i}=\mathbf{s}_t^i-\mathbf{s}_0^i-\int_0^t\Big(\hat{b}(\mathbf{s}_u^i)+\frac{1}{N}\sum_{j=1}^Nh(\mathbf{s}_u^j)\Big)\rmd u\eqsp.
\end{align*}
Then $(\{{M}_t^i\}_{i=1}^N,\mathcal{F}_t,\PP)$ and $(\{{N}_t^i\}_{i=1}^N,\mathcal{F}_t,\PP)$ are continuous martingales using the same argument as in \eqref{eq:martingale_property}. Further, the quadratic variation $([\{{M}_t^i,{N}_t^i\}_{i=1}^N]_t)_{t\geq 0}$ exists $\PP$-almost surely and is given by \eqref{eq:quadraticvariation_N1} and \eqref{eq:quadraticvariation_N2} $\PP$-almost surely, which holds following the computations in the proof of \Cref{lem:existence_stickySDE_step1} and \Cref{lem:existence_NstickySDE_step1}. As in \Cref{lem:existence_NstickySDE_step1}, we conclude by a martingale representation theorem that there are a probability space $(\Omega,\mathcal{A},P)$ and a Brownian motion $\{W^i\}_{i=1}^N$ and random variables $(\{r^i\}_{i=1}^N,\{s^i\}_{i=1}^N)$ on this space such that $P\circ(\{r^i,s^i\}_{i=1}^N)^{-1}=\PP\circ(\{\mathbf{r}^i,\mathbf{s}^i\}_{i=1}^N)^{-1}$ and such that $(\{r^i,s^i,W^i\}_{i=1}^N)$ is a weak solution of \eqref{eq:two_onedim_stickydiff}.

By the Portmanteau theorem the monotonicity carries over to the limit, since $\PP^m\circ(\{\mathbf{r}^i,\mathbf{s}^i\}_{i=1}^N)^{-1}$ converges weakly to $\PP\circ(\{\mathbf{r}^i,\mathbf{s}^i\}_{i=1}^N)^{-1}$.
\end{proof}

\begin{proof}[Proof of \Cref{thm:existence_comparison_Nparticles}]
The proof is a direct consequence of  \Cref{lem:existence_NstickySDE_step1} and \Cref{lem:existence_NstickySDE_step2}. 
\end{proof}

\appendix
\section{Appendix} \label{appendix}
\subsection{Kuramoto model} \label{appendix1}

Lower bounds on the contraction rate can also be shown for nonlinear SDEs on the one-dimensional torus using the same approach. Here, we consider the Kuramoto model given by
\begin{align} \label{eq:SDE_kuramoto}
\rmd X_t=-k\parentheseDeux{\int_{\torus} \sin(X_t-x)\rmd\mu_t(x)}\rmd t+\rmd B_t
\end{align}
on the torus $\mathbb{T} = \mathbb{R}/(2\pi\mathbb{Z})$. 

\begin{theorem} \label{thm:kuramoto}

Let $\mu_t$ and $\nu_t$ be laws of $X_t$ and $Y_t$ where $(X_s)_{s\geq 0}$ and $(Y_s)_{s\geq 0}$ are two solutions of \eqref{eq:SDE_kuramoto} with initial distributions $\mu_0$ and $\nu_0$ on $(\torus,\mathcal{B}(\torus))$, respectively.
If 
\begin{align} \label{eq:k_kuramoto}
4k\int_0^\pi\exp(2k-2k\cos(r/2))\rmd r\leq 1
\end{align} 
holds, then for all $t\geq 0$,
\begin{align*}
\mathcal{W}_{\tilde{f}}(\mu_t, \nu_t)\leq \rme^{-c_\mathbb{T} t} \mathcal{W}_{\tilde{f}}(\mu_0,\nu_0) \qquad \text{and} \qquad \mathcal{W}_{1}(\mu_t, \nu_t)\leq 2\exp(2k) \rme^{-c_\mathbb{T} t} \mathcal{W}_{1}(\mu_0,\nu_0) \eqsp,
\end{align*}
where
\begin{align}  \label{eq:c_T}
    c_\mathbb{T}=1/(2\int_0^{\pi} \int_0^r \exp[2k(\cos(r/2)-\cos(s/2))]\rmd s\rmd r)
\end{align}
and $\tilde{f}$ is a concave, increasing function given in \eqref{eq:f_Kuramoto}.
\end{theorem}
In \cite[Appendix A]{De20}, a contraction result is stated for a general drift using a similar approach.

We prove \Cref{thm:kuramoto} via a sticky coupling approach. In the same line as in \Cref{sec:main_contraction_results} the coupling  $(X_t, Y_t)_{t\ge 0}$ is defined as 
the weak limit of Markovian couplings $\{(X_t^\delta,Y_t^\delta)_{t\geq 0}: \delta>0\}$ on $\mathbb{T}\times \mathbb{T}=\mathbb{R}/(2\pi\mathbb{Z})\times\mathbb{R}/(2\pi\mathbb{Z})$ given by
\begin{equation}\label{eq:coupling_kuramoto}
\begin{aligned} 
&\rmd X_t^\delta=-k\parentheseDeux{\int_\mathbb{T}\sin(X_t^\delta  -x)\rmd \mu_t^\delta(x)}\rmd t+\mathrm{rc}^\delta(\bar{r}_t^\delta)\rmd B_t^1+\mathrm{sc}^\delta(\bar{r}_t^\delta)\rmd B_t^2
\\ &\rmd Y_t^\delta=-k\parentheseDeux{\int_\mathbb{T}\sin(Y_t^\delta  -x)\rmd \nu_t^\delta(x)}\rmd t-\mathrm{rc}^\delta(\bar{r}_t^\delta)\rmd B_t^1+\mathrm{sc}^\delta(\bar{r}_t^\delta)\rmd B_t^2\eqsp,
\end{aligned}
\end{equation}
where $\bar{r}_t^\delta=d_\mathbb{T}(X_t^\delta,Y_t^\delta)$ with $d_\mathbb{T}(\cdot,\cdot)$ defined by 
\begin{align} \label{eq:distance_torus}
d_{\mathbb{T}}(x,y)=\begin{cases}(|x-y| \ \mathrm{mod} \ 2\pi) & \text{if } (|x-y| \ \mathrm{mod} \ 2\pi)\leq \pi\eqsp, \\ (2\pi-|x-y| \ \mathrm{mod}\ 2\pi) & \text{otherwise}\eqsp.  \end{cases}
\end{align}
The functions $\mathrm{rc}^\delta,\mathrm{sc}^\delta$ are given by \eqref{eq:condition_rc_sc} and 
satisfy that there exists $\epsilon_0>0$ such that $\mathrm{rc}^\delta(r)\geq r/2$ for any $0\leq r\leq \delta\leq \epsilon_0$.

\begin{theorem}
  \label{theo:coupling_kuramoto}
  Assume \eqref{eq:k_kuramoto}. Let $\mu_0$ and $\nu_0$ be probability measures on $(\mathbb{T},\mathcal{B}(\mathbb{T}))$ having finite forth moment.
  Then, $(X_t,Y_t)_{t \ge 0} $ is a subsequential limit in distribution as $\delta \to 0$ of $\{(X_t^{\delta},Y_t^{\delta})_{t \geq 0} \, : \, \delta >0\}$, where $(X_t)_{t\ge 0}$ and $(Y_t)_{t\ge 0}$ are solutions of \eqref{eq:SDE_kuramoto} with initial distributions ${\mu}_0$ and ${\nu}_0$, respectively. 
  Further, there exists a process $(r_t)_{t \geq 0}$ satisfying for any $t\geq 0$,  $  d_{\mathbb{T}}(X_t , Y_t) \leq r_t$ almost surely, and which is a weak solution of
  \begin{equation}
    \label{eq:Kuramoto_sticky_sde}
\rmd r_t=(2k\sin(r_t/2)+2k\PP(r_t))\rmd t+2 \1_{(0,\pi]}(r_t) \rmd W_t-2\rmd \ell_t^\pi \eqsp,
\end{equation}
where $(W_t)_{t \geq 0}$ is a one-dimensional Brownian motion on $\mathbb{T}$ and $\ell^\pi$ is the local time at $\pi$. 
\end{theorem}

\begin{proof}
The proof works analogously to the proof of \Cref{theo:1}  stated in \Cref{sec:proof_thm_nonlinear}. 
It holds similarly to \Cref{lemma:SDEr_t^delta} by Meyer-Tanaka's formula, cf. \cite[Chapter 6, Theorem 1.1]{ReYo99}, and using \eqref{eq:distance_torus},
\begin{align*}
\bar{r}_t^\delta-\bar{r}_0^\delta&=\int_0^t \mathrm{sgn}(X_t^\delta-Y_t^\delta)(-k)e_t\parentheseDeux{\int_\mathbb{T}\sin(X_t^\delta  -x)\rmd \mu_t(x)-\int_\mathbb{T}\sin(Y_t^\delta  -x)\rmd \nu_t(x)}\rmd t
\\ & +\int_0^t \mathrm{sgn}(X_t^\delta-Y_t^\delta)2\mathrm{rc}^\delta(\bar{r}_t^\delta)e_t\rmd B_t^1+\int_\mathbb{R}2\mathrm{rc}^\delta(\bar{r}_t^\delta)^2 \ell_t^a(\delta_0-\delta_\pi)(\rmd a) \eqsp,
\end{align*}
where $\mathrm{sgn}(x)=\1_{(0,\pi]}(x)-\1_{(\pi,2\pi]}(x)$, $(\ell_t^a)_{t\geq 0}$ is the local time at $a$ associated with $(X_t^\delta-Y_t^\delta)_{t\geq 0}$ and  $e_t=(X_t^\delta-Y_t^\delta)/d_{\mathbb{T}}(X_t^\delta,Y_t^\delta)$ for $\bar{r}_t^\delta\neq 0$. For $\bar{r}_t^\delta=0$, $e_t$ is some arbitrary unit vector. 
For any $a$ the support of $\ell_t^a$ as a function of $t$ is a subset of the set of $t$ such that $r_t=a$ \cite[Theorem 19.1]{Ka02}, hence $\1_{(0,\pi]}(r_t)\ell_t^0=0$ almost surely and 
so the term involving the local time reduces to $-2\ell_t^\pi$. 
Further, we note that $W_t=\int_0^t \mathrm{sgn}(X_t^\delta-Y_t^\delta)e_t\rmd B_t^1$ is a Brownian motion. 
As in \Cref{lemma:SDEr_t^delta}, it holds for the process $(\bar{r}_t^\delta)_{t\geq 0}$ for $\epsilon<\epsilon_0$ with $\epsilon_0$ given by \eqref{eq:condition_rc_sc2},
\begin{align*}
\rmd \bar{r}_t^\delta&\leq  (2k\sin(\bar{r}_t^\delta/2)+2k\mathbb{E}_{x\sim \mu_t^{\delta},y\sim\nu_t^\delta}(\mathrm{rc}^\epsilon(d_{\mathbb{T}}(x,y))))\rmd t
+ 2\mathrm{rc}^\delta(\bar{r}_t^\delta)\rmd W_t-2\rmd \ell_t^\pi \eqsp,
\end{align*}
where we used the properties of $\mathrm{rc}^\delta$ and 
\begin{align*}
( x-y)\cdot(\sin(x-\tilde{x})-\sin(y-\tilde{x}) )\leq 2\sin(|x-y|/2)|x-y| 
\end{align*} 
for any $x,y,\tilde{x}\in \torus$.
Consider $(r_t^{\delta,\epsilon})_{t\geq 0}$ given by
\begin{align*}
\rmd r_t^{\delta,\epsilon}&= (2k\sin(r_t^{\delta,\epsilon}/2)+2k\int_0^\pi \mathrm{rc}^\epsilon(u)\rmd P_t^{\delta,\epsilon}(u))\rmd t
 +2\mathrm{rc}^\delta({r}_t^{\delta,\epsilon})\rmd W_t-2\rmd \ell_t^\pi\eqsp,
\end{align*}
where $P_t^{\delta,\epsilon}$ is the law of $r_t^{\delta,\epsilon}$.
Then as in \Cref{lem:comparision_r_t^delta}, for the processes $(\bar{r}_t^\delta)_{t\geq 0}$ and $(r_t^{\delta,\epsilon})_{t\geq 0}$ with the same initial condition and driven by the same noise it holds $\bar{r}_t^\delta\leq r_t^{\delta,\epsilon}$ almost surely for every $t$ and $\epsilon<\epsilon_0$.

Consider the process $(U_t^{\delta,\epsilon})_{t\geq 0}=(X_t^\delta,Y_t^\delta,r_t^{\delta,\epsilon})_{t\geq 0}$ on $\mathbb{T}^2\times[0,\pi]$ for each $\epsilon,\delta>0$.  
We define by $\mathbf{X},\mathbf{Y}:\mathcal{C}(\mathbb{R}_+,\mathbb{T}^2\times[0,\pi])\to \mathcal{C}(\mathbb{R}_+,\mathbb{T})$ and $\mathbf{r}:\mathcal{C}(\mathbb{R}_+,\mathbb{T}^2\times[0,\pi])\to \mathcal{C}(\mathbb{R}_+,[0,\pi])$ the canonical projections onto the first component, onto the second component and onto the last component, respectively. 
Analogously to the proof of \Cref{theo:1}, the law $\PP^{\delta,\epsilon}$ of the process $(U_t^{\delta,\epsilon})_{t\geq 0}$ converges along a subsequence $(\delta_k,\epsilon_k)_{k\in\mathbb{N}}$ to a probability measure $\PP$. Let $(X_t,Y_t,r_t)_{t\geq 0}$ be some process on $\mathbb{T}^2\times[0,\pi]$ with distribution $\PP$ on $(\bar{\Omega},\bar{\mathcal{F}},\bar{P})$. Since $(X_t^\delta)_{t\geq 0}$ and $(Y_t^\delta)_{t\geq 0}$  are solutions of \eqref{eq:SDE_kuramoto} which are unique in law, we have that for any $\epsilon, \delta>0$, $\PP^{\delta,\epsilon}\circ \mathbf{X}^{-1}=\PP\circ\mathbf{X}^{-1}$ and $\PP^{\delta,\epsilon}\circ \mathbf{Y}^{-1}=\PP\circ\mathbf{Y}^{-1}$.
And therefore $(X_t)_{t\geq 0}$ and $(Y_t)_{t\geq 0}$ are solutions of \eqref{eq:SDE_kuramoto} as well with the same initial condition.
Hence $\PP\circ(\mathbf{X},\mathbf{Y})^{-1}$ is a coupling of two copies of \eqref{eq:SDE_kuramoto}.

 Further, the monotonicity $\bar{r}_t^{\delta}\leq r_t^{\delta,\epsilon}$ carries over to the limit by the Portmanteau theorem.
Finally, similarly to the proof of \Cref{lem:existence_stickySDE_step1} and \Cref{lem:existence_stickySDE_step2} 
there exist an extended probability space and a one-dimensional Brownian motion 
$(W_t)_{t\geq 0}$ such that $(r_t,W_t)_{t\geq 0}$ is a solution to \eqref{eq:kuramoto_stickySDE}. 

\end{proof}

\begin{proof} [Proof of \Cref{thm:kuramoto}]

Similarly to \eqref{eq:definition_f} we consider a function ${\tilde{f}}$ on $[0,\pi]$ defined by 
\begin{align} \label{eq:f_Kuramoto}
\tilde{f}(t)=\int_0^t\tilde{\varphi}(r)\tilde{g}(r)\rmd r\eqsp,
\end{align}
where
\begin{align*}
&\tilde{\varphi}(r)=\exp\{2k(\cos(r/2)-1)\}\eqsp, \qquad \tilde{\Phi}(r)=\int_0^r\tilde{\varphi}(s)\rmd s\eqsp,
\\ & \tilde{g}(r)=1-\frac{c_\mathbb{T}}{2}\int_0^r\{\tilde{\Phi}(s)/\tilde{\varphi}(s)\}\rmd s- k\int_0^r\{1/\tilde{\varphi}(s)\}\rmd s\eqsp,
\\ & c_\mathbb{T}=\parenthese{2 \int_0^{\pi}\{\tilde{\Phi}(s)/\tilde{\varphi}(s)\}\rmd s}^{-1}\eqsp.
\end{align*} Then for $k$ satisfying \eqref{eq:k_kuramoto}, $\tilde{g}(r)\in[1/2,1]$ and $\tilde{f}$ is a concave function satisfying similarly to \eqref{eq:norm_equivalence} 
\begin{align} \label{eq:norm_equivalence_kuramoto}
\exp(-2k)/2 r\leq \tilde{f}\leq \tilde{\Phi}(r)\leq r
\end{align}
and
\begin{equation} \label{eq:f_kuramoto_prop} 
\begin{aligned}
& \tilde{f}''(0)=-k
\\ & 2(\tilde{f}''(r)- \tilde{f}''(0))\leq -2k\sin(r/2)\tilde{f}'(r)-c_\mathbb{T}\tilde{f}(r) \qquad \text{for all }r\in[0,\pi]\eqsp.
\end{aligned}
\end{equation}
By It\=o's formula it holds 
\begin{align*}
\rmd \tilde{f}(r_t)&=\tilde{f}'(r_t)(2k\sin(r/2)+2k\PP(r_t>0))\rmd t+2\tilde{f}'(r_t)\1_{(0,\pi]}(r_t) \rmd W_t-2\tilde{f}'(r_t)\rmd \ell_t^\pi
\\& +2\tilde{f}''(r_t)\1_{(0,\pi]}(r_t)\rmd t\eqsp.
\end{align*}
Taking expectation and using that the term involving the local time is negative, we obtain
\begin{align*}
\frac{\rmd}{\rmd t}\mathbb{E}[\tilde{f}(r_t)]&\leq \mathbb{E}[2(\tilde{f}''(r_t)-f''(0))+\tilde{f}'(r_t)2k\sin(r_t/2)]+(2\tilde{f}''(0)+2k)\PP(r_t>0)
\\ & \leq -c_\mathbb{T}\mathbb{E}[\tilde{f}(r_t)]\eqsp,
\end{align*}
where the last step holds by \eqref{eq:f_kuramoto_prop}.
Then 
\begin{align} \label{eq:contraction_kuramoto}
\mathbb{E}[\tilde{f}(d_{\mathbb{T}}(\bar{X}_t,\bar{Y}_t))]\leq \mathbb{E}[\tilde{f}(r_t)]\leq \rme^{-c_\mathbb{T}t} \mathbb{E}[\tilde{f}(r_0)]=\rme^{-c_\mathbb{T}t}\mathbb{E}[\tilde{f}(d_{\mathbb{T}}(\bar{X}_0,\bar{Y}_0))]\eqsp,
\end{align}
provided \eqref{eq:k_kuramoto} holds. Thus 
\begin{align*}
\mathcal{W}_{\tilde{f}}(\mu_t, \nu_t)\leq \rme^{-c_\mathbb{T} t} \mathcal{W}_{\tilde{f}}(\mu_0,\nu_0)\eqsp,
\end{align*}
and 
by \eqref{eq:norm_equivalence_kuramoto}
\begin{align*}
\mathcal{W}_{1}(\mu_t, \nu_t)\leq 2\exp(2k)\rme^{-c_\mathbb{T} t} \mathcal{W}_{1}(\mu_0,\nu_0)\eqsp.
\end{align*}
\end{proof}

\begin{remark} \label{rem:contr_kuramoto}
Let us finally remark that we can relax the condition \eqref{eq:k_kuramoto} and we can obtain contraction with a modified contraction rate $c_\mathbb{T}$ for all $k<k_0$, where $k_0$ is given by
\begin{align} \label{eq:k0_kuramoto}
k_0\int_0^{\pi}\exp(2k_0-2k_0\cos(r/2)) \rmd r=1\eqsp.
\end{align} 
More precisely, set $\zeta =1-k\int_0^{\pi}\exp(2k-2k\cos(r/2)) \rmd r$ and $c_\mathbb{T}=\zeta\parenthese{ \int_0^{\pi}\{\tilde{\Phi}(s)/\tilde{\varphi}(s)\}\rmd s}^{-1}$. Then, $\tilde{g}(r)\in[\zeta/2,1]$ and $\zeta\exp(-2k)/2 r\leq \tilde{f}(r)\leq r$. Following the previous computations, we obtain
\begin{align*}
\mathcal{W}_{1}(\mu_t, \nu_t)\leq 2\exp(2k)/\zeta\rme^{-c_\mathbb{T} t} \mathcal{W}_{1}(\mu_0,\nu_0)\eqsp,
\end{align*}
where for $k$ close to $k_0$, the contraction rate becomes small and the prefactor $2\exp(2k)/\zeta$ explodes.
\end{remark}

\subsection{Sticky nonlinear SDEs on bounded state space}\label{appendix2}

In the same line as in \Cref{thm:existence_comparison}, existence, uniqueness in law and comparison results hold for solutions to the sticky SDE
on $[0,\pi]$ given by
\begin{align} \label{eq:kuramoto_stickySDE}
\rmd r_t=(\tilde{b}(r_t)+2k\PP(r_t >0))\rmd t+2\1_{(0,\pi)}(r_t)\rmd W_t-2\rmd \ell^{\pi}_t,
\end{align}
where $k\in\mathbb{R}_+$ and $\ell^\pi$ is the local time at $\pi$.

The analysis of invariant measures and phase transitions can be easily adapted to the case of the sticky SDE on $[0,\pi]$ given by \eqref{eq:kuramoto_stickySDE}.
\begin{theorem} \label{thm:kuramoto_invmeas_generaldrift}
Let $(r_t)_{t\geq 0}$ be a solution of \eqref{eq:kuramoto_stickySDE} with drift $\tilde{b}$ satisfying \Cref{H1b}. Then, the Dirac measure at zero, $\delta_0$, is an invariant probability measure on $[0,\pi]$ for \eqref{eq:kuramoto_stickySDE}. If there exists $p\in(0,1)$ solving $(1/k)=(1-p)I(k,p)$
where
\begin{align*}
I(k,p)=\int_0^\pi \exp\Big(kpx+\frac{1}{2}\int_0^x \tilde{b}(r) \rmd r\Big)\rmd x\eqsp,
\end{align*}
then the probability measure $\pi$ on $[0,\pi]$ given by 
\begin{align} \label{eq:kuramoto_nontrivialinvmeas}
\pi(\rmd x)\propto \frac{1}{kp}\delta_0(\rmd x)+\exp\Big(kpx+\frac{1}{2}\int_0^x \tilde{b}(r) \rmd r\Big)\lambda_{(0,\pi)}(\rmd x)
\end{align}
is another invariant probability measure for \eqref{eq:kuramoto_stickySDE}.

\end{theorem}

\begin{proof}[Proof of \Cref{thm:kuramoto_invmeas_generaldrift}]
The proof works analogously to the proof of \Cref{thm:one_dim_SDE_statdistr_general} for sticky SDEs on $\mathbb{R}_+$. Note that here the condition \eqref{eq:necessarycondition_stationarydistr} transforms for $p\in (0,1]$ to 
\begin{align*}
p=\pi((0,\pi))=\frac{I(k,p)}{1/(kp)+I(k,p)} \Leftrightarrow (1-p)I(k,p)=1/k\eqsp.
\end{align*}
\end{proof}

\begin{example} \label{example:kuramoto_inv}
Consider a solution $(r_t)_{t\geq 0}$ of \eqref{eq:kuramoto_stickySDE} with drift $\tilde{b}(r)=2k\sin(r/2)$. Consider a solution $p\in(0,1]$ solving $1/k=(1-p)I(k,p)$
with
\begin{align*}
I(k,p)=\int_0^\pi \exp\Big(kpx+\int_0^x k\sin(r/2) \rmd r\Big)\rmd x=\int_0^\pi \exp\Big(kpx+2k-2k\cos(x/2)\Big) \rmd x\eqsp.
\end{align*}
Then by \Cref{thm:kuramoto_invmeas_generaldrift}, the Dirac measure at zero, $\delta_0$ and the probability measure
\begin{align} \label{eq:inv_meas_Kuramoto}
\pi(\rmd x)\propto \frac{1}{kp}\delta_0(\rmd x)+\exp(kpx+2k-2k\cos(x/2))\lambda_{(0,\pi)}(\rmd x)
\end{align}
are invariant probability measures for \eqref{eq:kuramoto_stickySDE}.
We specify a necessary and sufficient condition for the existence of a solution $p$ satisfying $1/k=(1-p)I(k,p)$.
We define $\hat{I}(k,p)=(1-p)I(k,p)$. 
We first consider the case $1/k<\hat{I}(k,0)=\int_0^\pi\exp(2k-2k\cos(x/2))\rmd x $.
Then since $1/k> \hat{I}(k,1)=0$ and by the mean value theorem there exists a $p$ solving $1/k=\hat{I}(k,p)$ and therefore there exist multiple invariant distributions for \eqref{eq:inv_meas_Kuramoto}. On the other hand, if $1/k>\hat{I}(k,0)=\int_0^\pi\exp(2k-2k\cos(x/2))\rmd x $, since 
$\pi\leq \int_0^{\pi}\exp(2k-2k\cos(x/2))\rmd x$ and for $k<1/\pi$, it holds
\begin{align*}
\frac{\rmd }{\rmd p}\hat{I}(k,p)&=-I(k,p)+(1-p) \int_0^\pi kx\exp(kpx+2k-2k\cos(x/2)) \rmd x
\\ & = \int_0^\pi ((1-p)kx-1)\exp(kpx+2k-2k\cos(x/2)) \rmd x\leq 0\eqsp,
\end{align*}
there is no $p$ satisfying \eqref{eq:inv_meas_Kuramoto}.

\end{example}

\begin{remark}\label{rem:kuramoto_conv}
The contraction result given in \Cref{thm:one-dim_stickydiff} carries over to the sticky diffusion $(r_t)$ given by \eqref{eq:kuramoto_stickySDE} on $[0,\pi]$ with $\tilde{b}(r)=2k\sin(r/2)$. If \eqref{eq:k_kuramoto} holds, 
then for $t\geq 0$, \eqref{eq:stickydiff_estimateinthm} holds with $\tilde{f}$ defined in \eqref{eq:f_Kuramoto} and $c_{\mathbb{T}}$ defined in \eqref{eq:c_T} using \eqref{eq:contraction_kuramoto}.
Moreover by \Cref {rem:contr_kuramoto}, we can deduce that if \eqref{eq:k0_kuramoto} holds, the Dirac measure at zero, $\delta_0$, is the unique invariant measure and contraction towards $\delta_0$ holds.
\end{remark}

 \subsection*{Funding}
A.~E. and K.~S. have been supported by the \textit{Hausdorff Center for Mathematics}. 
Gef\"ordert durch die Deutsche Forschungsgemeinschaft (DFG) im Rahmen der Exzellenzstrategie des Bundes und der L\"ander - GZ 2047/1, Projekt-ID 390685813. The work of A.G. has been (partially) supported by the Project EFI ANR-17-CE40-0030 of the French National Research Agency
A.D. acknowledges support  of the Lagrange Mathematical and Computing Research Center. 
\nocite{*}

\bibliographystyle{plainurl}
\bibliography{stickycouplings_bib} 

%
%
%
%

\end{document}